\documentclass[twoside,11 pt,reqno]{amsart} 
\usepackage{graphicx}

\usepackage[margin=2cm]{geometry} 
\title{Semi-classical Heat Kernel Asymptotics and Morse Inequalities} 
\author{Eric Jian-Ting Chen}
\address{Institute of Mathematics, Academia Sinica, 6F, Astronomy-Mathematics Building, 
No. 1, Sec. 4, Roosevelt Road, Taipei 10617, Taiwan}
\email{as0210903@gate.sinica.edu.tw or ejtchen1997@gmail.com}

\usepackage{amsmath,amsthm,amssymb,amsfonts}

\usepackage{appendix}

\newtheorem{theo}{\bf Theorem} [section]
\newtheorem{lem}[theo]{\bf Lemma} 
\newtheorem{coro}[theo]{\bf Corollary} 
\newtheorem{prop} [theo] {\bf Proposition} 
\theoremstyle{remark} 

\DeclareMathOperator\supp{supp}
\DeclareMathOperator\crit{Crit}
\DeclareMathOperator\ind{Ind}
\DeclareMathOperator\Ker{Ker}
\DeclareMathOperator\im{Im}

\DeclareMathOperator\spe{Spec}
\DeclareMathOperator\tr{tr}
\DeclareMathOperator\hess{Hess}
\DeclareMathOperator\dom{Dom}

\numberwithin{equation}{section}

\usepackage{hyperref}

\usepackage{pgf,tikz}
\usepackage{mathrsfs}
\usetikzlibrary{arrows}

\usepackage{tikz-cd}

\begin{document}

\begin{abstract}
In this paper, we study the asymptotic behavior of the heat kernel with respect to the Witten Laplacian. We introduce the localization and the scaling technique in semi-classical analysis, and study the semi-classical asymptotic behavior of the family of the heat kernel, indexed by $k$, near the critical point $p$ of a given Morse function, as $k\to \infty$. It is shown that this family is approximately close to the heat kernel with respect to a system of the harmonic oscillators attached to $p$. We also furnish some asymptotic results regarding heat kernels away from the critical points. These heat kernel asymptotic results lead to a novel proof of the Morse inequalities.
\end{abstract}

\maketitle

\tableofcontents

\allowdisplaybreaks 

\section{Introduction}\label{Intro}

In his marvellous paper \cite{W82}, Witten gave a new proof of the Morse inequalities by considering the  family of the so called Witten Laplacians $\Delta_{k}=\Delta+ k^{2}\left|df\right|^{2}+kA$, where $k>0$ is a parameter, $f$ is a Morse function, and $A$ is an operator of order $0$. He proved that the spectral functions of $\Delta_{k}$ is approximately close to the spectral functions of a system of harmonic oscillators attached to the critical point of $f$, as $k\to \infty$. His idea of studying deformed operators indexed by $k$ has led to several breakthrough in several fields.

In complex geometry, Demailly \cite{D85} discovered the holomorphic Morse inequalities that describe the $k$-large asymptotic upper bounds for the Morse-Witten sums of the Betti numbers of $\overline{\partial}$ on $L^{\otimes k}$ in terms of the Chern curvature of the Hermitian holomorphic line bundle $L$. The key of finding such inequalities is that he managed to localize the problem by converting the local frame of $L$ to the local weight function that plays a role of a Morse function in \cite{W82}. This in turn, led to the consideration of the family of operators $\Box_{k}$ analogous to $\Delta_{k}$. Meanwhile, Bismut gave the heat equation proofs of Morse inequalities in \cite{B86} and of Demailly's holomorphic Morse inequalities in \cite{B87}, using the probability theory. Subsequently, Demailly \cite{D91} replaced Bismut's probabilistic argument and recovered his holomorphic Morse inequalities by investigating the heat kernel asymptotics. In view of the recent progress, we feel it important to investigate the heat kernel asymptotics in Witten's classical setting.

In this present paper, we study the asymptotic behavior of the heat kernel with respect to the Witten Laplacian $\Delta_{k}$ and in turn recover the Morse inequalities. 

Let us briefly illustrate how we obtain our semi-classical heat kernel asymptotics (see Theorem \ref{Main Thm: Semi-classical Heat Kernel Asymptotics}). Our asymptotic behaviors of the heat kernel are achieved based on the two techniques: localization and scaling technique. We seek to localized the heat kernel near the critical point $p$ of the Morse function by constructing a metric that is flat around $p$ together with the Morse lemma. The localization of this kind was indeed motivated by Witten's work \cite{W82}. To obtain the asymptotics, we introduced the scaling technique in semi-classical analysis. This technique allows us to consider the family of the so-called scaled heat kernels near the critical point $p$, indexed by $k$, and study the asymptotic behavior of this family as $k\to \infty$. It turns out that this family is approximately close to the heat kernel with respect to a system of the harmonic oscillators attached to $p$ that resembles Witten's finding. We would like to stress that these two techniques effectively tackle the asymptotic behavior of heat kernels in a computable way. 

These techniques have been applied to several projects. In CR geometry, for example, Hsiao and Zhu \cite{HZ23} investigated the semi-classical asymptotics of the heat kernel with respect to Kohn Laplacian using these tricks and obtained the CR and $\mathbb{R}$-equivariant Morse inequalities. On the other hand, in complex geometry, Chiang \cite{C23} made use of these techniques and obtained the semi-classical asymptotic behaviors of Bergman kernels and spectral kernels. We refer the reader to \cite{B04}, \cite{MM07}, \cite{HM12}, \cite{HM14} for further related scaling techniques.

To recover the Morse inequalities from the heat kernel asymptotics, we use the classical trace integral formula (see \cite{MS67}) in the local index theory. It turns out that, somewhat surprisingly, we need more delicate asymptotic results regarding the heat kernels away from the critical points (see Theorems \ref{Main Thm: Asymptotic Behavior of Scaled Heat Kernel in between} and \ref{Main Thm: Asymptotic Behavior of Heat Kernel outside Critical Points}). Perhaps one of the causes of the difference from Hsiao and Zhu's work \cite{HZ23} would be that they investigate their semi-classical heat kernel asymptotics for each point in the underlying manifold, while we only consider the asymptotics of this kind near the critical point. These delicate results are obtained based on certain Bochner-type estimates.

\subsection{Statements of the Main Results}

In this subsection, we state our results in detail. We refer the reader to Section \ref{Prelim} and Subsections \ref{local.} and \ref{s.t.} for the terminologies we use. 

Let $M$ be a compact smooth manifold of dimension $n$ and let $f$ be a Morse function on $M$. We equip $M$ with a metric $g=\langle \cdot|\cdot\rangle$ such that for every critical point $p$ of $f$, we can choose a coordinate chart such that
$\langle \frac{\partial}{\partial x^{i}},\frac{\partial}{\partial x^{j}}\rangle=\delta^{i}_{j}$ and $f$ is written as a quadratic form (according to Morse Lemma) near $p$, where $\delta^{i}_{j}$ is the Kronecker delta (see Theorem \ref{Thm: Locally Flat Metric}).  For each $k>0$, denote the Witten Laplacian (acting on $r$-forms) by
\begin{equation*}
\begin{aligned}
\Delta_{k}^{\left(r\right)} := d_{k}d_{k}^{*}+d_{k}^{*}d_{k},
\end{aligned}
\end{equation*}
where $d_{k}:=e^{-kf}de^{kf}$ is the deformed exterior operator and $d_{k}^{*}$ is its formal adjoint with respect to the induced $L^{2}$-inner product from the metric $g$.  

For each critical point $p\in \crit\left(f\right)$, denote the scaled heat kernel by
\begin{equation*}
\begin{aligned}
A_{\left(k\right),p}^{r}\left(t,x,y\right) :=k^{-\frac{n}{2}} e^{-\frac{t}{k}\Delta_{k}^{\left(r\right)}}\left(\frac{x}{\sqrt{k}},\frac{y}{\sqrt{k}}\right),
\end{aligned}
\end{equation*}
where $e^{-t\Delta_{k}^{\left(r\right)}}\left(x,y\right)$ denotes the heat kernel with respect to $\Delta_{k}^{\left(r\right)}$ on $r$-forms.  Moreover, denote  by $e^{-t\Delta_{f,p}^{\left(r\right)}}\left(x,y\right)$ the heat kernel of the system of harmonic oscillators $\Delta_{f,p}^{\left(r\right)}$ attached to $p$ (see \eqref{Eq: Model Laplacian}). Then our semi-classical heat kernel asymptotics is stated as follows:

\begin{theo} \label{Main Thm: Semi-classical Heat Kernel Asymptotics} For each critical point $p\in \crit\left(f\right)$,
\begin{equation*}
\begin{aligned}
\lim_{k\to \infty} A_{\left(k\right),p}^{r}\left(t,x,y\right) = e^{-t\Delta_{f,p}^{\left(r\right)}}\left(x,y\right).
\end{aligned}
\end{equation*}
in $C^{\infty}$-topology in each compact subset of $\mathbb{R}^{+}\times \mathbb{R}^{n}\times \mathbb{R}^{n}$. Consequently, we obtain the following pointwise asymptotic  
\begin{equation*}
\begin{aligned}
\lim_{k\to \infty} k^{-\frac{n}{2}}e^{-\frac{t}{k}\Delta_{k}^{\left(r\right)}}\left(p,p\right) = e^{-t\Delta_{f,p}^{\left(r\right)}}\left(0,0\right).
\end{aligned}
\end{equation*}
\end{theo}

Theorem \ref{Main Thm: Semi-classical Heat Kernel Asymptotics} shows that the leading term of on-diagonal heat kernel expansion of $e^{-\frac{t}{k}\Delta_{k}^{\left(r\right)}}\left(x,y\right)$ at the critical point $p$ as $k$ large is given by $e^{-t\Delta_{f,p}^{\left(r\right)}}\left(0,0\right)$. We can further unwind $e^{-t\Delta_{f,p}^{\left(r\right)}}\left(x,x\right)$ by Mehler's formula and then capture the information of the critical point $p$ (see Theorem \ref{Thm: Trace integral for perturbed harmonic oscillator}).

Now, let us state the asymptotic results of heat kernels away from the critical points. Denote by $B_{R}\left(q\right)\subset \mathbb{R}^{n}$ an Euclidean ball centered at $q$ with radius $R$. For each $r$-form $\omega$ (on $\mathbb{R}^{n}$), let $\left|\omega\right|$ be the norm of $\omega$ induced from the (flat) metric $g$ (on $\mathbb{R}^{n}$). Define the norm $\left|\cdot\right|_{x}$  for the space of the linear transformations $\bigwedge^{r}T^{*}_{x}\mathbb{R}^{n}\otimes \left(\bigwedge^{r}T^{*}_{x}\mathbb{R}^{n}\right)^{*}$ by 
\begin{equation*}
\begin{aligned}
\left| S \right|_{x} := \sup_{\omega_{x}\in \bigwedge^{r}T^{*}_{x}\mathbb{R}^{n}, \omega_{x}\neq 0} \frac{ \left| S\omega_{x} \right| }{ \left| \omega_{x} \right|} 
\end{aligned}
\end{equation*}
for each $S \in \bigwedge^{r}T^{*}_{x}\mathbb{R}^{n}\otimes \left(\bigwedge^{r}T^{*}_{x}\mathbb{R}^{n}\right)^{*}$.

\begin{theo}\label{Main Thm: Asymptotic Behavior of Scaled Heat Kernel in between}
There exists $D>0$ such that if $k$ large enough, then for each $p\in \crit\left(f\right)$ and for each $N\in \mathbb{N}$, 
\begin{equation*}
\begin{aligned}
\left| A_{\left(k\right),p}^{r}\left(t,x,x\right) \right|_{x}  \leq  C\left(t,N\right)\left|x\right|^{-N},
\end{aligned}
\end{equation*}
where $C\left(t,N\right)$ depends on $N$ and smoothly on $t$ and is independent of $D,x,k$, for each $x\in B_{k^{\varepsilon}}\left(0\right) \setminus B_{2D}\left(0\right)$. 
\end{theo}

For each large $k$, set $\mathcal{U}^{k}= \bigcup_{p\in \crit\left(f\right)} U_{p}^{k}$, where $U_{p}^{k}$ is identified with $B_{k^{-\frac{1}{2}+\varepsilon}} \left(0\right)$, $\varepsilon\in \left(0,\frac{1}{2}\right)$, under the coordinate chart of $p\in \crit\left(f\right)$. 

\begin{theo} \label{Main Thm: Asymptotic Behavior of Heat Kernel outside Critical Points}
If $k$ is sufficiently large, then for each $t>0$ and for each $N\in \mathbb{N}$, 
\begin{equation*}
\begin{aligned}
\left\| e^{-\frac{t}{k}\Delta_{k}^{\left(r\right)}}\left(x,x\right)\right\|_{\mathcal{C}^{0}\left(M\setminus \mathcal{U}^{k} \right)} \leq C\left(t,N\right)k^{-N},
\end{aligned}
\end{equation*}
where $C\left(t,N\right)$ depends on $N$ and smoothly on $t$ and is independent of $k$ and the $\mathcal{C}^{0}$-norm is determined by a choice of partition of unity and orthonormal frame. This implies
\begin{equation*}
\begin{aligned}
\lim_{k\to\infty} e^{-\frac{t}{k}\Delta_{k}^{\left(r\right)}}\left(x,x\right) = 0
\end{aligned}
\end{equation*}
for each $x \in M\setminus \crit\left(f\right)$.
\end{theo}

From Theorems \ref{Main Thm: Semi-classical Heat Kernel Asymptotics}, \ref{Main Thm: Asymptotic Behavior of Scaled Heat Kernel in between}, and \ref{Main Thm: Asymptotic Behavior of Heat Kernel outside Critical Points}, we recover the Morse inequalities (See section \ref{M.I.}). In fact, from these three theorems, we can deduce
\begin{equation*}
\begin{aligned}
\lim_{t\to \infty} \lim_{k\to \infty} \int_{M} \tr e^{-\frac{t}{k}\Delta_{k}^{\left(r\right)}}\left(x,x\right) dV = m_{r},
\end{aligned}
\end{equation*}
where $m_{r}$ is the number of the critical points of $f$ with index $r$ (see Theorem \ref{Main Thm: Critical point info captured by the asymptotic behaviors}), which together with the local index theory, establishes the Morse inequalities. We refer the reader to \cite{M63} for a topological proof of the Morse inequalities.  

This paper is organized as follows: In Section \ref{Prelim}, we set up some notations and review some essential notions and theorems. In Section \ref{A.B.S.H.K.}, we introduce the localization and scaling technique, and prove Theorem \ref{Main Thm: Semi-classical Heat Kernel Asymptotics}. Later on, in Section \ref{H.K.A.O.C.P.}, we introduce the Bochner-type estimate and prove Theorems \ref{Main Thm: Asymptotic Behavior of Scaled Heat Kernel in between} and \ref{Main Thm: Asymptotic Behavior of Heat Kernel outside Critical Points}. Finally, in Section \ref{M.I.}, we investigate the model kernel and present the new heat kernel proof of the Morse inequalities.

\section{Preliminaries}\label{Prelim}

\subsection{Notations and Terminologies}\label{nota}

Let $\alpha=\left(\alpha_{1},\cdots,\alpha_{n}\right)$ be a multi-index and we set $\left| \alpha\right|=\sum_{i=1}^{n}\alpha_{i}$ and put
\begin{equation*}
\begin{aligned}
\partial^{\alpha}_{x} = \frac{\partial^{\left|\alpha\right|}}{\left( \partial x^{1}\right)^{\alpha_{1}} \cdots \left(\partial x^{n}\right)^{\alpha_{n}} },
\end{aligned}
\end{equation*}
where $x=\left(x^{1},\cdots, x^{n}\right)\in \mathbb{R}^{n}$.

Let $M$ be a smooth manifold of dimension $n$. We denote by $TM$ the tangent bundle of $M$ and by $T^{*}M$ the cotangent bundle of $M$. We say a (differential) $r$-form $\omega$ to be a section (not necessarily smooth) of the exterior bundle $\bigwedge^{r} \left(T^{*}M\right)$. Choosing a local coordinate chart with the coordinates $x^{1},\cdots,x^{n}$, we can locally write $\omega$ as $\omega=\sideset{}{'}\sum_{I} \omega_{I} dx^{ I }$, where the primed sum $\sideset{}{'}\sum_{I}$ refers to the one that runs over all multi-indices $I=\left(i_{1},\cdots, i_{r}\right)$ with $\left|I\right|=r$ arranged in increasing order (namely, $1\leq i_{1}<\cdots< i_{r} \leq n$), and where $ \omega_{ I } $ are component functions of $\omega$ and $dx^{I}=dx^{i_1} \wedge \cdots \wedge dx^{i_r}$. In particular, let $U\subset M$ be an open subset, and we denote by $\Omega^{r}\left(U\right) = \mathcal{C}^{\infty}\left(U,\bigwedge^{r}\left(T^{*}M\right)\right)$ the space of smooth $r$-forms and by $\Omega_{c}^{r}\left(U\right)$ the space of all smooth $r$-forms compactly supported in $U$. We denote by $H_{dR}^{r}\left(M\right)$ the $r$-th de Rham cohomology group on $M$.

Let $\left(M,g\right)$ be an orientable Riemannian manifold endowed with the metric $g\left(\cdot,\cdot\right)=\langle \cdot|\cdot\rangle$ and let $dV$ be the volume form induced by $g$. We denote by $\left|\cdot\right|$ the norm associated to the metric $g$.
Note that $g$ induces the metric on $\bigwedge^{r}\left(T^{*}M\right)$ (still denoted as $g$ and its associated norm is also denoted as $\left|\cdot\right|$), and thus the $L^{2}$-inner product $\left(\cdot|\cdot\right)$ on $\Omega^{r}_{c}\left(M\right)$ with respect to $dV$. The completion of $\Omega^{r}_{c}\left(M\right)$ with respect to $\left(\cdot|\cdot\right)$ is denoted by $L^{2}_{r}\left(M\right)$. We denote the associated $L^{2}$-norm on $M$ by $\left\| \cdot\right\|_{L^{2}\left(M\right)}$, and we will omit the subscript $L^{2}\left(M\right)$ if there is no ambiguity. Let $d$ be the exterior derivative and denote by $d^{*}$ the formal adjoint of $d$ with respect to 
$\left(\cdot|\cdot\right)$.

Let $f$ be a Morse function. The set consisting of all critical points of $f$ is denoted by $\crit \left(f\right)$. The index of $f$ at $p\in \crit \left(f\right)$ is denoted by $\ind_{f} p$. Besides, the $j$-th Morse number $m_{j}$ is defined to be the number of the set $\left\{ p\in \crit\left(f\right):\ind_{f} p=j\right\}$.

We give a final remark on the appearance of constants. Throughout this paper, $C\left(\cdot\right)$ denotes a constant that depends on what appears within the parenthesis.

\subsection{Witten Laplacians and Heat Kernels}\label{w.l.}

Let $M$ be a compact Riemannian manifold of dimension $n$ and let $f$ be a Morse function on $M$. For each $k> 0$, define the deformed exterior derivative $d_{k}: \Omega^{r}\left(M\right)\to\Omega^{r+1}\left(M\right)$ by
\begin{equation*}
\begin{aligned}
	d_{k} := e^{-kf} d e^{kf} = d+kdf\wedge.
\end{aligned}
\end{equation*}

It is evident to see that $d_{k}^{2}=0$. Define the deformed $r$-th de Rham cohomology group on $M$ by
\begin{equation*}
\begin{aligned}
H_{k}^{r}\left(M\right)
=\frac{ \Ker\left(d_{k} : \Omega^{r}\left(M\right)\to \Omega^{r+1}\left(M\right)\right)}{\im\left(d_{k} : \Omega^{r-1}\left(M\right)\to \Omega^r\left(M\right)\right)}.
\end{aligned}
\end{equation*}

\begin{prop}[\cite{M}, Theorem 2.3]\label{Prop: Isomorphism between two coho. gps}
For each $k> 0$, $H_{k}^{r}\left(M\right)$ is isomorphic to $H_{dR}^{r}\left(M\right)$.
\end{prop}

Let $d_{k}^{*}:\Omega^{r}\left(M\right)\to \Omega^{r-1}\left(M\right)$ be the formal adjoint of $d_{k}$ with respect to the $L^{2}$-inner product $\left(\cdot|\cdot\right)$. Then
\begin{equation*}
\begin{aligned}
d_{k}^{*}=e^{kf}d^{*}e^{-kf}= d^{*}+ k\iota_{\nabla f},
\end{aligned}
\end{equation*}
where $d^{*}=\left(-1\right)^{n\left(r+1\right)+1} \ast d\ast$, $\ast$ is the Hodge star operator, and $\nabla f$ is the gradient of $f$.

For each $k>0$, the Witten Laplacian on $\Omega^{r}\left(M\right)$ is defined to be
\begin{equation*}
\begin{aligned}
\Delta_{k}^{\left(r\right)} := d_{k}^{*}d_{k}+d_{k}d_{k}^{*}: \Omega^{r}\left(M\right)\to \Omega^{r}\left(M\right).
\end{aligned}
\end{equation*}
By direct computation, we obtain
\begin{equation}\label{Eq: Global expression of Witten Laplacian}
\begin{aligned}
\Delta_{k}^{\left(r\right)} 
=\Delta^{\left(r\right)}  + k^{2}\left|df\right|^{2}
+ k\left( \mathcal{L}_{\nabla f} + \mathcal{L}_{\nabla f}^{*}\right),
\end{aligned}
\end{equation}
where $\mathcal{L}_{\nabla f}$ denotes the Lie derivative of $\nabla f$ and $\mathcal{L}_{\nabla f}^{*}$ the formal adjoint of $\mathcal{L}_{\nabla f}$ with respect to the $L^{2}$-inner product $\left(\cdot|\cdot\right)$.
Note that $\Delta_{k}^{\left(r\right)}$ is elliptic. Moreover, $\Delta_{k}^{\left(r\right)}$ has the  local expression
\begin{equation}\label{Eq: Local expression of Witten Laplacian}
\begin{aligned}
\Delta_{k}^{\left(r\right)} = \Delta^{\left(r\right)} + k^{2}\left|df\right|^{2}+k\sum_{i,j} \hess_{f} \left(\frac{\partial}{\partial x^{i}},\frac{\partial}{\partial x^{j}}\right) \left[ dx^{j}\wedge,\iota_{\left(dx^{i}\right)^{\sharp}}\right],
\end{aligned}
\end{equation}
where $\hess_{f}=\nabla^{TM}df$ is the Hessian form with respect to the Levi-Civita connection $\nabla^{TM}$ and $\left(dx^{i}\right)^{\sharp}$ is the dual element to $dx^{i}$ with respect to the inner product $\langle\cdot|\cdot\rangle$.

We extend the Witten Laplaican to $\Delta_{k}^{\left(r\right)}:\dom \Delta_{k}^{\left(r\right)}\to L^{2}_{r}\left(M\right),$ where
\begin{equation*}
\begin{aligned}
	\dom \Delta_{k}^{\left(r\right)}=\left\{ \omega \in L^{ 2 }_{ r }\left( M \right) : \Delta_{k}^{\left( r\right)} \omega \in L^{2}_{r} \left( M\right) \right\}.
\end{aligned}
\end{equation*}
Note that $ \Omega^{r}\left(M\right)\subset \dom \Delta_{k}^{\left(r\right)}$ is dense in $L^{2}_{r}\left(M\right)$. Also, denote the adjoint of $\Delta_{k}^{\left(r\right)}$ with respect to the $L^{2}$-inner product by $\Delta_{k}^{\left( r\right) *}:\dom \Delta_{k}^{\left(r\right)*} \to L^{2}_{r}\left(M\right)$, where $\dom \Delta_{k}^{\left(r\right)*}$ consists of elements $\omega$ in $L^{2}_{r}\left(M\right)$ for which there exists a constant $C>0$ such that
\begin{equation*}
\begin{aligned}
	\left| \left(\omega|\Delta_{k}^{ \left( r \right) } \eta\right)\right| \leq C\left\| \eta \right\|
\end{aligned}
\end{equation*}
for each $\eta\in \dom \Delta_{k}^{\left(r\right)}$. In fact, we can see that $\Delta_{k}^{\left(r\right)}$ is self-adjoint and non-negative.

Let $\spe \Delta_{k}^{\left(r\right)}$  be the spectrum of $\Delta_{k}^{\left(r\right)}$ and  $E_{\lambda,k}^{\left(r\right)}\left(M\right)$  be the eigenspace of $\Delta_{k}^{\left(r\right)}$ corresponding to the eigenvalue $\lambda$. We have the following property related to the alternative sum of the dimensions of the eigenspaces. 

\begin{prop} \label{Prop: Alternative sum related to eigenspaces}
For each $k$, for each $r$, and for each $\mu\in \spe \Delta_{k}^{\left(r\right)}\setminus \left\{0\right\}$, we have
\begin{equation*}
\begin{aligned}
	\sum_{j=0}^{r}\left(-1\right)^{r-j}\dim E_{\mu,k}^{\left(j\right)}\left(M\right)=\dim d_{k} \left( E_{\mu,k}^{\left(r\right)}\left(M\right) \right),
\end{aligned}
\end{equation*}
where $d_{k}\left( E_{\mu,k}^{\left(r\right)}\left(M\right) \right)=\left\{d_{k} \omega:\omega\in E_{\mu,k}^{\left(r\right)}\left(M\right) \right\}$. Subsequently, we obtain for each $r$, 
\begin{equation*}
\begin{aligned} 
\sum_{j=0}^{r}\left(-1\right)^{r-j}\dim E_{\mu,k}^{\left(j\right)}\left(M\right)\geq 0,
\end{aligned}
\end{equation*}
and if $r=n$, the equality occurs; namely,
\begin{equation*}
\begin{aligned} 
\sum_{j=0}^{n}\left(-1\right)^{n-j}\dim E_{\mu,k}^{\left(j\right)}\left(M\right)= 0.
\end{aligned}
\end{equation*}
\end{prop}
\begin{proof}
It follows from the fact that the complex $d_{k}: E_{\mu,k}^{\left(r\right)}\left(M\right)\to E_{\mu,k}^{\left(r+1\right)}\left(M\right)$ is exact. 
\end{proof}

For each $t>0$, define the heat operator $e^{-t\Delta_{k}^{\left(r\right)}}:L_{r}^{2}\left(M\right)\to \dom \Delta_{k}^{\left(r\right)}$ such that $e^{-t\Delta_{k}^{\left(r\right)}}\omega\in \Omega^{r}\left(\mathbb{R}^{+}\times M\right)$ and the operator satisfies
\begin{equation*}
\begin{aligned}
\left\{ \begin{array}{l}
\frac{\partial}{\partial t} e^{-t\Delta_{k}^{\left(r\right)}}\omega+ \Delta_{k}^{\left(r\right)} e^{-t\Delta_{k}^{\left(r\right)}}\omega = 0\\
\lim_{t\to 0^{+}} \left\| e^{-t\Delta_{k}^{\left(r\right)}}\omega-\omega\right\|=0
\end{array}\right. .
\end{aligned}
\end{equation*}
The associated distribution kernel
\begin{equation*}
\begin{aligned}
e^{-t\Delta_{k}^{\left(r\right)}}\left(x,y\right)\in \mathcal{C}^{\infty}\left(\mathbb{R}^{+}\times M\times M;\bigwedge^{r}\left(T^{*} M\right) \boxtimes \left(\bigwedge^{r}\left(T^{*}M\right)\right)^{*}\right)
\end{aligned}
\end{equation*}
is called the heat kernel with respect to $\Delta_{k}^{\left(r\right)}$. Here, we denote by $E\boxtimes F^{*}$ the vector bundle over $M\times M$ whose fiber at $\left(x,y\right)$ is the space of linear transformations from $F_{y}$ to $E_{x}$. 

The heat kernel can be expressed in different ways. First, choose a local orthonormal frame $\left\{E^{I}\right\}_{I}$ for $\bigwedge^{r}\left(T^{*} M\right)$ and denote by $\left(E^{I}\right)^{*}$ the dual element to $E^{I}$, and we can write
\begin{equation*}
\begin{aligned}
e^{-t\Delta_{k}^{\left(r\right)}}\left(x,y\right)
=\sideset{}{'}\sum_{I,J} e^{-t\Delta_{k}^{\left(r\right)}}\phantom{}_{I,J}\left(x,y\right) E^{I}\left(x\right) \otimes \left( E^{J}\right)^{*} \left(y\right),
\end{aligned}
\end{equation*}
where $E^{I}\left(x\right)\otimes \left(E^{J}\right)^{*}\left(y\right)\in \bigwedge^{r}T^{*}_{x}M\otimes \left(\bigwedge^{r}T^{*}_{y}M\right)^{*}$ satisfies
\begin{equation*}
\begin{aligned}
\left( E^{I}\left(x\right)\otimes \left(E^{J}\right)^{*}\left(y\right) \right) \left( E_{K}\left(y\right)\right) = \langle E^{K}\left(y\right) | E^{J}\left(y\right) \rangle \cdot  E^{I}\left(x\right) = \delta^{K}_{J} E^{I}\left(x\right),
\end{aligned}
\end{equation*}
and the corresponding component function $e^{-t\Delta_{k}^{\left(r\right)}}\phantom{}_{I,J}\left(x,y\right)\in \mathcal{C}^{\infty}\left(\mathbb{R}^{+}\times M\times M\right)$. 

Set $d_{\lambda}= \dim E_{\lambda,k}^{\left(r\right)}\left(M\right)$ and let $\left\{\varphi_{i}^{\lambda}\right\}_{i=1,\cdots, d_{\lambda},\lambda\in \spe \Delta_{k}^{\left(r\right)}}$  be  a complete orthonormal basis for $L_{r}^{2}\left(M\right)$ such that $\Delta_{k}^{\left(r\right)}\varphi_{i}^{\lambda}= \lambda \varphi_{i}^{\lambda}$, and we can write the heat kernel as
\begin{equation}\label{Eq: Heat Kernel in terms of eigenforms}
\begin{aligned}
e^{-t\Delta_{k}^{\left(r\right)}}\left(x,y\right)=\sum_{\lambda\in \spe \Delta_{k}^{\left(r\right)}} \sum_{i=1}^{d_{\lambda}} e^{-t\lambda} \varphi_{i}^{\lambda}\left(x\right)\otimes \left(\varphi_{i}^{\lambda}\right)^{*}\left(y\right).
\end{aligned}
\end{equation}
In fact, this series converges uniformly on compact subsets of $\mathbb{R}^{+}\times M\times M$.

\subsection{Analytic Tools}\label{f.t. and s.s.}

In this subsection, we review some well-known analytic tools.
First, we review the notion of Sobolev norms and adopt it for differential forms.

We begin by recalling that for each $f\in \mathcal{C}^{\infty}_{c}\left(\mathbb{R}^{n}\right)$, the Fourier transform of $f$ is defined by $\hat{f}\left(\xi\right) = \left(2\pi\right)^{-\frac{n}{2}}\int_{\mathbb{R}^{n}} e^{-\sqrt{-1}x\cdot \xi} f\left(x\right) dx$, where $x\cdot \xi = \sum_{i=1}^{n} x_{i}\xi_{i}$ is the standard dot product in $\mathbb{R}^{n}$ and $dx=dx^{1}\cdots dx^{n}$ is the standard volume element on $\mathbb{R}^{n}$. Recall that the $L^{2}$ space on $\mathbb{R}^{n}$ is given by $L^{2}\left(\mathbb{R}^{n}\right) := \left\{ \text{$f:\mathbb{R}^{n}\to \mathbb{R}$ measurable functions}: \int_{\mathbb{R}^{n}} \left|f\right|^{2} dx \right\}$ together with the inner product $\left(f|g\right):=\int_{\mathbb{R}^{n}} f\overline{g} dx$, where $\overline{g}$ is the complex conjugate of $g$.  It is well-known that we can extend the notion of Fourier transform to $L^{2}$ functions. Recall also that the Parseval's formula: $\left(\widehat{f} | \widehat{g}\right) =\left( f | g \right)$ for any two $f,g\in L^{2}\left(\mathbb{R}^{n}\right)$.

Let $m\in \mathbb{N}\cup \left\{0\right\}$. Define the Sobolev norm of order $m$ on $L^{2}\left(\mathbb{R}^{n}\right)$ by
\begin{equation*}
\begin{aligned}
\left\| f\right\|_{m} := \left(\int_{\mathbb{R}^{n}} \left(1+\left|\xi\right|^{2}\right)^{m} \left|\hat{f}\left(\xi\right)\right|^{2} d\xi\right)^{\frac{1}{2}}.
\end{aligned}
\end{equation*}
We put $W^{m}\left(\mathbb{R}^{n}\right):= \left\{ f\in L^{2}\left(\mathbb{R}^{n}\right): \left\| f\right\|_{m}<\infty \right\}$. By the Parseval's formula, we see that $\left\| f\right\|_{0}= \left\| f\right\|_{L^{2}\left(\mathbb{R}^{n}\right)}$.
Let $U\subset \mathbb{R}^{n}$ be an open subset and we set
\begin{equation*}
\begin{aligned}
W^{m}_{c}\left(U\right) := \left\{ f\in W^{m}\left(\mathbb{R}^{n}\right): \text{$\supp f\subset U$ is compact} \right\}. 
\end{aligned}
\end{equation*}

Let $U$ be a subset of $\mathbb{R}^{n}$. For each $l\in \mathbb{N}\cup\left\{0\right\}$, define the $\mathcal{C}^{l}$-norm on $\mathcal{C}^{l}\left(U\right)$ by
\begin{equation*}
\begin{aligned}
\left\| f\right\|_{\mathcal{C}^{l}\left(U\right)} := \sum_{\left|\alpha\right|\leq l}\sup_{U} \left| \partial^{\alpha}_{x} f \right|
\end{aligned}
\end{equation*}
for each $f\in \mathcal{C}^{l}\left(U\right)$.
Let us state the Sobolev embedding theorem as follows:
\begin{theo} \label{Thm: Sobolev embedding thm}
If $f\in W^{m}\left(\mathbb{R}^{n}\right)$, $m\geq \frac{n}{2}+1+l$, then $f\in \mathcal{C}^{l}\left(\mathbb{R}^{n}\right)$ and
\begin{equation*}
\begin{aligned}
\left\| f\right\|_{\mathcal{C}^{l}\left(\mathbb{R}^{n}\right)} \leq C\left\| f\right\|_{m}. 
\end{aligned}
\end{equation*}
\end{theo}

Let $m\in \mathbb{N}\cup\left\{0\right\}$. We define the Sobolev norm of order $-m$ by duality: given an open subset $U\subset \mathbb{R}^{n}$, let $W^{-m}_{c}\left(U\right)$ be the set consisting of measurable functions $f$ on $U$ for which there exists $C>0$ such that $\left|\int_{U} fg dx \right|\leq C\left\|g\right\|_{m}$ for each $g\in W^{m}_{c}\left(U\right)$; then we define the Sobolev norm of order $-m$ on $W^{m}_{c}\left(U\right)$ by
\begin{equation*}
\begin{aligned}
\left\| f\right\|_{-m} := \sup_{\substack{ 
\text{$g \in W^{m}_{c}\left(U\right)$}\\
\text{$g \neq 0$}
}} \frac{\left| \int_{U} fg dx \right|}{\left\| g \right\|_{m}}.
\end{aligned}
\end{equation*}
If $m=0$, then this duality defined norm coincides the usual $L^{2}$-norm $\left\|\cdot\right\|_{L^{2}\left(U\right)}$, so we included this case as we defined the negative norms.

It is clear to see $\mathcal{C}^{\infty}_{c}\left(U\right)\subset W^{-m}_{c}\left(U\right)$. In fact, the negative norm of $f\in \mathcal{C}^{\infty}_{c}\left(U\right)$ has an upper bound in terms of Fourier transform. 

\begin{prop}\label{Prop: estimate on negative order sobolev norms}
Let $U\subset \mathbb{R}^{n}$ be an open subset and let $f\in \mathcal{C}^{\infty}_{c}\left(U\right)$. Then for each $m\in \mathbb{N}\cup \left\{0\right\}$, 
\begin{equation}\label{Ineq: equivalence norm of negative s. norm}
\begin{aligned}
\left\| f\right\|_{-m} \leq  \left(\int_{\mathbb{R}^{n}}\left(1+\left|\xi\right|^{2}\right)^{-m}\left| \widehat{f}\left(\xi\right)\right|^{2} \ d \xi \right)^{\frac{1}{2}}.
\end{aligned}
\end{equation}
\end{prop}
\begin{proof}
By the Parseval's formula, we derive that
\begin{equation*}
\begin{aligned}
\left|\int_{U} fg d\xi \right| 
&=\left| \int_{\mathbb{R}^{n}} \hat{f}\hat{g} d\xi \right| 
= \left| \int_{\mathbb{R}^{n}} \left(1+\left|\xi\right|^{2}\right)^{-\frac{m}{2}}\hat{f} \cdot \left(1+\left|\xi\right|^{2}\right)^{\frac{m}{2}}\hat{g} d\xi \right|\\
&\leq\left(\int_{\mathbb{R}^{n}}\left(1+\left|\xi\right|^{2}\right)^{-m}\left| \widehat{f}\left(\xi\right)\right|^{2} \ d \xi \right)^{\frac{1}{2}} \left\|g\right\|_{m},
\end{aligned}
\end{equation*}
which implies \eqref{Ineq: equivalence norm of negative s. norm}.
\end{proof}

To extend the notions of Sobolev norms and $\mathcal{C}^{l}$-norms to the Riemannian vector bundle $\left(E,M\right)$ over a Riemannian manifold $M$, we choose a pair $\left(\mathcal{V},\mathcal{P},\mathcal{E}\right)$ as follows: let $\mathcal{V}$ be a set given by chosen coordinate charts of $M$ such that their coordinate domains cover $M$, let $\mathcal{P}$ be a partition of unity $\mathcal{P}$ subordinate to the open cover from $\mathcal{V}$ and $\sum_{\psi\in \mathcal{P}}\psi^{2}=1$, and let $\mathcal{E}$ be the set collecting chosen local  orthonormal frames $\left\{E_{I}\right\}_{I}$ of $E$ over the coordinate domains from $\mathcal{V}$. 

Let $m\in \mathbb{Z}$. For each smooth section $s\in \mathcal{C}^{\infty}_{c}\left(M,E\right)$, define the Sobolev norm of order $m$ by
\begin{equation*}
\begin{aligned}
\left\| s\right\|_{m} 
:= \left( \sum_{\psi_{i}\in \mathcal{P}} \left\| \psi_{i}  s\right\|_{m}^{2}\right)^{\frac{1}{2}},
\end{aligned}
\end{equation*}
where
\begin{equation*}
\begin{aligned}
\left\| \psi_{i}  s\right\|_{m} :=  \left( \sum_{I} \left\| \left(\psi_{i} s_{I} \right)\circ \varphi_{i}^{-1}\right\|_{m}^{2} \right)^{\frac{1}{2}},
\end{aligned}
\end{equation*}
and $\left(V_{i},\varphi_{i}\right) \in \mathcal{V}$, $\psi_{i}s = \sum_{I} \psi_{i} s_{I} E_{I}$ in $V_{i}$, and $\left\| \left( \psi_{i} s_{I}\right) \circ \varphi_{i}^{-1}\right\|_{m}$ is then defined as above. Note that $\left\|\cdot\right\|_{0}=\left\| \cdot\right\|_{L^{2}\left(M,E\right)}$.

Let $U\subset M$ be a subset. For each $l\in \mathbb{N}\cup \left\{0\right\}$, define the $\mathcal{C}^{l}$-norm of $s\in \mathcal{C}^{\infty}\left(M,E\right)$ by
\begin{equation*}
\begin{aligned}
\left\| s\right\|_{\mathcal{C}^{l}\left(U\right)} := \left( \sum_{\psi_{i}\in \mathcal{P}} \left\| \psi_{i}s\right\|_{\mathcal{C}^{l}\left(U\right)}^{2} \right)^{\frac{1}{2}},
\end{aligned}
\end{equation*}
where 
\begin{equation*}
\begin{aligned}
 \left\| \psi_{i}s\right\|_{\mathcal{C}^{l}\left(U\right)}
:=\sup_{U\cap V_{i}} \sum_{\left|\alpha\right|\leq l} \left(\sum_{I} \left|\partial^{\alpha}\left[ \left(\psi_{i} s_{I}\right) \circ \varphi_{i}^{-1}\right]\right|^{2}\right)^{\frac{1}{2}}
\end{aligned}
\end{equation*}
and $\left(V_{i},\varphi_{i}\right) \in \mathcal{V}$, $\psi_{i}s = \sum_{I} \psi_{i} s_{I} E_{I}$ in $V_{i}$.

If both of the manifold $M$ and the vector bundle $E$ are global (for example, $M=\mathbb{R}^{n}$ and $E=\bigwedge^{r}T^{*}\mathbb{R}^{n}$), then choosing a partition of unity to define the Sobolev norms and $\mathcal{C}^{l}$-norms is redundant.

Finally, we review the spectral theorem in functional analysis. 
\begin{theo} \label{Spectral Theorem} Let $X$ be a Hilbert space and let $A:\dom A\subset X\to X$ be a self-adjoint operator with the spectrum $S=\spe A$. Then there exist a finite measure $\mu$ on $S\times \mathbb{N}$ and a unitary operator $U:H \to L^{2}\left(S\times\mathbb{N}\right)$ that is one-to-one, onto with the following properties:
\begin{enumerate}
\item[(a)] Let $\eta\in X$. Then $\eta\in \dom A$ if and only if $s\cdot U\left(\eta\right) \in L^{2}\left(S\times \mathbb{N}\right)$;
\item[(b)] Define the operator $\mathcal{S}$ by 
\begin{equation*}
\begin{aligned}
\mathcal{S}:\dom \mathcal{S}\subset L^{2}\left(S\times \mathbb{N}\right) &\to L^{2}\left(S\times \mathbb{N}\right)\\
g\left(s,n\right) &\mapsto sg\left(s,n\right),
\end{aligned}
\end{equation*}
where 
\begin{equation*}
\begin{aligned}
\dom \mathcal{S} = \left\{ g\left(s,n\right)\in L^{2}\left(S\times \mathbb{N}\right): sg\left(s,n\right)\in  L^{2}\left(S\times \mathbb{N}\right) \right\};
\end{aligned}
\end{equation*}
then $UAU^{-1}=\mathcal{S}$ on $U\left(\dom A\right)$.
\end{enumerate}
\end{theo}

It follows from Theorem \ref{Spectral Theorem} that we can identify the element $\omega\in \dom A$ with the element $g=U\left(\omega\right) \in L^{2}\left(S\times \mathbb{N}\right)$, the operator $A$ with the operator $\mathcal{S}$. Additionally, if $A$ is non-negative, then the heat operator $e^{-tA}$ can be identified with the operator defined by
\begin{equation*}
\begin{aligned}
P\left(t\right): L^{2}\left(S\times \mathbb{N}\right) &\to L^{2}\left(S\times \mathbb{N}\right)\\
g\left(s,n\right) &\mapsto e^{-ts}g\left(s,n\right).
\end{aligned}
\end{equation*}

\section{Scaled Heat Kernel Asymptotics}\label{A.B.S.H.K.}

In this section, we prove Theorem \ref{Main Thm: Semi-classical Heat Kernel Asymptotics}. To do so, we introduce the localization and scaling technique. 

\subsection{Localization}\label{local.}

To capture the local geometric data attached to the critical points of a Morse function, we introduce the locally flat metric as follows:

\begin{theo}[\cite{CM74}, \cite{W82}] \label{Thm: Locally Flat Metric} Let $M$ be a compact orientable smooth manifold of dimension $n$ and $f$ be a Morse function. Then $M$ admits a metric $g$ and coordinate charts $\left(U_{p},\varphi_{p}\right)$ around critical points $p\in \crit\left(f\right)$ such that in each of the coordinate charts $\left(U_{p},\varphi_{p}\right)$, we have
\begin{equation*}
\begin{aligned}
f\left(x\right) =  -\sum_{i=1}^{l} \frac{1}{2}\left(x^{i}\right)^{2} + \sum_{i=l+1}^{n} \frac{1}{2}\left(x^{i}\right)^{2}
\end{aligned}
\end{equation*}
with $l=\ind_{f} p$ and
\begin{equation*}
\begin{aligned}
g\left(\frac{\partial}{\partial x^{i}},\frac{\partial}{\partial x^{j}}\right) = \delta_{j}^{i},
\end{aligned}
\end{equation*}
where $\delta_{j}^{i}$ is the Kronecker delta.
\end{theo}
The construction of this metric is due to the Morse lemma and techniques in \cite{CM74}, but for the reader's convenience, we give a proof:
\begin{proof}[Proof of Theorem \ref{Thm: Locally Flat Metric}]

By the Morse lemma, we have for each $p\in \crit\left(f\right)$, there exists a coordinate chart $\left( V_{p}, \varphi_{p} \right)$ around $p$ such that $f$ is expressed as shown above.

Now, assume $V_{p}\cap V_{p'}=\emptyset$ if $p\neq p'$. For each $p$, let $U_{p}\subset \overline{U_{p}}\subset V_{p}$ and put
\begin{equation*}
\begin{aligned}
W=M\setminus \bigsqcup_{p\in \crit\left(f\right)} \overline{U}_{p}.
\end{aligned}
\end{equation*}
Note that $W$ is open. Let $\left\{W_{\alpha}\right\}_{\alpha}$ be an open cover of $W$ consisting of coordinate neighborhoods. Hence, $\left\{G_{\beta}\right\}_{\beta}=\left\{ W_{\alpha} \right\}_{\alpha} \cup \left\{ V_{p}\right\}_{p\in \crit\left(f\right)}$ is an open cover of $M$. $\left(G_{\beta},\varphi_{\beta}\right)$ denotes the coordinate chart for each $\beta$. Define 
\begin{equation*}
\begin{aligned}
g_{\beta}=\varphi_{\beta}^{*}\overline{g},
\end{aligned}
\end{equation*}
where $\overline{g}$ is the usual Euclidean metric on $\varphi_{\beta}\left( G_{\beta}\right)$ and $\varphi_{\beta} ^{*}\overline{g}$ is the pullback of $\overline{g}$ by $\varphi_{\beta}$.

Finally, let $\left\{\psi_{\beta}\right\}_{\beta}$ be a partition of unity subordinate to the cover $\left\{G_{\beta}\right\}_{\beta}$ and put
\begin{equation*}
\begin{aligned}
g=\sum_{\beta} \psi_{\beta}g_{\beta}.
\end{aligned}
\end{equation*}
Then $g$ is a Riemannian metric on $M$. Note that in each of the coordinate charts $\left( U_{p},\varphi_{p}\right)$, we have
\begin{equation*}
\begin{aligned}
g\left(\frac{\partial}{\partial x^{i}}, \frac{\partial}{\partial x^{j}}\right)=g_{p}\left(\frac{\partial}{\partial x^{i}}, \frac{\partial}{\partial x^{j}}\right)=\delta^{i}_{j},
\end{aligned}
\end{equation*}
since $\psi_{p}\left(x\right)=1$ in $U_{p}$. Therefore, we have furnished the desired metric.

\end{proof}

In the sequel, we adopt this metric throughout this paper together with the local charts around the critical points of a Morse function in question. 

Under such a metric, the local representation \eqref{Eq: Local expression of Witten Laplacian} can be reduced to
\begin{equation}\label{Eq: Local expression of Witten Laplacian under flat metric}
\begin{aligned}
\Delta_{k}^{\left(r\right)} \left( \omega dx^{I}\right)
= \left[ -\sum_{i=1}^{n} \frac{\partial^{2}}{\partial\left(x^{i}\right)^{2}} + k^{2}\left(x^{i}\right)^{2} + k\sum_{i=1}^{n}\varepsilon_{i}\varepsilon_{i}^{I}\right] \omega dx^{I}
\end{aligned}
\end{equation}
under the coordinate chart  $\left(U_{p},\varphi_{p}\right)$ around $p\in \crit f$, where $\varepsilon_{i}$ and $\varepsilon_{i}^{I}$ are defined respectively by
\begin{equation*}
\begin{aligned}
\varepsilon_{i} =\left\{ \begin{array}{ll}
-1 &; i\leq \ind _{f} p\\
1 &; \text{otherwise}
\end{array}\right.
\end{aligned}
\end{equation*}
and by
\begin{equation*}
\begin{aligned}
\varepsilon_{i}^{I} =\left\{ \begin{array}{ll}
1 &; \text{$i$ appears in $I$}\\
-1 &; \text{otherwise}
\end{array} \right. ,
\end{aligned}
\end{equation*}
for each smooth $r$-form $\omega=\omega_{I} d x^{I}$ acting on $\Delta_{k}^{\left(r\right)}$ with $I=\left(i_{1},\cdots,i_{r}\right)$ in increasing order.
As you can see, this local representation looks exactly like the harmonic oscillator with perturbation accordingly by the critical point $p$.

For each $p\in \crit\left(f\right)$, define $\Delta_{f,p}^{\left(r\right)}$ to be the differential operator acting on smooth $r$-forms on an open subset in $\mathbb{R}^{n}$ given by  
\begin{equation}\label{Eq: Model Laplacian}
\begin{aligned}
\Delta_{f,p}^{\left(r\right)}\left(\omega dx^{I}\right) = \Delta_{1}^{\left(r\right)} \left( \omega dx^{I}\right)
= \left[ -\sum_{i=1}^{n} \frac{\partial^{2}}{\partial\left(x^{i}\right)^{2}} +\left(x^{i}\right)^{2} + \sum_{i=1}^{n}\varepsilon_{i}\varepsilon_{i}^{I}\right] \omega dx^{I}.
\end{aligned}
\end{equation}
Note that $\Delta_{f,p}^{\left(r\right)}$ is the system of the harmonic oscillators attached to the critical point $p\in \crit\left(f\right)$ as in Witten's paper \cite{W82}.

\subsection{Scaling Technique}\label{s.t.}

Let $p\in \crit\left(f\right)$ and let $\left(U_{p},\varphi_{p}\right)$ be the coordinate chart around $p$ such that $U_{p}=\varphi_{p}^{-1}\left(B_{\frac{3}{2}}\left(0\right)\right)$ with a fixed $\delta>0$.

Given a sufficiently large $k>0$, put
\begin{equation*}
\begin{aligned}
U_{p}^{k} := \varphi_{p}^{-1}\left(B_{k^{-\frac{1}{2}+\varepsilon}}\left(0\right)\right)
\subset U_{p}
\end{aligned}
\end{equation*}
with $\varepsilon\in \left(0,\frac{1}{2}\right)$.

Let $\omega = \sum_{I}\nolimits' \omega_{I} dx^{I}\in \Omega^{r}\left(B_{k^{\varepsilon}}\left(0\right)\right)$ and define
\begin{equation*}
\begin{aligned}
\omega_{\left[k\right]} \left(x\right) = \sideset{}{'} \sum_{I} \omega_{I}\left(\sqrt{k}x\right) dx^{I} \in \Omega^{r}\left(B_{k^{-\frac{1}{2}+\varepsilon}}\left(0\right)\right).
\end{aligned}
\end{equation*}
Through the coordinate map $\varphi_{p}$, we see that the pulled back form of $\omega_{\left[k\right]}$ is of $\Omega^{r}\left(U_{p}^{k}\right)$ and we still denote it by $\omega_{\left[k\right]}$.

We give the following formula to illustrate how these two operators $\Delta_{f,p}^{\left(r\right)}$ and $\Delta_{k}^{\left(r\right)}$ relate to one another. 

\begin{prop} \label{Prop: Scaling formula}
For each $k>0$, and for each $\omega\in \Omega^{r}\left(B_{k^{\varepsilon}}\left(0\right)\right)$, 
\begin{equation}\label{Scaling formula 1}
\begin{aligned}
\Delta_{f,p}^{\left(r\right)}\omega=\frac{1}{k}\left(\Delta_{k}^{\left(r\right)}\omega_{\left[k\right]}\right)_{\left[\frac{1}{k}\right]}.
\end{aligned}
\end{equation}
\end{prop}
\begin{proof}

This formula follows from the local expression \eqref{Eq: Local expression of Witten Laplacian under flat metric} and change of variables.

\end{proof}

Define the scaled heat kernel at $p\in \crit \left(f\right)$ by
\begin{equation*}
\begin{aligned}
A_{\left(k\right),p}^{r}\left(t,x,y\right) &:= k^{-\frac{n}{2}} e^{-\frac{t}{k} \Delta_{k}^{\left(r\right)}} \left(\frac{x}{\sqrt{k}},\frac{y}{\sqrt{k}}\right)\\
& \in \mathcal{C}^{\infty} \left(\mathbb{R}^{+}\times B_{k^{\varepsilon}}\left(0\right) \times  B_{k^{\varepsilon}}\left(0\right), \bigwedge^{r}\left(T^{*}\mathbb{R}^{n}\right)\boxtimes \left( \bigwedge^{r}\left(T^{*}\mathbb{R}^{n}\right)\right)^{*}\right),
\end{aligned}
\end{equation*}
where $\frac{x}{\sqrt{k}}\in \mathbb{R}^{n}$, for the sake of convenience, stands for $\varphi_{p}^{-1}\left(\frac{x}{\sqrt{k}}\right)$. We can write $A_{\left(k\right),p}^{r}\left(t,x,y\right)$ as
\begin{equation*}
\begin{aligned}
A_{\left(k\right),p}^{r}\left(t,x,y\right) 
= \sideset{}{'} \sum_{I,J} A_{\left(k\right),p}^{r}\phantom{}_{I,J}\left(t,x,y\right) dx^{I}\left(x\right)\otimes \left(dx^{J}\right)^{*}\left(y\right),
\end{aligned}
\end{equation*}
where the component functions
\begin{equation*}
\begin{aligned}
A_{\left(k\right),p}^{r}\left(t,x,y\right)\phantom{}_{I,J}\left(t,x,y\right) = k^{-\frac{n}{2}}e^{-\frac{t}{k}\Delta_{k}^{\left(r\right)}}\phantom{}_{I,J}\left(\frac{x}{\sqrt{k}},\frac{y}{\sqrt{k}}\right).
\end{aligned}
\end{equation*}

For each $t>0$, define the scaled heat operator at $p$ by
\begin{equation*}
\begin{aligned}
A_{\left(k\right),p}^{r}\left(t\right) : \Omega^{r}_{c}\left(B_{k^{\varepsilon}}\left(0\right)\right) &\to \Omega^{r}\left(B_{k^{\varepsilon}}\left(0\right)\right)\\
\omega  &\mapsto \int_{B_{k^{\varepsilon}}\left(p\right)} A_{\left(k\right),p}^{r}\left(t,x,y\right) \omega\left(y\right) \ d y.
\end{aligned}
\end{equation*}
Note that $A_{\left(k\right),p}^{r}\left(t\right)\omega \in \Omega^{r}\left(\mathbb{R}^{+}\times B_{k^{\varepsilon}}\left(0\right)\right)$ for each $\omega\in \Omega^{r}_{c}\left(B_{k^{\varepsilon}}\left(0\right)\right)$.

Let us show how the scaled heat kernel/operator relates to the ordinary heat kernel/operator as follows:

\begin{prop} \label{Prop: Scaling formula for heat kernel} For each $k>0$
\begin{equation}\label{Eq: Scaling formula for heat kernel}
\begin{aligned}
A_{\left(k\right),p}^{r}\left(t\right)\omega
=\left(e^{-\frac{t}{k}\Delta_{k}^{\left(r\right)}}\omega_{\left[k\right]}\right)_{\left[\frac{1}{k}\right]}
\end{aligned}
\end{equation}
for each $\omega\in \Omega^{r}_{c}\left(B_{k^{\varepsilon}}\left(0\right)\right)$.
\end{prop}
\begin{proof}
By change of variables, for each $\omega\in \Omega^{r}_{c}\left(B_{k^{\varepsilon}}\left(0\right)\right)$, we deduce
\begin{equation*}
\begin{aligned}
\left(A_{\left(k\right),p}^{r}\left(t\right)\omega\right)\left(x\right) 
& = k^{-\frac{n}{2}} \int_{B_{k^{\varepsilon}}\left(p\right)} e^{-\frac{t}{k}\Delta_{k}^{\left(r\right)}}\left(\frac{x}{\sqrt{k}},\frac{y}{\sqrt{k}}\right) \omega\left(y\right) \ d y\\
& = \int_{B_{k^{-\frac{1}{2}+\varepsilon}}\left(p\right)} e^{-\frac{t}{k}\Delta_{k}^{\left(r\right)}}\left(\frac{x}{\sqrt{k}},y\right) \omega_{\left[k\right]}\left(y\right) \ d y\\
& = \int_{M} e^{-\frac{t}{k}\Delta_{k}^{\left(r\right)}}\left(\frac{x}{\sqrt{k}},y\right) \omega_{\left[k\right]}\left(y\right) \ d V_{M}
 = \left(e^{-\frac{t}{k}\Delta_{k}^{\left(r\right)}}\omega_{\left[k\right]}\right)_{\left[\frac{1}{k}\right]}\left(x\right).
\end{aligned}
\end{equation*}
\end{proof}

Now, we have the following important observation:
\begin{lem} \label{Lem: s.h.k satisfies h.e.}
For each $k>0$
\begin{equation*}
\begin{aligned}
\left\{ \begin{array}{l}
\frac{\partial}{\partial t} A_{\left(k\right),p}^{r}\left(t\right)\omega+ \Delta_{f,p}^{\left(r\right)} A_{\left(k\right),p}^{r}\left(t\right)\omega = 0 \\
\lim_{t\to 0^{+}} \left\| A_{\left(k\right),p}^{r}\left(t\right)\omega - \omega\right\|_{L^{2}\left(B_{k^{\varepsilon}}\left(0\right)\right)}=0
\end{array}\right. 
\end{aligned}
\end{equation*}
for each $\omega\in \Omega_{c}^{r}\left(B_{k^{\varepsilon}}\left(0\right)\right)$.
\end{lem}
\begin{proof}
This lemma essentially follows from \eqref{Scaling formula 1} and \eqref{Eq: Scaling formula for heat kernel}. Given $\omega  \in \Omega_{c}^{r} \left(B_{k^{\varepsilon}}\left(0\right)\right)$, we deduce
\begin{equation*}
\begin{aligned}
\frac{\partial}{\partial t} A_{\left(k\right),p}^{r}\left(t\right)\omega 
&= \frac{1}{k} \left( \frac{\partial}{\partial t}  e^{-\frac{t}{k}\Delta_{k}^{\left(r\right)}}\omega_{\left[k\right]}\right)_{\left[\frac{1}{k}\right]}
=-\frac{1}{k}\left( \Delta_{k}^{\left(r\right)}e^{-\frac{t}{k}\Delta_{k}^{\left(r\right)}}\omega_{\left[k\right]}\right)_{\left[\frac{1}{k}\right]}\\
&=-\Delta_{f,p}^{\left(r\right)}\left(e^{-\frac{t}{k}\Delta_{k}^{\left(r\right)}}\omega_{\left[k\right]}\right)_{\left[\frac{1}{k}\right]}
=-\Delta_{f,p}^{\left(r\right)}A_{\left(k\right),p}^{r}\left(t\right)\omega.
\end{aligned}
\end{equation*}
where the firste equation holds by the chain rule.
Hence, we obtain
\begin{equation*}
\begin{aligned}
\frac{\partial}{\partial t} A_{\left(k\right),p}^{r} \left(t\right) \omega + \Delta_{f,p}^{\left(r\right)} A_{\left(k\right),p}^{r}\left(t\right) \omega = 0
\end{aligned}
\end{equation*}
for each $\omega\in \Omega^{r}_{c}\left(B_{k^{\varepsilon}}\left(0\right)\right)$.

Finally, note that
\begin{equation*}
\begin{aligned}
\lim_{t\to 0^{+}}\left\| A_{\left(k\right),p}^{r}\left(t\right)\omega -\omega\right\|_{L^{2}\left(B_{k^{\varepsilon}\left(0\right)}\right)}
&=\lim_{t\to 0^{+}} \left\| e^{-\frac{t}{k}\Delta_{k}^{\left(r\right)}}\omega_{\left[k\right]}-\omega_{\left[k\right]}\right\|_{L^{2}\left(B_{k^{-\frac{1}{2}+\varepsilon}}\left(0\right)\right)}\\
&=\lim_{t\to 0^{+}} \left\| e^{-\frac{t}{k}\Delta_{k}^{\left(r\right)}}\omega_{\left[k\right]}-\omega_{\left[k\right]}\right\|_{L^{2}\left(M\right)}=0.
\end{aligned}
\end{equation*}
\end{proof}

Lemma \ref{Lem: s.h.k satisfies h.e.} motivates us to seek a local bound for scaled heat kernels stated as follows:

\begin{theo}[Local Boundedness of the Scaled Heat Kernels] \label{Main Thm 1: locally uniform boundedness of s.h.}
Given $p\in \crit\left(f\right)$, let $T$ and $K$ be compact subsets in $\mathbb{R}^{+}$ and in  $\mathbb{R}^{n}$, respectively. Then for each $l\in \mathbb{N}\cup\left\{ 0\right\}$, the sequence $\left\{A_{\left(k\right),p}^{r}\left(t,x,y\right) \right\}_{k}$ is uniformly bounded in $T\times K \times K$ with respect to the $\mathcal{C}^{l}$-norm; namely, there exists a constant $C\left(T,  K\right)>0$ such that
\begin{equation*}
\begin{aligned}
\left\| A_{\left(k\right),p}^{r} \left(t,x,y\right) \right\|_{\mathcal{C}^{l}\left(T\times K\times K\right)}\leq C\left(T,K\right)
\end{aligned}
\end{equation*}
for each $k$ sufficiently large. 
\end{theo}

We will prove Theorem \ref{Main Thm 1: locally uniform boundedness of s.h.} in Section \ref{m.p.s.h.o}.

\subsection{Locally Uniform Bound for Scaled Heat Kernels} \label{m.p.s.h.o}

We find a locally uniform bound through investigating the scaled heat operators. The key of finding such a bound is that we manage to show the scaled heat operators are "locally uniformly bounded" by a constant independent of $k$. We illustrate it in what we call the mapping property as follows:

\begin{theo}[Mapping Property] \label{Critical Thm: mapping property of s.h.o.}  
Given a compact subset $K\subset \mathbb{R}^{n}$, choose a bounded open subset $U$ such that $K\subset U$ and choose a cut-off function $\chi\in \mathcal{C}^{\infty}_{c}\left(U\right)$ such that $\chi=1$ in $K$. Then for each $p\in \crit\left(f\right)$ and for each cut-off function $\tilde{\chi}\in \mathcal{C}^{\infty}_{c}\left(U\right)$,  there exists a constant $C\left(\chi,\tilde{\chi},t\right)>0$ such that
\begin{equation*}
\begin{aligned}
\left\| \tilde{\chi} A_{\left(k\right),p}^{r}\left(t\right)\chi \omega \right\|_{2m}\leq C\left(\chi,\tilde{\chi},t\right)\left\| \omega\right\|_{-2m},
\end{aligned}
\end{equation*}
where $C\left(\chi,\tilde{\chi},t\right)$ depends on the choices of $\chi$ and $\tilde{\chi}$ and smoothly on $t$ only, for each $\omega\in \Omega^{r}_{c}\left(U\right)$, if $k$ large enough.
\end{theo}
\begin{proof}

Before we start, we point out that the following reasoning works for $k$ large enough for us to have $U\subset B_{k^{\varepsilon}}\left(0\right)$, $\varepsilon\in \left(0,\frac{1}{2}\right)$.

Firstly, by the Gårding's inequality, we obtain
\begin{equation*}
\begin{aligned}
\left\| \tilde{\chi}A^{r}_{\left(k\right),p}\left(t\right)\chi \omega\right\|_{2m}
\leq C_{1}\left(\tilde{\chi}\right) \left\| \chi_{1}\left(\Delta_{f,p}^{\left(r\right)}\right)^{m}A_{\left(k\right),p}^{r}\left(t\right) \chi \omega\right\|_{0} + C_{2}\left(\tilde{\chi}\right)\left\| \chi_{2}A_{\left(k\right),p}^{r}\left(t\right)\chi\omega\right\|_{0},
\end{aligned}
\end{equation*}
where $\chi_{1},\chi_{2}\in \mathcal{C}^{\infty}_{c}\left(U\right)$ are the cut-off functions chosen to satisfy $\chi_{1}=1=\chi_{2}$ in $\supp \tilde{\chi}$. Therefore, we reduce to the $L^{2}$-norm estimates for the two term on the right. 

Next, to estimate the $L^{2}$-norm of the first term, observe that
\begin{equation*}
\begin{aligned}
\left\| \chi_{1}\left(\Delta_{f,p}^{\left(r\right)}\right)^{m}A_{\left(k\right),p}^{r}\left(t\right)\chi \omega\right\|_{0} := \sup_{\eta\in \Omega^{r}_{c}\left(U\right), \eta\neq 0} \frac{\left|\left( \chi_{1}\left(\Delta_{f,p}^{\left(r\right)}\right)^{m}A_{\left(k\right),p}^{r}\left(t\right)\chi \omega \bigg| \eta\right)\right|}{\left\|\eta\right\|_{0}},
\end{aligned}
\end{equation*}
so we consider the inner product on the right hand side. By Theorem \ref{Spectral Theorem}, \eqref{Scaling formula 1} and \eqref{Eq: Scaling formula for heat kernel}, we see that
\begin{equation*}
\begin{aligned}
\left| \left( \chi_{1}\left(\Delta_{f,p}^{\left(r\right)}\right)^{m}A_{\left(k\right),p}^{r}\left(t\right)\chi \omega \bigg| \eta\right)\right|
&=k^{-m}\left| \left( \left(\chi_{1}\right)_{\left[k\right]}\left( \Delta_{k}^{\left(r\right)}\right)^{m} e^{-\frac{t}{k}\Delta_{k}^{\left(r\right)}}\chi_{\left[k\right]}\omega_{\left[k\right]} \bigg| \eta_{\left[k\right]} \right)\right|\\
&=k^{-m}\left| \left( \omega_{\left[k\right]} \bigg| \chi_{\left[k\right]} e^{-\frac{t}{k}\Delta_{k}^{\left(r\right)}}\left(\Delta_{k}^{\left(r\right)}\right)^{m} \left(\chi_{1}\right)_{\left[k\right]}\eta_{\left[k\right]} \right)\right|\\
&=\left| \left( \omega \bigg| \chi A_{\left(k\right),p}^{r}\left(t\right)\left(\Delta_{f,p}^{\left(r\right)}\right)^{m} \chi_{1}\eta \right)\right|
\end{aligned}
\end{equation*}
for each $\eta\in \Omega^{r}_{c}\left(U\right)$ with $\eta\neq 0$. Moreover, by definition of Sobolev norms, we see that
\begin{equation*}
\begin{aligned}
\left| \left( \omega \bigg| \chi A_{\left(k\right),p}^{r}\left(t\right)\left(\Delta_{f,p}^{\left(r\right)}\right)^{m} \chi_{1}\eta \right)\right|
\leq\left\| \omega\right\|_{-2m}\left\| \chi A_{\left(k\right),p}^{r}\left(t\right)\left(\Delta_{f,p}^{\left(r\right)}\right)^{m} \chi_{1}\eta \right\|_{2m}.
\end{aligned}
\end{equation*}
Again, by the Gårding's inequality, we obtain
\begin{equation*}
\begin{aligned}
\left\| \chi A_{\left(k\right),p}^{r}\left(t\right)\left(\Delta_{f,p}^{\left(r\right)}\right)^{m} \chi_{1}\eta \right\|_{2m}
&\leq C_{3}\left(\chi\right) \left\| \chi_{3}\left(\Delta_{f,p}^{\left(r\right)}\right)^{m} A_{\left(k\right),p}^{r}\left(t\right) \left(\Delta_{f,p}^{\left(r\right)}\right)^{m} \chi_{1}\eta \right\|_{0} \\
&+C_{4}\left(\chi\right) \left\| \chi_{4}A_{\left(k\right),p}^{r}\left(t\right) \left(\Delta_{f,p}^{\left(r\right)}\right)^{m} \chi_{1}\eta \right\|_{0},
\end{aligned}
\end{equation*}
where $\chi_{3},\chi_{4}\in \mathcal{C}^{\infty}_{c}\left(U\right)$ are the cut-off functions chosen to satisfy $\chi_{3}=1=\chi_{4}$ in $\supp \chi$. Moreover, by Theorem \ref{Spectral Theorem}, \eqref{Scaling formula 1} and \eqref{Eq: Scaling formula for heat kernel}, we derive
\begin{equation*}
\begin{aligned}
\left\| \chi_{3}\left(\Delta_{f,p}^{\left(r\right)}\right)^{m} A_{\left(k\right),p}^{r}\left(t\right) \left(\Delta_{f,p}^{\left(r\right)}\right)^{m} \chi_{1}\eta \right\|_{0}^{2} 
&\leq k^{\frac{n}{2}-4m} \left\| \left( \Delta_{k}^{\left(r\right)}\right)^{m} e^{-\frac{t}{k}\Delta_{k}^{\left(r\right)}} \left( \Delta_{k}^{\left(r\right)}\right)^{m} \left(\chi_{1}\eta\right)_{\left[k\right]} \right\|_{L^{2}\left(M\right)}^{2}\\
&\leq k^{\frac{n}{2}} \int_{\mathcal{S}\times \mathbb{N}} \left( \frac{s}{k}\right)^{4m}e^{-2t\frac{s}{k}} \left|g\left(s,n\right)\right|^{2} d\mu \\
&\leq k^{\frac{n}{2}} C_{5}\left(t\right) \int_{\mathcal{S}\times \mathbb{N}} \left|g\left(s,n\right)\right|^{2} d\mu \\
&= C_{5}\left(t\right) \left\| \chi_{1}\eta\right\|_{0}^{2} \leq C_{5}\left(t\right) \left\| \eta\right\|_{0}^{2},
\end{aligned}
\end{equation*}
where $C_{5}\left(t\right) = \left(\frac{2m}{t}\right)^{4m}e^{-4m}>0$ depends smoothly on $t$. Hence, we conclude
\begin{equation}\label{Ineq: L^2 estimate 1 for dual 1}
\begin{aligned}
\left\| \chi_{3}\left(\Delta_{f,p}^{\left(r\right)}\right)^{m} A_{\left(k\right),p}^{r}\left(t\right) \left(\Delta_{f,p}^{\left(r\right)}\right)^{m} \chi_{1}\eta \right\|_{0}
\leq C_{6}\left(t\right) \left\|\eta\right\|_{0},
\end{aligned}
\end{equation}
where $C_{6}\left(t\right)=\sqrt{C_{5}\left(t\right)}$. Similarly, we can obtain
\begin{equation}\label{Ineq: L^2 estimate 1 for dual 2}
\begin{aligned}
\left\| \chi_{4}A_{\left(k\right),p}^{r}\left(t\right) \left(\Delta_{f,p}^{\left(r\right)}\right)^{m} \chi_{1}\eta \right\|_{0} \leq C_{7}\left(t\right) \left\| \eta\right\|_{0},
\end{aligned}
\end{equation}
where $C_{7}\left(t\right)$ depends smoothly on $t$. By \eqref{Ineq: L^2 estimate 1 for dual 1} and \eqref{Ineq: L^2 estimate 1 for dual 2}, we conclude
\begin{equation*}
\begin{aligned}
\left\| \chi A_{\left(k\right),p}^{r}\left(t\right) \left( \Delta_{f,p}^{\left(r\right)}\right)^{m} \chi_{1}\eta\right\|_{2m} \leq C_{8}\left(\chi,t\right) \left\|\eta\right\|_{0},
\end{aligned}
\end{equation*}
for each $\eta\in \Omega^{r}_{c}\left(U\right)$ with $\eta\neq 0$, which in turn, implies
\begin{equation}\label{Ineq: Reduced term 1 by Garding}
\begin{aligned}
\left\| \chi_{1}\left(\Delta_{f,p}^{\left(r\right)}\right)^{m}A_{\left(k\right),p}^{r}\left(t\right)\chi \omega\right\|_{0} \leq C_{9}\left(\chi,t\right) \left\| \omega\right\|_{-2m},
\end{aligned}
\end{equation}
where $C_{9}\left(\chi,t\right)$ depends on $\chi$ and smoothly on $t$.

Using the similar argument, we can obtain
\begin{equation}\label{Ineq: Reduced term 2 by Garding}
\begin{aligned}
\left\| \chi_{2}A_{\left(k\right),p}^{r}\left(t\right)\chi \omega\right\|_{0}\leq C_{10}\left(\chi,t\right) \left\| \omega\right\|_{-2m},
\end{aligned}
\end{equation}
where $C_{9}\left(\chi,t\right)$ depends on $\chi$ and smoothly on $t$.

Finally, by \eqref{Ineq: Reduced term 1 by Garding} and \eqref{Ineq: Reduced term 2 by Garding}, we have established
\begin{equation*}
\begin{aligned}
\left\| \tilde{\chi} A_{\left(k\right),p}^{r}\left(t\right)\chi \omega \right\|_{2m} \leq C\left(\chi, \tilde{\chi},t\right) \left\| \omega\right\|_{-2m},
\end{aligned}
\end{equation*}
where $C\left(\chi, \tilde{\chi},t\right)$ depends on $\chi, \tilde{\chi}$ and smoothly on $t$, for each $\omega \in \Omega^{r}_{c}\left(U\right)$ with $\supp \omega \subset K$. 

\end{proof}

Now, to show Theorem \ref{Main Thm 1: locally uniform boundedness of s.h.}, 
let us recall the definition of approximate identities. Let $B_{1}\left(0\right)$ be the unit Euclidean ball centered at 0 and choose a cut-off function $\chi\in \mathcal{C}^{\infty}_{c}\left(B_{1}\left(0\right)\right)$ such that
\begin{equation*}
\begin{aligned}
\int_{\mathbb{R}^{n}} \chi\left(x\right) \ d x  = 1.
\end{aligned}
\end{equation*}
For each $x_{0}\in \mathbb{R}^{n}$ and for each $\delta>0$, define
\begin{equation} \label{Defi: approximate identity}
\begin{aligned}
\chi_{x_{0},\delta}\left(x\right) := \delta^{-n} \chi \left(\frac{x-x_{0}}{\delta} \right) \in \mathcal{C}^{\infty}_{c}\left(B_{\delta}\left(x_{0}\right)\right).
\end{aligned}
\end{equation}
We call the family $\left\{\chi_{x_{0},\delta}\right\}_{\delta}$ an approximate identity with respect to the point $x_{0}$.
Note that a change of variable gives
\begin{equation*}
\begin{aligned}
\int_{\mathbb{R}^{n}} \chi_{x_{0},\delta}\left(x\right)  \ d x  
= \int_{\mathbb{R}^{n}} \chi\left(x\right)  \ d x 
= 1.
\end{aligned}
\end{equation*}

An important feature of an approximate identity is illustrated as below.

\begin{lem}\label{Lem: subseqential convergence of seq. of norms of a.i.} 
Given a multi-index $\alpha$, there exists a constant $C>0$ such that for each $x_{0}\in \mathbb{R}^{n}$, for each $\alpha$ and for each $I$, set
\begin{equation}\label{Eq: Approximated identity for form}
\begin{aligned}
\omega_{\alpha,x_{0},I,\delta} = \partial^{\alpha}\chi_{x_{0},\delta} \ d x^{I} \in \Omega^{r}_{c}\left(B_{\delta}\left(x_{0}\right)\right)
\end{aligned}
\end{equation}
and we obtain
\begin{equation*}
\begin{aligned}
\limsup_{\delta\to 0} \left\| \omega_{\alpha, x_{0},I,\delta}\right\|_{-m} \leq C,
\end{aligned}
\end{equation*}
if $m$ is sufficiently large. 
\end{lem}

\begin{proof}

By Proposition \ref{Prop: estimate on negative order sobolev norms}, it suffices to show there exists $C>0$ such that 
\begin{equation*}
\begin{aligned}
\int_{\mathbb{R}^{n}}\left(1+\left|\xi\right|^{2}\right)^{-m} \left| \widehat{\omega_{x_{0},\alpha,I,\delta}}\left(\xi\right) \right|^{2} \ d\xi
\leq C
\end{aligned}
\end{equation*}
for each $\delta$ and for each $x_{0}$, if $m$ large enough.

First, observe that
\begin{equation*}
\begin{aligned}
\left| \widehat{\omega_{x_{0},\alpha,I,\delta}}\left(\xi\right) \right|^{2} 
= \left| \int_{\mathbb{R}^{n}} e^{-\sqrt{-1}x\cdot\xi} \partial^{\alpha} \chi_{x_{0},\delta}\left(x\right) \ dx \right|^{2}
\leq \left|\xi\right|^{2\left|\alpha\right|} \int_{\mathbb{R}^{n}} \chi \ dx = \left|\xi\right|^{2\left|\alpha\right|}.
\end{aligned}
\end{equation*}

Now, if $m\geq \frac{n}{2}+\left|\alpha\right|+1$, then 
\begin{equation*}
\begin{aligned}
\int_{\mathbb{R}^{n}}\left(1+\left|\xi\right|^{2}\right)^{-m} \left| \widehat{\omega_{x_{0},\alpha,I,\delta}}\left(\xi\right) \right|^{2} d\xi
&\leq \int_{\mathbb{R}^{n}} \left(1+\left|\xi\right|^{2}\right)^{-m}\left|\xi\right|^{2\left|\alpha\right|} d\xi \\
&\leq C_{1}\int_{\mathbb{R}^{n}} \left(1+\left|\xi \right|^{2}\right)^{-m+\left|\alpha\right|} d\xi \leq C_{2},
\end{aligned}
\end{equation*}
where $C_{1},C_{2}$ are independent of $\delta$ and the choice of $x_{0}$. Hence, we have established Lemma \ref{Lem: subseqential convergence of seq. of norms of a.i.}.

\end{proof}

\begin{proof}[Proof of Theorem \ref{Main Thm 1: locally uniform boundedness of s.h.}]

Let $T$ be a compact interval in $\mathbb{R}^{+}$ and $K$ be a compact subset in $\mathbb{R}^{n}$. We choose $M_{K}\in \mathbb{N}$ such that
\begin{equation*}
\begin{aligned}
K\subset B_{k^{\varepsilon}}\left(0\right)
\end{aligned}
\end{equation*}
for all $k\geq M_{K}$ and, without loss of generosity, discard the rest $k$'s for which $K\not \subset B_{k^{\varepsilon}}\left(0\right)$. 

Write 
\begin{equation*}
\begin{aligned}
A_{\left(k\right),p}^{r}\left(t,x,y\right) 
= \sideset{}{'}\sum_{I,J} A_{\left(k\right),p}^{r}\left(t,x,y\right) dx^{I}\left(x\right)\otimes \left(dx^{J}\right)^{*}\left(y\right).
\end{aligned}
\end{equation*}
To establish Theorem \ref{Main Thm 1: locally uniform boundedness of s.h.}, it suffices to show its component functions are locally uniformly bounded; in other words, we show for any two multi-indices $I',J'$, there exists a constant $C\left(T,K\right)$ such that
\begin{equation*}
\begin{aligned}
\left\| A_{\left(k\right)}^{r} \phantom{}_{I',J'}\left(t,x,y\right)\right\|_{\mathcal{C}^{l}\left(T\times K\times K\right)} \leq C\left(T,K\right).
\end{aligned}
\end{equation*}

First of all, we show that for each $t_{0}\in T$, for each $y_{0}\in K$, for each $\alpha$, and for each $J'$, 
\begin{equation}\label{Eq: Approximated identity applied to A_(k)}
\begin{aligned}
\lim_{\delta\to 0} \left\| A_{\left(k\right),p}^{r}\left(t_{0}\right)\omega_{\alpha,y_{0},J',\delta} - \sideset{}{'}\sum_{I} \partial^{\alpha}_{y}A_{\left(k\right),p}\phantom{}_{I,J'}\left(t_{0},x,y_{0}\right) dx^{I}\right\|_{\mathcal{C}^{l}\left(K\right)} = 0,
\end{aligned}
\end{equation}  
where $\omega_{\alpha,y_{0},J',\delta}$ is as defined in \eqref{Eq: Approximated identity for form}. To see it, for each $x\in K$, note that
\begin{equation*}
\begin{aligned}
&\sum_{\left|\alpha\right|\leq l}\left( \sideset{}{'}\sum_{I} \left| \partial_{x}^{\gamma}\left(\int_{\mathbb{R}^{n}} A_{\left(k\right),p}^{r}\phantom{}_{I,J'}\left(t_{0},x,y\right)\partial^{\alpha}_{y}\chi_{y_{0},\delta}\  dy -\partial^{\alpha}_{y}A_{\left(k\right),p}^{r} \phantom{}_{I,J'}\left(t_{0},x,y_{0}\right)\right) \right|^{2}\right)^{\frac{1}{2}}\\
&=\sum_{\left|\alpha\right|\leq l}\left( \sideset{}{'}\sum_{I} \left| \left(\int_{\mathbb{R}^{n}} \partial_{x}^{\gamma} \left( \partial^{\alpha}_{y} A_{\left(k\right),p}^{r}\phantom{}_{I,J'}\left(t_{0},x,y\right)-\partial^{\alpha}_{y}A_{\left(k\right),p}^{r}\phantom{}_{I,J'} \left(t_{0},x,y_{0}\right)\right) \chi_{y_{0},\delta}\  dy \right) \right|^{2}\right)^{\frac{1}{2}} \\
&=\sum_{\left|\alpha\right|\leq l} \left( \sideset{}{'}\sum_{I} \left| \left(\int_{\mathbb{R}^{n}} \left( \partial_{x}^{\gamma} \partial^{\alpha}_{y} A_{\left(k\right),p}^{r}\phantom{}_{I,J'}\left(t_{0},x,\delta y+y_{0}\right)-\partial_{x}^{\gamma} \partial^{\alpha}_{y}A_{\left(k\right),p}^{r}\phantom{}_{I,J'} \left(t_{0},x,y_{0}\right)\right) \chi \  dy \right) \right|^{2}\right)^{\frac{1}{2}}.
\end{aligned}
\end{equation*}
Now, since the function $\partial_{x}^{\gamma} \partial^{\alpha}_{y} A_{\left(k\right),p}^{r}\phantom{}_{I,J'}\left(t_{0},x,y\right)$ is uniformly continuous in $K\times K$, we obtain \eqref{Eq: Approximated identity applied to A_(k)} and thus
\begin{equation}\label{Eq: Limit of C^l norm}
\begin{aligned}
\left\| \sideset{}{'}\sum_{I} \partial_{y}^{\alpha}A_{\left(k\right),p}^{r}\phantom{}_{I,J'}\left(t_{0},x,y_{0}\right) dx^{I}\right\|_{\mathcal{C}^{l}\left(K\right)} = \lim_{\delta\to 0} \left\| A_{\left(k\right),p}\left(t_{0}\right)\omega_{y_{0},\alpha,J',\delta} \right\|_{\mathcal{C}^{l}\left(K\right)}. 
\end{aligned}
\end{equation}

Next, we deduce to find a local bound. Let $W$ be a chosen open subset such that $K\subset W \Subset B_{k^{\varepsilon}}\left(p\right)$ for each $k\geq M_{K}$ and let $\chi, \tilde{\chi}\in \mathcal{C}^{\infty}_{c}\left(W\right)$ be two cut-off functions such that $\chi=1$ in $K$ and $\tilde{\chi}=1$ in $\supp \chi$. By Theorem \ref{Thm: Sobolev embedding thm}, there exists $C_{1}>0$ independent of $k$ and $y_{0}$, such that
\begin{equation}\label{Ineq: Sobolev embedding applied to heat kernel}
\begin{aligned}
\left\| A_{\left(k\right),p}\left(t_{0}\right)\omega_{y_{0},\alpha,J',\delta} \right\|_{\mathcal{C}^{l}\left(K\right)}
\leq C_{1} \left\| \tilde{\chi} A_{\left(k\right),p}^{r}\left(t_{0}\right) \chi\omega_{y_{0},\alpha,J',\delta} \right\|_{2m},
\end{aligned}
\end{equation}
for some  $m\in \mathbb{N}$. Moreover, by Theorem \ref{Critical Thm: mapping property of s.h.o.}, we obtain
\begin{equation}\label{Ineq: Mapping prop applied to heat kernel}
\begin{aligned}
\left\| \tilde{\chi} A_{\left(k\right),p}^{r}\left(t_{0}\right) \chi\omega_{y_{0},\alpha,J',\delta} \right\|_{2m} 
\leq C_{2}\left(\chi,\tilde{\chi},t\right) \left\| \omega_{y_{0},\alpha,J,\delta}\right\|_{-2m} \leq  C_{3}\left(\chi,\tilde{\chi},t_{0}\right),
\end{aligned}
\end{equation}
where $C_{2},C_{3}$ depend on $\chi, \tilde{\chi}, t_{0}$, but on neither $k$ nor $y_{0}$. Hence, by \eqref{Eq: Limit of C^l norm}, \eqref{Ineq: Sobolev embedding applied to heat kernel} and \eqref{Ineq: Mapping prop applied to heat kernel}, we conclude
\begin{equation}\label{Ineq: Locally bound for x,y deri}
\begin{aligned}
\left\| A_{\left(k\right),p}^{r}\phantom{}_{I',J'}\left(t_{0},x,y\right) \right\|_{\mathcal{C}^{l}\left(K\times K\right)}
\leq \left\| \sideset{}{'}\sum_{I} A_{\left(k\right),p}^{r}\phantom{}_{I,J'}\left(t_{0},x,y\right) dx^{I}\right\|_{\mathcal{C}^{l}\left(K\times K\right)}
\leq C_{3}\left(\chi,\tilde{\chi},t_{0}\right)
\end{aligned}
\end{equation}
for any two $I',J'$. 

Finally, to establish Theorem \ref{Main Thm 1: locally uniform boundedness of s.h.}, it remains to deal with $t$-derivatives. By nature of scaled heat kernels, we see
\begin{equation*}
\begin{aligned}
\partial_{t}^{\beta} A_{\left(k\right),p}^{r}\left(t\right)\omega
=-\left( \Delta_{f,p}^{\left(r\right)}\right)^{\beta} A_{\left(k\right),p}^{r}\left(t\right)\omega
\end{aligned}
\end{equation*}
for each $\omega\in \Omega^{r}_{c}\left(W\right)$ and the constant $C_{3}$ depends smoothly on $t_{0}\in T$ (by Theorem \ref{Critical Thm: mapping property of s.h.o.}), so by \eqref{Ineq: Locally bound for x,y deri}, we conclude that for each $l\in \mathbb{N}\cup \left\{0\right\}$ and for each $I',J'$
\begin{equation*}
\begin{aligned}
\left\| A_{\left(k\right),p}^{r}\phantom{}_{I',J'}\left(t,x,y\right) \right\|_{\mathcal{C}^{l}\left(T\times K\times K\right)} \leq C\left(T, K\right),
\end{aligned}
\end{equation*}
where $C$ depends on $T$ and $K$. Therefore, Theorem \ref{Main Thm 1: locally uniform boundedness of s.h.} has been shown. 
\end{proof}

\subsection{Proof of Theorem 1.1}\label{s.h.k.a.}

In this subsection, we prove Theorem \ref{Main Thm: Semi-classical Heat Kernel Asymptotics}.

Let $\left\{T_{i}\right\}_{i=1}^{\infty}\subset \mathbb{R}^{+}$ be a collection of compact subsets satisfying $T_{i}\subset T_{i+1}^{\circ}$ and 
$\bigcup_{i=1}^{\infty} T_{i}=\mathbb{R}^{+}$
and let $\left\{K_{i}\right\}_{i=1}^{\infty}\subset \mathbb{R}^{n}$ be a collection of compact subsets satisfying $K_{i}\subset K_{i+1}^{\circ}$ and
$\bigcup_{i=1}^{\infty} K_{i}=\mathbb{R}^{n}.$

Owing to Theorem \ref{Main Thm 1: locally uniform boundedness of s.h.} and the Arzela-Ascoli theorem, there exists a strictly increasing sequence $\left\{n_{1,k}\right\}_{k}\subset \mathbb{N}$ such that the subsequence $\left\{ A_{\left(n_{1,k}\right),p}^{r}\left(t,x,y\right)\right\}_{k}$ that converges in $\mathcal{C}^{1}$-norm in $T_{1}\times K_{1}\times K_{1}$. Again, according to Theorem \ref{Main Thm 1: locally uniform boundedness of s.h.} and the Arzela-Ascoli theorem, there exists a strictly increasing sequence $\left\{n_{2,k}\right\}_{k}\subset \left\{n_{1,k}\right\}_{k}$ such that the subsequence $\left\{ A_{\left(n_{2,k}\right),p}^{r}\left(t,x,y\right)\right\}_{k}$ that converges in $\mathcal{C}^{2}$-norm in $T_{2}\times K_{2}\times K_{2}$. We proceed in the same manner to obtain a subsequence $\left\{ A_{\left(n_{i,k}\right),p}^{r}\left(t,x,y\right)\right\}_{k}$ for each $i$ that converges in $\mathcal{C}^{i}$-norm in $T_{i}\times K_{i}\times K_{i}$. Finally, the diagonal argument enables us to find a subsequence $\left\{ A_{\left(n_{k}\right),p}^{r}\left(t,x,y\right)\right\}_{n_{k}}$ such that
\begin{equation*}
\begin{aligned}
\lim_{k\to \infty} A_{\left(n_{k}\right),p}^{r}\left(t,x,y\right) = B^{r}_{p}\left(t,x,y\right)
\end{aligned}
\end{equation*}
in $\mathcal{C}^{\infty}$-topology in each compact subset of $\mathbb{R}^{+}\times \mathbb{R}^{n}\times \mathbb{R}^{n}$. 
Note that
\begin{equation*}
\begin{aligned}
B^{r}_{p} \left(t,x,y\right) 
\in \mathcal{C}^{\infty}\left( \mathbb{R}^{+}\times \mathbb{R}^{n}\times \mathbb{R}^{n}; \bigwedge^{r}\left(T^{*}\mathbb{R}^{n}\right) \boxtimes \left(\bigwedge^{r}\left(T^{*}\mathbb{R}^{n}\right) \right)^{*}\right).
\end{aligned}
\end{equation*}
In addition, for each $t>0$, define the operator $B^{r}_{p}\left(t\right):\Omega^{r}_{c}\left(\mathbb{R}^{n}\right)\to \Omega^{r}\left( \mathbb{R}^{n}\right)$ by
\begin{equation*}
\begin{aligned}
\left( B^{r}_{p}\left(t\right)\omega\right) \left(x\right) 
= \int_{\mathbb{R}^{n}} B^{r}_{p}\left(t,x,y\right)\omega\left(y\right) \ d y.
\end{aligned}
\end{equation*}
Note that $B^{r}_{p}\left(t\right)\omega\in \Omega^{r}\left(\mathbb{R}^{+}\times \mathbb{R}^{n}\right)$. 

Now, to see $B_{p}^{r}\left(t,x,y\right)=e^{-t\Delta_{f,p}^{\left(r\right)}}\left(x,y\right)$, we need to show that for each $t>0$, 
\begin{equation}\label{Mapping: range B(t) in domain}
\begin{aligned}
B_{p}^{r}\left(t\right):\Omega^{r}_{c}\left(\mathbb{R}^{n}\right)\to \dom \Delta_{f,p}^{\left(r\right)} := \left\{ \omega\in L^{2}_{r}\left(\mathbb{R}^{n}\right): \Delta_{f,p}^{\left(r\right)}\omega \in L^{2}_{r}\left(\mathbb{R}^{n}\right)\right\}
\end{aligned}
\end{equation}
and that
\begin{equation}\label{Eq: B^r (t) satisfies h.e}
\begin{aligned}
\left\{ \begin{array}{l}
\frac{\partial}{\partial t} B^{r}_{p}\left(t\right)\omega+ \Delta_{f,p}^{\left(r\right)} B^{r}_{p}\left(t\right)\omega = 0 \\
\lim_{t\to 0^{+}}\left\| B_{p}^{r}\left(t\right)\omega - \omega\right\|_{L^{2}\left(\mathbb{R}^{n}\right)} = 0
\end{array}\right.  
\end{aligned}
\end{equation}
for each $\omega\in \Omega^{r}_{c}\left(\mathbb{R}^{n}\right)$.

To see \eqref{Mapping: range B(t) in domain}, it suffices to show for each $m\in \mathbb{N}\cup\left\{0\right\}$, 
\begin{equation*}
\begin{aligned}
\left\| \left( \Delta_{f,p}^{\left(r\right)} \right)^{m} B_{p}^{r}\left(t\right) \omega\right\|_{L^{2}\left(\mathbb{R}^{n}\right)} < \infty
\end{aligned}
\end{equation*}
for each $\omega\in \Omega^{r}_{c}\left(\mathbb{R}^{n}\right)$.
Observe that, by Fatou's lemma, we obtain
\begin{equation*}
\begin{aligned}
\left\|  \left( \Delta_{f,p}^{\left(r\right)} \right)^{m} B_{p}^{r}\left(t\right) \omega\right\|_{L^{2}\left(\mathbb{R}^{n}\right)} 
\leq \liminf_{R\to \infty} \left\| \left( \Delta_{f,p}^{\left(r\right)} \right)^{m} B_{p}^{r}\left(t\right) \omega\right\|_{L^{2}\left(\overline{B_{R}\left(0\right)}\right)}.
\end{aligned}
\end{equation*}
Now, choose two cut-off function $\chi,\tilde{\chi}\in \mathcal{C}^{\infty}_{c}\left(B_{1}\left(0\right)\right)$ such that $\tilde{\chi}=1$ in $\supp \chi$,  and for each $k$, put
\begin{equation*}
\begin{aligned}
\chi_{k}\left(x\right):=\chi\left(\frac{x}{k^{\varepsilon}}\right)\in \mathcal{C}^{\infty}_{c}\left(B_{k^{\varepsilon}}\left(0\right)\right), \tilde{\chi}_{k} := \tilde{\chi}\left(\frac{x}{k^{\varepsilon}}\right) \in \mathcal{C}^{\infty}_{c}\left(B_{k^{\varepsilon}}\left(0\right)\right).
\end{aligned}
\end{equation*}
Then by Theorem \ref{Spectral Theorem}, we derive
\begin{equation*}
\begin{aligned}
 \left\| \left( \Delta_{f,p}^{\left(r\right)} \right)^{m} B_{p}^{r}\left(t\right) \omega\right\|_{L^{2}\left(\overline{B_{R}\left(0\right)}\right)}
 \leq \lim_{k\to \infty} \left\| \tilde{\chi}_{n_{k}} \left( \Delta_{f,p}^{\left(r\right)} \right)^{m} A_{\left(n_{k}\right),p}^{r}\left(t\right) \chi_{n_{k}}\omega\right\|_{L^{2}\left(\mathbb{R}^{n}\right)}
 \leq C\left(t\right) \left\| \omega \right\|_{L^{2}\left(\mathbb{R}^{n}\right)},
\end{aligned}
\end{equation*}
where $C\left(t\right)$ is a constant that depends smoothly on $t$ but is independent of $R, k$. Hence, we conclude
\begin{equation*}
\begin{aligned}
\left\|  \left( \Delta_{f,p}^{\left(r\right)} \right)^{m} B_{p}^{r}\left(t\right) \omega\right\|_{L^{2}\left(\mathbb{R}^{n}\right)} 
\leq C\left(t\right)\left\| \omega\right\|_{L^{2}\left(\mathbb{R}^{n}\right)} <\infty
\end{aligned}
\end{equation*}
and \eqref{Mapping: range B(t) in domain} is now established.

To see \eqref{Eq: B^r (t) satisfies h.e}, let $\omega\in \Omega^{r}_{c}\left(\mathbb{R}^{n}\right)$. For each $\omega\in \Omega^{r}_{c}\left(\mathbb{R}^{n}\right)$, by Theorem \ref{Main Thm 1: locally uniform boundedness of s.h.}, we see
\begin{equation*}
\begin{aligned}
\left( \frac{\partial}{\partial t}B^{r}_{p}\left(t\right)\right)\omega
&=\int_{\mathbb{R}^{n}} \frac{\partial}{\partial t} B_{p}^{r}\left(t,x,y\right)\omega\left(y\right) dy
=\lim_{k\to \infty} \int_{\mathbb{R}^{n}} \frac{\partial}{\partial t} A_{\left(n_{k}\right),p}^{r}\left(t,x,y\right)\omega\left(y\right) dy\\
&=-\lim_{k\to\infty} \int_{\mathbb{R}^{n}} \Delta_{f,p}^{\left(r\right)} A_{\left(n_{k}\right),p}^{r}\left(t,x,y\right) \omega \left(y\right) dy\\
&=-\int_{\mathbb{R}^{n}}\Delta_{f,p}^{\left(r\right)}B_{p}^{r}\left(t,x,y\right)\omega\left(y\right) dy
=-\Delta_{f,p}^{\left(r\right)}B_{p}^{r}\left(t\right)\omega
\end{aligned}
\end{equation*}
for each $\left(t,x\right)\in K$, where $K$ is a compact subset in each compact subset of $\mathbb{R}^{+}\times \mathbb{R}^{n}$. This shows $\frac{\partial}{\partial t} B^{r}_{p}\left(t\right)\omega+ \Delta_{f,p}^{\left(r\right)} B^{r}_{p}\left(t\right)\omega = 0$.

To see $\lim_{t\to 0^{+}}\left\| B_{p}^{r}\left(t\right)\omega-\omega\right\|_{L^{2}\left(\mathbb{R}^{n}\right)}=0$, by Fatou's lemma, we obtain for each $t>0$, 
\begin{equation*}
\begin{aligned}
\left\| B_{p}^{r}\left(t\right)\omega-\omega\right\|_{L^{2}\left(\mathbb{R}^{n}\right)}
\leq \liminf_{R\to \infty} \left\| B_{p}^{r}\left(t\right)\omega-\omega\right\|_{L^{2}\left(\overline{B_{R}\left(0\right)}\right)}.
\end{aligned}
\end{equation*}
Again, let $\chi_{k}, \tilde{\chi}_{k}\in \mathcal{C}^{\infty}_{c}\left(B_{k^{\varepsilon}}\left(0\right)\right)$ as previously given. By Theorem \ref{Spectral Theorem} and the mean value theorem, we derive
\begin{equation*}
\begin{aligned}
\left\| B_{p}^{r}\left(t\right)\omega-\omega\right\|_{L^{2}\left(\overline{B_{R}\left(0\right)}\right)}^{2} 
&\leq \lim_{k\to \infty} \left\| \tilde{\chi}_{n_{k}}\left( A_{\left(n_{k}\right),p}^{r}\left(t\right) \chi_{n_{k}}\omega - \chi_{n_{k}\omega} \right) \right\|_{L^{2}\left(\mathbb{R}^{n}\right)}^{2}\\
&\leq \lim_{k\to \infty} \left(n_{k}\right)^{\frac{n}{2}} \left\|  e^{-\frac{t}{n_{k}}\Delta_{n_{k}}^{\left(r\right)}}\left(\chi_{n_{k}}\omega\right)_{\left[n_{k}\right]}- \left(\chi_{n_{k}}\omega\right)_{\left[n_{k}\right]} \right\|_{L^{2}\left(M\right)}^{2}\\
&=\lim_{k\to \infty} \left(n_{k}\right)^{\frac{n}{2}} \int_{\mathcal{S}\times \mathbb{N}} \left(e^{-\frac{t}{n_{k}}s}-1\right)^{2} \left|g\right|^{2} d\mu\\
&=\lim_{k\to \infty} \left(n_{k}\right)^{\frac{n}{2}} \left(\frac{t}{n_{k}}\right)^{2} \int_{\mathcal{S}\times \mathbb{N}} e^{-\frac{t_{k}}{n_{k}}s} \left|sg\right|^{2} d\mu \\
&\leq \lim_{k\to \infty} \left(n_{k}\right)^{\frac{n}{2}} t^{2} \left\| \frac{1}{n_{k}}\Delta_{n_{k}}^{\left(r\right)} \left( \chi_{n_{k}}\omega\right)_{\left[n_{k}\right]} \right\|^{2}_{L^{2}\left(M\right)}\\
&=t^{2} \left\| \Delta_{f,p}^{\left(r\right)} \omega \right\|_{L^{2}\left(\mathbb{R}^{n}\right)}^{2},
\end{aligned}
\end{equation*}
where $t_{k}\in \left(0,t\right)$ and $g\in L^{2}\left(S\times \mathbb{N}\right)$ is identified with $\left(\chi_{n_{k}}\omega\right)_{\left[n_{k}\right]}$ according to Theorem \ref{Spectral Theorem}.

This implies for each $\omega\in \Omega^{r}_{c}\left(\mathbb{R}^{n}\right)$,
\begin{equation*}
\begin{aligned}
\left\| B_{p}^{r}\left(t\right)\omega-\omega\right\|_{L^{2}\left(\mathbb{R}^{n}\right)} \leq t \left\| \Delta_{f,p}^{\left(r\right)}\omega\right\|_{L^{2}\left(\mathbb{R}^{n}\right)} \to 0
\end{aligned}
\end{equation*}
as $t\to 0^{+}$.

Finally, to see
\begin{equation*}
\begin{aligned}
B_{p}^{r}\left(t,x,y\right) = e^{-t\Delta_{f,p}^{\left(r\right)}}\left(x,y\right)
\end{aligned}
\end{equation*}
in $\mathbb{R}^{+}\times \mathbb{R}^{n}\times \mathbb{R}^{n}$,
it suffices to show
\begin{equation*}
\begin{aligned}
B_{p}^{r}\left(t\right)=e^{-t\Delta_{f,p}^{\left(r\right)}}
\end{aligned}
\end{equation*} 
in $\Omega^{r}_{c}\left(\mathbb{R}^{n}\right)$. To see it, we can first observe that
\begin{equation*}
\begin{aligned}
\lim_{h\to 0}\left\| \left( \frac{B^{r}\left(t+h\right)-B^{r}\left(t\right)}{h} - \frac{\partial}{\partial t}B^{r}\left(t\right)\right)\omega\right\|_{L^{2}\left(\mathbb{R}^{n}\right)} =0
\end{aligned}
\end{equation*}
from the following estimate obtained by Theorem \ref{Spectral Theorem}:
\begin{equation*}
\begin{aligned}
\left\| \left( \frac{B^{r}_{p}\left(t+h\right)-B^{r}_{p}\left(t\right)}{h} - \frac{\partial}{\partial t}B^{r}\left(t\right)\right)\omega\right\|_{L^{2}\left(\mathbb{R}^{n}\right)}  \leq h \left\|\left( \Delta_{f,p}^{\left(r\right)}\right)^{2}\omega \right\|_{L^{2}\left(\mathbb{R}^{n}\right)}
\end{aligned}
\end{equation*}
for each $\omega\in \Omega^{r}_{c}\left(\mathbb{R}^{n}\right)$. This shows
\begin{equation*}
\begin{aligned}
\frac{d}{d t} \left(B^{r}\left(t\right) \omega \bigg| \eta\right) = \left( \frac{\partial}{\partial t}B^{r}\left(t\right)\omega \bigg| \eta\right)
\end{aligned}
\end{equation*}
for each $\omega \in \Omega^{r}_{c}\left(\mathbb{R}^{n}\right)$ and for each $\eta\in L^{2}_{r}\left(\mathbb{R}^{n}\right)$. Hence, by the fundamental theorem of Calculus and the fact that $B^{r}\left(t\right)\omega \in \dom \Delta_{f,p}^{\left(r\right)}$ for each $\omega\in \Omega^{r}_{c}\left(\mathbb{R}^{n}\right)$, we derive
\begin{equation*}
\begin{aligned}
&\left( B^{r}_{p}\left(t\right)\omega \bigg|\eta\right) - \left( \omega \bigg| e^{-t\Delta_{f,p}^{\left(r\right)}}\eta\right)\\
&=\lim_{q\to 0^{+}} \int_{q}^{t} \frac{d}{ds} 
\left(B_{p}^{r}\left(s\right)\omega\bigg| e^{-\left(t+q-s\right)\Delta_{f,p}^{\left(r\right)}}\eta\right)  ds \\
&=\lim_{q\to 0^{+}} \int_{q}^{t} \left( \left( \frac{\partial}{\partial 
s}B_{p}^{r}\left(s\right)\right)\omega \bigg| e^{-\left(t+q-s\right)\Delta_{f,p}^{\left(r\right)}}\eta\right) -\left( B_{p}^{r}\left(s\right)\omega \bigg| \left(\frac{\partial}{\partial s}e^{-\left(t+q-s\right)\Delta_{f,p}^{\left(r\right)}}\right)\eta \right) ds\\
&=\lim_{q\to 0^{+}} \int_{q}^{t} \left( -\Delta_{f,p}^{\left(r\right)}B_{p}^{r}\left(s\right) \omega \bigg| e^{-\left(t+q-s\right)\Delta_{f,p}^{\left(r\right)}}\eta\right) +\left( B_{p}^{r}\left(s\right)\omega \bigg| \Delta_{f,p}^{\left(r\right)}e^{-\left(t+q-s\right)\Delta_{f,p}^{\left(r\right)}}\eta \right) ds\\
&=\lim_{q\to 0^{+}} \int_{q}^{t} \left( -\Delta_{f,p}^{\left(r\right)}B_{p}^{r}\left(s\right) \omega \bigg| e^{-\left(t+q-s\right)\Delta_{f,p}^{\left(r\right)}}\eta\right) +\left(\Delta_{f,p}^{\left(r\right)}B_{p}^{r}\left(s\right)\omega \bigg| e^{-\left(t+q-s\right)\Delta_{f,p}^{\left(r\right)}}\eta \right) ds\\
&=0
\end{aligned}
\end{equation*}
for any two $\omega,\eta\in \Omega^{r}_{c}\left(\mathbb{R}^{n}\right)$.
This shows
\begin{equation*}
\begin{aligned}
\left(B_{p}^{r}\left(t\right)\omega \bigg| \eta\right) = \left( \omega \bigg| e^{-t\Delta_{f,p}^{\left(r\right)}}\eta\right) = \left( e^{-t\Delta_{f,p}^{\left(r\right)}}\omega \bigg| \eta\right)
\end{aligned}
\end{equation*}
for any two $\omega,\eta\in \Omega^{r}_{c}\left(\mathbb{R}^{n}\right)$, implying $B_{p}^{r}\left(t\right)=e^{-t\Delta_{f,p}^{\left(r\right)}}$ in $\Omega^{r}_{c}\left(\mathbb{R}^{n}\right)$ and thus $B_{p}^{r}\left(t,x,y\right) = e^{-t\Delta_{f,p}^{\left(r\right)}}\left(x,y\right)$ in $\mathbb{R}^{+}\times \mathbb{R}^{n}\times \mathbb{R}^{n}$.  

Note that the previous argument implies that every convergent subseqeunce of the sequence of the scaled heat kernels $\left\{A_{\left(k\right),p}^{\left(r\right)}\left(t,x,y\right)\right\}_{k}$ converges to the same limit, so together with Theorem \ref{Main Thm 1: locally uniform boundedness of s.h.}, we can further conclude this sequence in effect converges to that limit. Therefore, we have established Theorem \ref{Main Thm: Semi-classical Heat Kernel Asymptotics}.

\section{Heat Kernel Asymptotics away from Critical Points}\label{H.K.A.O.C.P.}

In this section, we give proofs of Theorem \ref{Main Thm: Asymptotic Behavior of Scaled Heat Kernel in between} and Theorem \ref{Main Thm: Asymptotic Behavior of Heat Kernel outside Critical Points}. 

The key point of obtaining the two theorems is based on the remark that we can retrieve pointwise bound for the heat kernels in question by establishing their corresponding mapping properties (see Theorem \ref{Thm: Mapping property outside criti pts 2} and Theorem \ref{Thm: Mapping propety outside criti pts 1}). These mapping properties are established based on a Bochner-type estimate related to differential forms with support away from the critical points (see Lemma \ref{Lem: Bochner type estimate (Pointwise)} and  Lemma \ref{Lem: Bochner-type Estimate}).

Throughout this section,  for each $p\in \crit \left(f\right)$, we additionally identify the coordinate neighborhood $U_{p}$ of $p$ with the Euclidean ball $B_{\frac{3}{2}}\left(0\right)$ and $p$ with $0$, under the coordinate chart $\varphi_{p}$. For each large $k>0$, put
\begin{equation*}
\begin{aligned}
\mathcal{U}^{k} = \bigcup_{p\in \crit\left(f\right)} U_{p}^{k},
\end{aligned}
\end{equation*}
where $U_{p}^{k}$ is identified with the Euclidean ball $B_{k^{-\frac{1}{2}+\varepsilon}}\left(0\right)$, $\varepsilon\in \left(0,\frac{1}{2}\right)$ under the coordinate chart $\varphi_{p}$, as in Subsection \ref{s.t.}.

\subsection{Proof of Theorem 1.2}

Choose $D>1$ and let $k$ large enough so that  $2D<k^{\varepsilon}$.  For each $p\in \crit \left(f\right)$, let $x\in B_{k^{\varepsilon}}\left(0\right) \setminus B_{2D}\left(0\right)$ and put
\begin{equation*}
\begin{aligned}
A_{x}^{k} := \varphi_{p}^{-1}\left( B_{k^{-\frac{1}{2}+\varepsilon}}\left(0\right) \setminus B_{\frac{\left|x\right|}{2}k^{-\frac{1}{2}}}\left(0\right)\right)\subset U_{p}^{k}
\end{aligned}
\end{equation*}

The choice of $D$ plays a role in the following Bochner type estimate:

\begin{lem}\label{Lem: Bochner type estimate (Pointwise)}
If $D$ large enough, let $x\in B_{k^{\varepsilon}}\left(0\right)\setminus B_{2D}\left(0\right)$ and we have
\begin{equation}\label{Ineq: Adapted Bochner-type formula}
\begin{aligned}
\left( \Delta_{k}^{\left(r\right)}\omega | \omega\right) \geq Ck\left|x\right|^{2}\left\| \omega\right\|^{2}
\end{aligned}
\end{equation}
for each $\omega\in \Omega^{r}\left(M\right)$ with $\supp \omega \subset A_{x}^{k}$ and  for large $k$, where $C$ is independent of $D,x,k$.
\end{lem}
\begin{proof}
Recall that we can write the Witten Laplacian $\Delta_{k}^{\left(r\right)}$ as 
\begin{equation*}
\begin{aligned}
\Delta_{k}^{\left(r\right)} 
= \Delta^{\left(r\right)} + k^{2}\left|df\right|^{2}+ k\left(\mathcal{L}_{\nabla f} + \mathcal{L}_{\nabla f}^{*}\right).
\end{aligned}
\end{equation*}
Put $A=\mathcal{L}_{\nabla f} + \mathcal{L}_{\nabla f}^{*}$.

Note that 
\begin{equation*}
\begin{aligned}
\left|df\right|^{2}\geq \frac{1}{4k} \left|x\right|^{2} 
\end{aligned}
\end{equation*}
in $A_{x}^{k}$ (with respect to $\left(U_{p},\varphi_{p}\right)$). Moreover, by the local expression \eqref{Eq: Local expression of Witten Laplacian} and partition of unity, there exists $m\in \mathbb{R}$ such that $\left(A\eta|\eta\right)\geq m\left\|\eta\right\|^{2}$ for each $\eta\in \Omega^{r}\left(M\right)$. For this $m$, let $D$ large enough such that 
\begin{equation*}
\begin{aligned}
\left| \frac{m}{D^{2}} \right| <\frac{1}{8}.
\end{aligned}
\end{equation*}

Now, for each $\omega\in \Omega^{r}\left(M\right)$ with $\supp \omega \subset A_{x}^{k}$, we deduce
\begin{equation*}
\begin{aligned}
\left(\Delta_{k}^{\left(r\right)} \omega |\omega\right) \geq k\left|x\right|^{2}\left(\frac{1}{4}+\frac{m}{\left|x\right|^{2}}\right) \left\|\omega\right\|^{2}
\geq Ck\left|x\right|^{2} \left\|\omega\right\|^{2}
\end{aligned}
\end{equation*}
with $C=\frac{1}{4}$. 
\end{proof}

Let $D>1$ be large enough such that Lemma \ref{Lem: Bochner type estimate (Pointwise)} holds and let $k>0$ large enough such that $2D<k^{\varepsilon}$. 
Given $p\in \crit\left(f\right)$, for each $x\in B_{k^{\varepsilon}}\left(0\right)\setminus B_{2D}\left(0\right)$, let $\chi^{\left[1\right]}\in \mathcal{C}^{\infty}_{c}\left(U_{p}\right)$ be a cut-off function such that $\chi^{\left[1\right]}=1$ in $\varphi_{p}^{-1}\left( B_{\frac{1}{2}}\left(0\right)\right)$  and  $\chi^{\left[1\right]}=0$ in $M\setminus \varphi_{p}^{-1} \left( B_{\frac{3}{4}}\left(0\right)\right)$. Put $\chi^{\left[1\right]}_{x,k}\in \mathcal{C}^{\infty}_{c}\left(U_{p}^{k}\right)$ such that
\begin{equation*}
\begin{aligned}
\chi^{\left[1\right]}_{x,k} \circ \varphi_{p}^{-1}\left(q\right) = \chi^{\left[1\right]}\left(\left|x\right|^{-1}k^{\frac{1}{2}} q \right)
\end{aligned}
\end{equation*}
for each $q\in \varphi_{p}\left(U_{p}\right)\subset \mathbb{R}^{n}$. Similarly, let $\chi^{\left[2\right]}\in \mathcal{C}^{\infty}_{c}\left(U_{p}\right)$ be a cut-off function such that $\chi^{\left[2\right]}=1$ in $\varphi_{p}^{-1}\left(B_{\frac{5}{4}}\left(0\right)\right)$ and $\chi^{\left[2\right]}=0$ in $\varphi_{p}^{-1}\left( B_{\frac{3}{2}}\left(0\right)\right)$. Put $\chi^{\left[2\right]}_{x,k}\in \mathcal{C}_{c}^{\infty}\left(U_{p}^{k}\right)$ such that
\begin{equation*}
\begin{aligned}
\chi^{\left[2\right]}_{x,k} \circ \varphi_{p}^{-1} \left(q\right) = \chi^{\left[2\right]}\left( \left|x\right|^{-1}k^{\frac{1}{2}}  q \right)
\end{aligned}
\end{equation*}
for each  $q\in \varphi_{p}\left(U_{p}\right)\subset \mathbb{R}^{n}$. Finally, set $\chi_{x,k}\in \mathcal{C}^{\infty}_{c}\left(A_{x}^{k}\right)$ by
\begin{equation}\label{Eq: Defi of chi_x,k}
\begin{aligned}
\chi_{x,k} = \chi_{x,k}^{\left[2\right]} - \chi_{x,k}^{\left[1\right]}.
\end{aligned}
\end{equation}
In particular,  $\chi_{x,k}=1$ in $\varphi_{p}^{-1}\left(B_{\frac{5\left|x\right|}{4}k^{-\frac{1}{2}}}\left(0\right) \setminus B_{\frac{3\left|x\right|}{4}k^{-\frac{1}{2}}}\left(0\right)\right)$. Note that
\begin{equation}\label{Ineq: Supnorm of chi_x,k}
\begin{aligned}
\sup_{\mathbb{R}^{n}} \left| D^{\alpha}\chi_{x,k}^{\left[1\right]} \right| \leq C\left(\chi^{\left[1\right]},\chi^{\left[2\right]}\right) \left|x\right|^{-\left|\alpha\right|} k^{\frac{\left|\alpha\right|}{2}},
\end{aligned}
\end{equation}
where $C\left(\chi^{\left[1\right]},\chi^{\left[2\right]}\right)$ depends on $\chi^{\left[1\right]}$ and $\chi^{\left[2\right]}$, with respect to the coordinates given by $\left(U_{p}, \varphi_{p}\right)$. The construction of $\chi_{x,k}$ can be given from the same cut-off functions $\chi^{\left[1\right]}$ and $\chi^{\left[2\right]}$ as long as $x\in B_{k^{\varepsilon}}\left(0\right)\setminus B_{2D}\left(0\right)$ and $k>0$,  in which circumstance, the constant $C\left(\chi^{\left[1\right]},\chi^{\left[2\right]}\right)$ in \eqref{Ineq: Supnorm of chi_x,k} does not rely on $x$ and $k$.

With Lemma \ref{Lem: Bochner type estimate (Pointwise)} at hand, we can obtain the following $L^{2}$-estimate.

\begin{theo}\label{Thm: L^2 estimate bdd by |x|^-N w/o diff.}
Let $D>1$ be large enough such that Lemma \ref{Lem: Bochner type estimate (Pointwise)} holds, and for large $k>0$ and for each $x\in B_{k^{\varepsilon}}\left(0\right)\setminus B_{2D}\left(0\right)$, let $\chi_{x,k}$ be the cut-off function given by the cut-off function $\chi^{\left[1\right]},\chi^{\left[2\right]}$ as in \eqref{Eq: Defi of chi_x,k}. Then  for each $N\in \mathbb{N}$, 
\begin{equation}\label{Ineq: L^2 estimate bdd by |x|^-N w/o diff.}
\left\| \chi_{x,k}e^{-\frac{t}{k}\Delta_{k}^{\left(r\right)}}\omega\right\| \leq C\left(\chi^{\left[1\right]},\chi^{\left[2\right]}, t,N\right) \left|x\right|^{-N}\left\| \omega\right\|,
\end{equation}
for each $\omega\in \Omega^{r}\left(M\right)$, where $C\left(\chi^{\left[1\right]},\chi^{\left[2\right]}, t,N\right)$ depends on $\chi^{\left[1\right]},\chi^{\left[2\right]}$, $N$ and smoothly on $t$ but is independent of $x$ and $k$.
\end{theo}
\begin{proof}

First, we show \eqref{Ineq: L^2 estimate bdd by |x|^-N w/o diff.} holds for $N=1$. For higher positive integers $N$, we can achieve similarly.

Since $\chi_{x,k}e^{-\frac{t}{k}\Delta_{k}^{\left(r\right)}}\omega$ has its support in $A_{x}^{k}$, by Lemma \ref{Lem: Bochner type estimate (Pointwise)}, we obtain
\begin{equation}\label{Ineq: First split using Bochner}
\begin{aligned}
\left\| \chi_{x,k}e^{-\frac{t}{k}\Delta_{k}^{\left(r\right)}}\omega\right\|^{2}
&\leq C_{1}k^{-1}\left|x\right|^{-2} \left( \Delta_{k}^{\left(r\right)}\chi_{x,k}e^{-\frac{t}{k}\Delta_{k}^{\left(r\right)}}\omega \bigg| \chi_{x,k}e^{-\frac{t}{k}\Delta_{k}^{\left(r\right)}}\omega\right)\\
&=C_{1}k^{-1}\left|x\right|^{-2}\left( \left\| d_{k}\left(\chi_{x,k}e^{-\frac{t}{k}\Delta_{k}^{\left(r\right)}}\omega\right)\right\|^{2} + \left\| d_{k}^{*} \left(\chi_{x,k}e^{-\frac{t}{k}\Delta_{k}^{\left(r\right)}}\omega\right) \right\|^{2}\right).
\end{aligned}
\end{equation}
Hence, the $L^{2}$-norm estimate of $\chi_{x,k}e^{-\frac{t}{k}\Delta_{k}^{\left(r\right)}}\omega$ is determined by the two $L^{2}$-norms in the right hand side \eqref{Ineq: First split using Bochner}. Before we proceed, note that by direct computation, we see the following analogous Leibniz rules: 
\begin{equation}\label{Eq: d_k chi_k eta}
\begin{aligned}
d_{k}\left( \chi_{x,k} \eta \right)
= d\chi_{x,k} \wedge \eta + \chi_{x,k}d_{k}\eta
\end{aligned}
\end{equation}
and that
\begin{equation}\label{Eq: d_k^* chi_k eta}
\begin{aligned}
d_{k}^{*}\left(\chi_{x,k}\eta\right)
=\left(-1\right)^{n\left(r+1\right)+1} \ast\left( d\chi_{x,k}\wedge \ast \eta\right) + \chi_{x,k}d_{k}^{*}\eta
\end{aligned}
\end{equation}
for each $\eta\in \Omega^{r}\left(M\right)$. 

To deal with the first term on the right hand side of \eqref{Ineq: First split using Bochner}, by \eqref{Eq: d_k chi_k eta} with $\eta=e^{-\frac{t}{k}\Delta_{k}^{\left(r\right)}}\omega$, we see that
\begin{equation}\label{Ineq: First order d_k from first split}
\begin{aligned}
\left\| d_{k}\left( \chi_{x,k} e^{-\frac{t}{k}\Delta_{k}^{\left(r\right)}}\omega\right) \right\|^{2} 
\leq 4\left( \left\| d\chi_{x,k}\wedge e^{-\frac{t}{k}\Delta_{k}^{\left(r\right)}}\omega\right\|^{2}+ \left\| \chi_{x,k}d_{k}e^{-\frac{t}{k}\Delta_{k}^{\left(r\right)}}\omega\right\|^{2}\right).
\end{aligned}
\end{equation}
Note that $\chi_{x,k}\in \mathcal{C}^{\infty}_{c}\left(U_{p}^{k}\right)$ and has the local expression $d\chi_{x,k} = \sum_{i=1}^{n} \frac{\partial \chi_{x,k}}{\partial y^{i}} dy^{i}$
in $\left(U_{p}^{k},\varphi_{p}\right)$. Hence, by the local nature of $\chi_{x,k}$ \eqref{Ineq: Supnorm of chi_x,k} and Theorem \ref{Spectral Theorem}, we obtain
\begin{equation}\label{Ineq: First order term w/ chi_x,k diff}
\begin{aligned}
\left\| d\chi_{x,k} \wedge e^{-\frac{t}{k}\Delta_{k}^{\left(r\right)}}\omega \right\|^{2}
&\leq C_{2}\left(\chi^{\left[1\right]},\chi^{\left[2\right]}\right) k\left|x\right|^{-2}\left\| e^{-\frac{t}{k}\Delta_{k}^{\left(r\right)}}\omega\right\|^{2} \\
&\leq  C_{3}\left(\chi^{\left[1\right]},\chi^{\left[2\right]},t\right) k\left|x\right|^{-2}\left\| \omega\right\|^{2},
\end{aligned}
\end{equation}
where $C_{3}\left(\chi^{\left[1\right]},\chi^{\left[2\right]},t\right)$ depends on $\chi^{\left[1\right]},\chi^{\left[2\right]}$ and smoothly on $t$ but independent of $D,x,k$. 
Moreover, by Theorem \ref{Spectral Theorem}, we obtain
\begin{equation}\label{Ineq: First order term w/ chi_x,k not diff.}
\begin{aligned}
\left\| \chi_{x,k}d_{k}e^{-\frac{t}{k}\Delta_{k}^{\left(r\right)}}\omega\right\|^{2}
\leq \left\| d_{k}e^{-\frac{t}{k}\Delta_{k}^{\left(r\right)}}\omega\right\|^{2}
\leq \left(\Delta_{k}^{\left(r\right)} e^{-\frac{t}{k}\Delta_{k}^{\left(r\right)}}\omega\bigg| e^{-\frac{t}{k}\Delta_{k}^{\left(r\right)}}\omega\right)
\leq C_{4}\left(t\right)k\left\|\omega\right\|^{2}.
\end{aligned}
\end{equation}
Hence, using \eqref{Ineq: First order d_k from first split}, \eqref{Ineq: First order term w/ chi_x,k diff} along with the fact that $\left|x\right|^{-2}<1$, and \eqref{Ineq: First order term w/ chi_x,k not diff.}, we conclude
\begin{equation}\label{Ineq: First term estimate from first split}
\begin{aligned}
\left\| d_{k}\left( \chi_{x,k}e^{-\frac{t}{k}\Delta_{k}^{\left(r\right)}}\omega\right) \right\|^{2} \leq C_{5}\left(\chi^{\left[1\right]},\chi^{\left[2\right]},t\right)k\left\|\omega\right\|^{2},
\end{aligned}
\end{equation}
where $C_{5}\left(\chi^{\left[1\right]},\chi^{\left[2\right]},t\right)$ depends on $\chi^{\left[1\right]},\chi^{\left[2\right]}$ and smoothly on $t$ but independent of $D,x,k$.

We can achieve an upper bound for the second term in the similar fashion and let us briefly go through the deduction. By \eqref{Eq: d_k^* chi_k eta} with $\eta=e^{-\frac{t}{k}\Delta_{k}^{\left(r\right)}}\omega$, we see
\begin{equation}\label{Ineq: First order d_k^* from first split}
\begin{aligned}
\left\| d_{k}^{*}\left( \chi_{x,k} e^{-\frac{t}{k}\Delta_{k}^{\left(r\right)}}\omega\right) \right\|^{2} \leq 4\left( \left\| \ast \left( d\chi_{x,k}\wedge\ast e^{-\frac{t}{k}\Delta_{k}^{\left(r\right)}}\omega\right) \right\|^{2}+ \left\| \chi_{x,k}d_{k}^{*}e^{-\frac{t}{k}\Delta_{k}^{\left(r\right)}}\omega\right\|^{2}\right);
\end{aligned}
\end{equation}
by doing so, we break down this $L^{2}$-norm term into two terms, one is with $\chi_{x,k}$ differentiated and the other without. By local nature of $\chi_{x,k}$ along with Theorem \ref{Spectral Theorem}, we can obtain
\begin{equation}\label{Ineq: ast d chi_x,k wedge estimate}
\begin{aligned}
\left\| \ast \left( d\chi_{x,k}\wedge\ast e^{-\frac{t}{k}\Delta_{k}^{\left(r\right)}}\omega\right) \right\|^{2} \leq C_{6}\left(\chi^{\left[1\right]},\chi^{\left[2\right]},t\right)k\left|x\right|^{-2} \left\|\omega\right\|^{2}.
\end{aligned}
\end{equation}
By Theorem \ref{Spectral Theorem}, we can see
\begin{equation}\label{Ineq: chi_x,k d_k^* estimate}
\begin{aligned}
\left\| \chi_{x,k} d_{k}^{*} e^{-\frac{t}{k}\Delta_{k}^{\left(r\right)}}\omega\right\|^{2} \leq C_{7}\left(\chi^{\left[1\right]},\chi^{\left[2\right]},t\right)k\left\|\omega\right\|^{2}.
\end{aligned}
\end{equation}
Finally, by \eqref{Ineq: First order d_k^* from first split}, \eqref{Ineq: ast d chi_x,k wedge estimate}, and \eqref{Ineq: chi_x,k d_k^* estimate}, we can conclude
\begin{equation}\label{Ineq: Second term estimate from first split}
\begin{aligned}
\left\| d_{k}^{*}\left( \chi_{x,k} e^{-\frac{t}{k}\Delta_{k}^{\left(r\right)}}\omega\right) \right\|^{2}  \leq C_{8}\left(\chi^{\left[1\right]},\chi^{\left[2\right]},t\right)k \left\|\omega\right\|^{2},
\end{aligned}
\end{equation}
where $C_{8}\left(\chi^{\left[1\right]},\chi^{\left[2\right]},t\right)$ depends on $\chi^{\left[1\right]},\chi^{\left[2\right]}$ and smoothly on $t$ but independent of $D,x,k$.

Subsequently, by \eqref{Ineq: First split using Bochner}, \eqref{Ineq: First term estimate from first split}, and \eqref{Ineq: Second term estimate from first split}, we conclude
\begin{equation*}
\begin{aligned}
\left\| \chi_{x,k} e^{-\frac{t}{k}\Delta_{k}^{\left(r\right)}}\omega \right\|^{2}
\leq C_{9}\left(\chi^{\left[1\right]},\chi^{\left[2\right]},t\right)\left|x\right|^{-2} \left\|\omega\right\|^{2},
\end{aligned}
\end{equation*}
where $C_{9}\left(\chi^{\left[1\right]},\chi^{\left[2\right]},t\right)$ depends on $\chi^{\left[1\right]},\chi^{\left[2\right]}$ and smoothly on $t$ but independent of $D,x,k$. Namely, we have established \eqref{Ineq: L^2 estimate bdd by |x|^-N w/o diff.} for $N=1$.

Similarly, we can obtain \eqref{Ineq: L^2 estimate bdd by |x|^-N w/o diff.} holds for each positive integer $N>1$. To explain, we divide it into two cases: $N$ is even and is odd.  If $N$ is even, then we can apply Lemma \ref{Lem: Bochner type estimate (Pointwise)} repeatedly until we obtain
\begin{equation*}
\begin{aligned}
&\left\| \chi_{x,k} e^{-\frac{t}{k}\Delta_{k}^{\left(r\right)}}\omega \right\|^{2}\\
&\leq C_{10}k^{-N}\left|x\right|^{-2N}
\left(  \left\| \left(d_{k}d_{k}^{*}\right)^{\frac{N}{2}} \left(\chi_{x,k} e^{-\frac{t}{k}\Delta_{k}^{\left(r\right)}}\omega\right) \right\|^{2} + \left\| \left(d_{k}^{*}d_{k}\right)^{ \frac{N}{2}} \left(\chi_{x,k} e^{-\frac{t}{k}\Delta_{k}^{\left(r\right)}}\omega\right) \right\|^{2} \right);
\end{aligned}
\end{equation*}
By the Leibniz rules \eqref{Eq: d_k chi_k eta}, \eqref{Eq: d_k^* chi_k eta},  local nature of $\chi_{x,k}$, and Theorem \ref{Spectral Theorem}, we can deduce
\begin{equation*}
\begin{aligned}
\left\| \chi_{x,k} e^{-\frac{t}{k}\Delta_{k}^{\left(r\right)}}\omega \right\|^{2}
&\leq C_{11}\left(\chi^{\left[1\right]},\chi^{\left[2\right]},t,N\right) k^{-N}\left|x\right|^{-2N} \cdot  k^{N} \left\|\omega\right\|^{2} \\
&= C_{11}\left(\chi^{\left[1\right]},\chi^{\left[2\right]},t,N\right) \left|x\right|^{-2N}\left\|\omega\right\|^{2},
\end{aligned}
\end{equation*}
where $C_{11}\left(\chi^{\left[1\right]},\chi^{\left[2\right]},t,N\right)$ depends on $N$ due to \eqref{Eq: d_k chi_k eta} and \eqref{Eq: d_k^* chi_k eta} but is independent of $D,x,k$.  

If $N$ is odd, we can obtain
\begin{equation*}
\begin{aligned}
&\left\| \chi_{x,k} e^{-\frac{t}{k}\Delta_{k}^{\left(r\right)}}\omega \right\|^{2}\\
&\leq C_{12}k^{-N}\left|x\right|^{-2N}
\left(  \left\| \left(d_{k}d_{k}^{*}\right)^{\frac{N-1}{2}}d_{k} \left(\chi_{x,k} e^{-\frac{t}{k}\Delta_{k}^{\left(r\right)}}\omega\right) \right\|^{2} + \left\| \left(d_{k}^{*}d_{k}\right)^{ \frac{N-1}{2}}d_{k}^{*} \left(\chi_{x,k} e^{-\frac{t}{k}\Delta_{k}^{\left(r\right)}}\omega\right) \right\|^{2} \right)
\end{aligned}
\end{equation*}
and so we can proceed to obtain
\begin{equation*}
\begin{aligned}
\left\| \chi_{x,k} e^{-\frac{t}{k}\Delta_{k}^{\left(r\right)}}\omega \right\|^{2}\leq C_{13}\left(\chi^{\left[1\right]},\chi^{\left[2\right]},t,N\right) \left|x\right|^{-2N}\left\|\omega\right\|^{2}.
\end{aligned}
\end{equation*}

Hence, we have established Theorem \ref{Thm: L^2 estimate bdd by |x|^-N w/o diff.}.
\end{proof}

More generally, using similar iterative argument, we can in effect obtain the following $L^{2}$-estimate involving the Witten Laplacian $\Delta_{k}^{\left(r\right)}$:
\begin{coro} \label{Cor: L^2 estimate bdd by |x|^-N}
Let $D>1$ be large enough such that Lemma \ref{Lem: Bochner type estimate (Pointwise)} holds, and for large $k>0$ and for each $x\in B_{k^{\varepsilon}}\left(0\right)\setminus B_{2D}\left(0\right)$, let $\chi_{x,k}$ be the cut-off function given by the cut-off function $\chi^{\left[1\right]},\chi^{\left[2\right]}$ as in \eqref{Eq: Defi of chi_x,k}. Then for each $m\in \mathbb{N}\cup\left\{0\right\}$ and for each $N\in \mathbb{N}$,
\begin{equation}\label{Ineq: Adapted L^2 estimate in between}
\left\| \left(\Delta_{k}^{\left(r\right)}\right)^{m} \left( \chi_{x,k}e^{-\frac{t}{k}\Delta_{k}^{\left(r\right)}}\omega\right) \right\| 
\leq C\left(\chi^{\left[1\right]},\chi^{\left[2\right]}, t,m,N\right) k^{m}\left|x\right|^{-N}\left\| \omega\right\|,
\end{equation}
for each $\omega\in \Omega^{r}\left(M\right)$, where $C\left(\chi^{\left[1\right]},\chi^{\left[2\right]}, t,m,N\right)$ depends on $\chi^{\left[1\right]},\chi^{\left[2\right]}, m, N$ and smoothly on $t$ but is independent of $x$ and $k$.
\end{coro}
\begin{proof}
Corollary \ref{Cor: L^2 estimate bdd by |x|^-N} can be obtained similarly to Theorem \ref{Thm: L^2 estimate bdd by |x|^-N w/o diff.}. 

Since $d_{k}^{2}=0=\left(d_{k}^{*}\right)^{2}$, we see that 
\begin{equation*}
\begin{aligned}
\left\| \left(\Delta_{k}^{\left(r\right)}\right)^{m} \left( \chi_{x,k}e^{-\frac{t}{k}\Delta_{k}^{\left(r\right)}}\omega\right) \right\|^{2} 
\leq 4\left(  \left\| \left( d_{k}d_{k}^{*}\right)^{m}  \left( \chi_{x,k}e^{-\frac{t}{k}\Delta_{k}^{\left(r\right)}}\omega\right)\right\|^{2}+  \left\| \left( d_{k}^{*}d_{k}\right)^{m}  \left( \chi_{x,k}e^{-\frac{t}{k}\Delta_{k}^{\left(r\right)}}\omega\right)\right\|^{2}\right).
\end{aligned}
\end{equation*}

As in the proof of Theorem \ref{Thm: L^2 estimate bdd by |x|^-N w/o diff.}, repeated use of Lemma \ref{Lem: Bochner type estimate (Pointwise)} gives 
\begin{equation*}
\begin{aligned}
&\left\| \left(\Delta_{k}^{\left(r\right)}\right)^{m} \left(\chi_{x,k} e^{-\frac{t}{k}\Delta_{k}^{\left(r\right)}}\omega\right) \right\|^{2}\\
&\leq C_{1}k^{-N}\left|x\right|^{-2N}
\left(  \left\| \left(d_{k}d_{k}^{*}\right)^{\frac{N}{2}+m} \left(\chi_{x,k} e^{-\frac{t}{k}\Delta_{k}^{\left(r\right)}}\omega\right) \right\|^{2} + \left\| \left(d_{k}^{*}d_{k}\right)^{ \frac{N}{2}+m} \left(\chi_{x,k} e^{-\frac{t}{k}\Delta_{k}^{\left(r\right)}}\omega\right) \right\|^{2} \right);
\end{aligned}
\end{equation*}
for each even positive integer $N$, and
\begin{equation*}
\begin{aligned}
&\left\| \left(\Delta_{k}^{\left(r\right)}\right)^{m} \left(\chi_{x,k} e^{-\frac{t}{k}\Delta_{k}^{\left(r\right)}}\omega\right) \right\|^{2}\\
&\leq C_{1}k^{-N}\left|x\right|^{-2N}
\left(  \left\| \left(d_{k}d_{k}^{*}\right)^{\frac{N-1}{2}+m}d_{k} \left(\chi_{x,k} e^{-\frac{t}{k}\Delta_{k}^{\left(r\right)}}\omega\right) \right\|^{2} + \left\| \left(d_{k}^{*}d_{k}\right)^{ \frac{N-1}{2}+m}d_{k}^{*} \left(\chi_{x,k} e^{-\frac{t}{k}\Delta_{k}^{\left(r\right)}}\omega\right) \right\|^{2} \right);
\end{aligned}
\end{equation*}
for each odd positive integer $N$. Hence, the Leibniz rules \eqref{Eq: d_k chi_k eta}, \eqref{Eq: d_k^* chi_k eta}, the local nature of $\chi_{x,k}$ (see \eqref{Ineq: Supnorm of chi_x,k}) and Theorem \ref{Spectral Theorem} allow us to establish Corollary \ref{Cor: L^2 estimate bdd by |x|^-N}. Note that the constant $C\left(\chi^{\left[1\right]},\chi^{\left[2\right]}, t,m,N\right)$ in \eqref{Ineq: Adapted L^2 estimate in between} depends on $m$ due to not only the Leibniz rules \eqref{Eq: d_k chi_k eta} and \eqref{Eq: d_k^* chi_k eta} but also Theorem \ref{Spectral Theorem}, and on $N$ due to again \eqref{Eq: d_k chi_k eta} and \eqref{Eq: d_k^* chi_k eta}.
\end{proof}

Finally, we can show the following mapping property:
\begin{theo} \label{Thm: Mapping property outside criti pts 2}
Let $D>1$ be large enough such that Lemma \ref{Lem: Bochner type estimate (Pointwise)} holds, and for large $k>0$ and for each $x\in B_{k^{\varepsilon}}\left(0\right)\setminus B_{2D}\left(0\right)$, let $\chi_{x,k}$ be the cut-off function given by the cut-off function $\chi^{\left[1\right]},\chi^{\left[2\right]}$ as in \eqref{Eq: Defi of chi_x,k}. Then for each $m\in \mathbb{N}\cup\left\{0\right\}$ and for each $N\in \mathbb{N}$, there exists $C\left(\chi^{\left[1\right]},\chi^{\left[2\right]}, t,m,N\right)>0$ such that
\begin{equation}\label{Ineq: improved version of mapping property}
\begin{aligned}
\left\| \left(\chi_{x,k}\right)_{\left[\frac{1}{k}\right]} A_{\left(k\right),p}^{r}\left(t\right)\left(\chi_{x,k}\right)_{\left[\frac{1}{k}\right]} \omega \right\|_{2m}\leq C\left(\chi^{\left[1\right]},\chi^{\left[2\right]}, t,m,N\right) \left|x\right|^{-N} \left\|\omega\right\|_{-2m},
\end{aligned}
\end{equation}
for each $\omega\in \Omega^{r}_{c}\left(B_{k^{\varepsilon}}\left(0\right)\right)$,
where $C\left(\chi^{\left[1\right]},\chi^{\left[2\right]}, t,m,N\right)$ depends on $\chi^{\left[1\right]},\chi^{\left[2\right]},m,N$ and smoothly on $t$ and is independent of $D$, $k$,$x$.
\end{theo}
\begin{proof}

Thanks to Gårding's inequality and \eqref{Scaling formula 1}, \eqref{Eq: Scaling formula for heat kernel}, we obtain
\begin{equation*}
\begin{aligned}
&\left\| \left(\chi_{x,k}\right)_{\left[\frac{1}{k}\right]} A_{\left(k\right),p}^{r}\left(t\right)\left(\chi_{x,k}\right)_{\left[\frac{1}{k}\right]} \omega \right\|_{2m}\\
&\leq C_{1}\left\| \left(\Delta_{f,p}^{\left(r\right)}\right)^{m}  \left(\chi_{x,k}\right)_{\left[\frac{1}{k}\right]} A_{\left(k\right),p}^{r}\left(t\right)\left(\chi_{x,k}\right)_{\left[\frac{1}{k}\right]} \omega \right\|_{0} 
+C_{2} \left\| \left(\chi_{x,k}\right)_{\left[\frac{1}{k}\right]} A_{\left(k\right),p}^{r}\left(t\right)\left(\chi_{x,k}\right)_{\left[\frac{1}{k}\right]} \omega \right\|_{0} \\
&\leq C_{3}\left(\chi^{\left[1\right]},\chi^{\left[2\right]}\right) \left\| \tilde{\chi}_{1} \left(\Delta_{f,p}^{\left(r\right)}\right)^{m}  A_{\left(k\right),p}^{r}\left(t\right)\left(\chi_{x,k}\right)_{\left[\frac{1}{k}\right]} \omega \right\|_{0} + C_{4}\left(\chi^{\left[1\right]},\chi^{\left[2\right]}\right) \left\|\tilde{\chi}_{2} A_{\left(k\right),p}^{r}\left(t\right)\left(\chi_{x,k}\right)_{\left[\frac{1}{k}\right]} \omega \right\|_{0},
\end{aligned}
\end{equation*}
where $C_{3}\left(\chi^{\left[1\right]},\chi^{\left[2\right]}\right), C_{4}\left(\chi^{\left[1\right]},\chi^{\left[2\right]}\right)$ depends on $\chi^{\left[1\right]}$ and $\chi^{\left[2\right]}$, but is independent of $x,k$ since $\left(\chi_{x,k}\right)_{\left[\frac{1}{k}\right]}\left(\cdot\right)=\chi^{\left[1\right]}\left(\left|x\right|^{-1}\cdot\right)-\chi^{\left[2\right]}\left(\left|x\right|^{-1}\cdot\right)$ with respect to $\left(U_{p}^{k},\varphi_{p}\right)$ and $\left|x\right|>D>1$, and $\tilde{\chi}_{1},\tilde{\chi}_{2}\in \mathcal{C}^{\infty}_{c}\left(\sqrt{k} A_{x}^{k}\right)$ are cut-off functions such that $\tilde{\chi}_{1}=1=\tilde{\chi}_{2}$ in $\sqrt{k}\supp \chi_{x,k}$.

Next, for each $\eta\in \Omega^{r}_{c}\left(B_{k^{\varepsilon}}\left(0\right)\right)$ with $\eta\neq 0$, note that
\begin{equation*}
\begin{aligned}
\frac{\left(\tilde{\chi}_{1} \left(\Delta_{f,p}^{\left(r\right)}\right)^{m}  A_{\left(k\right),p}^{r}\left(t\right)\left(\chi_{x,k}\right)_{\left[\frac{1}{k}\right]} \omega\bigg| \eta\right)}{\left\|\eta\right\|_{0}}
\leq \frac{\left\| \omega\right\|_{-2m}\left\| \left(\chi_{x,k}\right)_{\left[\frac{1}{k}\right]}A_{\left(k\right),p}^{r}\left(t\right)\left(\Delta_{f,p}^{\left(r\right)}\right)^{m} \tilde{\chi}_{1}\eta\right\|_{2m}}{\left\|\eta\right\|_{0}}.
\end{aligned}
\end{equation*}
By Gårding's inequality, Corollary \ref{Cor: L^2 estimate bdd by |x|^-N} and Theorem \ref{Spectral Theorem}, 
\begin{equation*}
\begin{aligned}
&\left\| \left(\chi_{x,k}\right)_{\left[\frac{1}{k}\right]} A_{\left(k\right)}^{r}\left(t\right)\left( \Delta_{f,p}^{\left(r\right)} \right)^{m} \tilde{\chi}_{1}\eta \right\|_{2m}\\
&\leq C_{5}\left\| \left( \Delta_{f,p}^{\left(r\right)}\right)^{m} \left(\chi_{x,k}\right)_{\left[\frac{1}{k}\right]} A_{\left(k\right)}^{r}\left(t\right)\left( \Delta_{f,p}^{\left(r\right)} \right)^{m} \tilde{\chi}_{1}\eta \right\|_{0} 
+ C_{6} \left\| \left(\chi_{x,k}\right)_{\left[\frac{1}{k}\right]} A_{\left(k\right)}^{r}\left(t\right)\left( \Delta_{f,p}^{\left(r\right)} \right)^{m} \tilde{\chi}_{1}\eta \right\|_{0}\\
&= C_{5}k^{\frac{n}{2}-2m} \left\| \left(\Delta_{k}^{\left(r\right)}\right)^{m} \chi_{x,k} e^{-\frac{t}{k}\Delta_{k}^{\left(r\right)}}\left( \Delta_{k}^{\left(r\right)}\right)^{m} \left(\tilde{\chi}_{1}\eta\right)_{\left[k\right]} \right\|_{L^{2}\left(M\right)}\\
&\ +C_{6}k^{\frac{n}{2}-m} \left\| \chi_{x,k} e^{-\frac{t}{k}\Delta_{k}^{\left(r\right)}}\left( \Delta_{k}^{\left(r\right)}\right)^{m} \left(\tilde{\chi}_{1}\eta\right)_{\left[k\right]} \right\|_{L^{2}\left(M\right)}\\
&\leq C_{7}\left(\chi^{\left[1\right]},\chi^{\left[2\right]},\frac{t}{2},m,N\right)\left|x\right|^{-N} k^{\frac{n}{2}-m} \left\| e^{\frac{t}{2k}\Delta_{k}^{\left(r\right)}}\left(\Delta_{k}^{\left(r\right)}\right)^{m}\left(\tilde{\chi}_{1}\eta\right)_{\left[k\right]} \right\|_{L^{2}\left(M\right)}\\
&\leq  C_{8}\left(\chi^{\left[1\right]},\chi^{\left[2\right]},\frac{t}{2},m,N\right)\left|x\right|^{-N} \left\| \eta\right\|_{0},
\end{aligned}
\end{equation*}
where $C_{8}\left(\chi^{\left[1\right]},\chi^{\left[2\right]},\frac{t}{2}\right)$ depends on $\chi^{\left[1\right]},\chi^{\left[2\right]},m,N$ and smoothly on $t$ but  is independent of $D,x,k$. Thus, we deduce
\begin{equation*}
\begin{aligned}
\left\| \tilde{\chi}_{1} \left(\Delta_{f,p}^{\left(r\right)}\right)^{m}  A_{\left(k\right),p}^{r}\left(t\right)\left(\chi_{x,k}\right)_{\left[\frac{1}{k}\right]} \omega \right\|_{0} \leq C_{8}\left(\chi^{\left[1\right]},\chi^{\left[2\right]},\frac{t}{2},m,N\right)  \left|x\right|^{-N} \left\|\omega\right\|_{-2m}.
\end{aligned}
\end{equation*}
Similarly, we can obtain
\begin{equation*}
\begin{aligned}
 \left\|\tilde{\chi}_{2} A_{\left(k\right),p}^{r}\left(t\right)\left(\chi_{x,k}\right)_{\left[\frac{1}{k}\right]} \omega \right\|_{0}
 \leq  C_{9}\left(\chi^{\left[1\right]},\chi^{\left[2\right]},\frac{t}{2},m,N\right)  \left|x\right|^{-N} \left\|\omega\right\|_{-2m},
\end{aligned}
\end{equation*}
where $C_{9}\left(\chi^{\left[1\right]},\chi^{\left[2\right]},\frac{t}{2},m,N\right)$ depends on $\chi^{\left[1\right]},\chi^{\left[2\right]},m,N$ and smoothly on $t$ but is independent of $D,x,k$.

Finally, we conclude
\begin{equation*}
\begin{aligned}
\left\| \left(\chi_{x,k}\right)_{\left[\frac{1}{k}\right]} A_{\left(k\right),p}^{r}\left(t\right)\left(\chi_{x,k}\right)_{\left[\frac{1}{k}\right]} \omega \right\|_{2m}
\leq C_{10} \left(\chi^{\left[1\right]},\chi^{\left[2\right]},t,m,N\right)  \left|x\right|^{-N} \left\|\omega\right\|_{-2m},
\end{aligned}
\end{equation*}
where $C_{10}\left(\chi^{\left[1\right]},\chi^{\left[2\right]},t,m,N\right)$ depends on $\chi^{\left[1\right]},\chi^{\left[2\right]},m,N$ and smoothly on $t$ but is independent of $D,x,k$.

\end{proof}

As a consequence of Theorem \ref{Thm: Mapping property outside criti pts 2}, we can prove Theorem \ref{Main Thm: Asymptotic Behavior of Scaled Heat Kernel in between}:
\begin{proof}[Proof of Theorem \ref{Main Thm: Asymptotic Behavior of Scaled Heat Kernel in between}]
The main key of proving this theorem is similar to Theorem \ref{Main Thm 1: locally uniform boundedness of s.h.}.

Write
\begin{equation*}
\begin{aligned}
A_{\left(k\right),p}^{r}\left(t,x,x\right) 
&=\sideset{}{'}\sum_{I,J} A_{\left(k\right),p}^{r}\phantom{}_{I,J}\left(t,x,x\right) dx^{I} \otimes \left( dx^{J}\right)^{*}\\
&\in \mathcal{C}^{\infty}\left(\mathbb{R}^{+}\times B_{k^{\varepsilon}}\left(0\right)\times B_{k^{\varepsilon}}\left(0\right), \bigwedge^{r}T^{*}\mathbb{R}^{n}\boxtimes \left(\bigwedge^{r}T^{*}\mathbb{R}^{n}\right)^{*}\right).
\end{aligned}
\end{equation*}
Recall that for each $S\in \bigwedge^{r}T^{*}_{x}\mathbb{R}^{n}\otimes \left( \bigwedge^{r}T^{*}_{x}\mathbb{R}^{n}\right)^{*}$, the norm of $S$ is defined to be $\left|S\right|_{x} := \sup_{\omega_{x}\in \bigwedge^{r}T^{*}_{x}\mathbb{R}^{n},\omega_{x}\neq 0}\frac{\left|S\omega_{x}\right|}{\left|\omega_{x}\right|}$, and note that
\begin{equation*}
\begin{aligned}
\left| A_{\left(k\right),p}^{r}\left(t,x,x\right) \right|_{x} \leq \left( \sum_{I,J} \left|  A_{\left(k\right),p}^{r}\phantom{}_{I,J}\left(t,x,x\right) \right|^{2}\right)^{\frac{1}{2}},
\end{aligned}
\end{equation*}
so it suffices to show for any two $I,J$ and for each $N\in \mathbb{N}$,
\begin{equation*}
\begin{aligned}
\left|  A_{\left(k\right),p}^{r}\phantom{}_{I,J}\left(t,x,x\right)  \right|\leq C\left(t,N\right) \left|x\right|^{-N},
\end{aligned}
\end{equation*}
where $C\left(t,N\right)$ depends on $N$ and smoothly on $t$ and is independent of $D,x,k$.  

Let $D>1$ be large enough such that Lemma \ref{Lem: Bochner type estimate (Pointwise)} holds.  Let $x\in B_{k^{\varepsilon}}\left(0\right)\setminus B_{2D}\left(0\right)$, and put $\omega_{x,I,\delta} = \chi_{x,I,\delta} dx^{I}$ and $\omega_{x,J,\delta'} = \chi_{x,J,\delta'} dx^{J}$ as in \eqref{Eq: Approximated identity for form}. By integration by part, we see that
\begin{equation*}
\begin{aligned}
A_{\left(k\right),p}^{r}\phantom{}_{I,J}\left(t,x,x\right) =\lim_{\delta\to 0}\lim_{\delta'\to 0} \left( A_{\left(k\right),p}^{r}\left(t\right) \omega_{x,J,\delta'}\bigg| \omega_{x,I,\delta}\right).
\end{aligned}
\end{equation*}

Now, let $\chi_{x,k}$ be the cut-off function as in \eqref{Eq: Defi of chi_x,k} given by $\chi^{\left[1\right]},\chi^{\left[2\right]}$. Using Theorem \ref{Thm: Mapping property outside criti pts 2} and Lemma \ref{Lem: subseqential convergence of seq. of norms of a.i.} with fixed large $m$, we obtain , if $\delta,\delta'$ small enough, 
\begin{equation*}
\begin{aligned}
\left| \left( A_{\left(k\right),p}^{r}\left(t\right) \omega_{x,J,\delta'}\bigg| \omega_{x,I,\delta}\right) \right| 
&\leq \left\| \left(\chi_{x,k}\right)_{\left[\frac{1}{k}\right]} A_{\left(k\right),p}^{r}\left(t\right)  \left(\chi_{x,k}\right)_{\left[\frac{1}{k}\right]} \omega_{x,J,\delta'} \right\|_{2m} \left\| \omega_{x,I,\delta}\right\|_{-2m} \\
&\leq C\left(\chi^{\left[1\right]},\chi^{\left[2\right]},t,m,N\right) \left|x\right|^{-N},
\end{aligned}
\end{equation*}
where $C\left(\chi^{\left[1\right]},\chi^{\left[2\right]},t,m,N\right)$ depends on $\chi^{\left[1\right]},\chi^{\left[2\right]},m,N$ and smoothly on $t$ but independent of $D,x,k,\delta,\delta'$, for large $k$.

For each $x\in B_{k^{\varepsilon}}\left(0\right)\setminus B_{2D}\left(0\right)$, observe that the cut-off function $\chi_{x,k}$ is constructed from the same cut-off functions $\chi^{\left[1\right]},\chi^{\left[2\right]}$, leading to $A_{\left(k\right),p}^{r}\phantom{}_{I,J}\left(t,x,x\right)$ enjoying the same upper bound. Hence, we conclude
\begin{equation*}
\begin{aligned}
\left|  A_{\left(k\right),p}^{r}\phantom{}_{I,J}\left(t,x,x\right)  \right|\leq C\left(t,N\right) \left|x\right|^{-N},
\end{aligned}
\end{equation*}
where $C\left(t,N\right)$ depends on $N$ and smoothly on $t$ and independent of $D,x,k$ and $N$ is an arbitrary positive integer. (The contribution of $m$ to the constant is insignificant and thus is not marked in the parenthesis). 
\end{proof}

\subsection{Proof of Theorem 1.3} \label{pf. outside 1}

The main idea of this proof is similar to Theorem \ref{Main Thm: Asymptotic Behavior of Scaled Heat Kernel in between}. Recall that for each $p\in \crit \left(f\right)$, we additionally identify the coordinate neighborhood $U_{p}$ of $p$ with the Euclidean ball $B_{\frac{3}{2}}\left(0\right)$ and $p$ with $0$, and for each $k>0$, put
\begin{equation*}
\begin{aligned}
\mathcal{U}^{k} = \bigcup_{p\in \crit\left(f\right)} U_{p}^{k},
\end{aligned}
\end{equation*}
where $U_{p}^{k}$ is identified with the Euclidean ball $B_{k^{-\frac{1}{2}+\varepsilon}}\left(0\right)$, $\varepsilon\in \left(0,\frac{1}{2}\right)$.

First, we need the following estimate of Bochner type. 

\begin{lem} \label{Lem: Bochner-type Estimate}
If $k$ is sufficiently large,
\begin{equation}\label{Ineq: Bochner-type Formula}
\begin{aligned}
\left( \Delta_{k}^{\left(r\right)}\omega | \omega \right) 
\geq Ck^{1+2\varepsilon} \left\| \omega \right\|^{2}
\end{aligned}
\end{equation}
for each $\omega\in \Omega^{r}\left(M\right)$ with $\supp \omega \subset M\setminus \mathcal{U}^{k}$, where $C$ is independent of $k$.
\end{lem}
\begin{proof}

The key point of this proof  is virtually the same as in Lemma \ref{Lem: Bochner type estimate (Pointwise)}.

Since $M$ is compact and $df=0$ at $p\in \crit\left(f\right)$, we can see
\begin{equation*}
\begin{aligned}
\left|df\right|^{2} \geq k^{-1+2\varepsilon}
\end{aligned}
\end{equation*}
in $M\setminus \mathcal{U}^{k}$ if $k$ is sufficiently large; furthermore, we can let $k$ large enough such that
\begin{equation*}
\begin{aligned}
\left| \frac{m}{k^{2\varepsilon}}\right| <\frac{1}{2},
\end{aligned}
\end{equation*}
where $m\in \mathbb{R}$ is given from the fact that
$\left(A\omega | \omega\right) \geq m \left\| \omega\right\|^{2}$ for each $\omega\in \Omega^{r}\left(M\right)$ in which $A=\mathcal{L}_{\nabla f}+\mathcal{L}_{\nabla f}^{*}$. Hence, for each $\omega\in \Omega^{r}\left(M\right)$ with $\supp \omega \subset M\setminus \mathcal{U}^{k}$, we deduce
\begin{equation*}
\begin{aligned}
\left( \Delta_{k}^{\left(r\right)}\omega | \omega \right) 
\geq k^{-1+2\varepsilon} \left(1 +mk^{-2\varepsilon}\right) \left\| \omega\right\|^{2}
\geq \frac{1}{2}k^{-1+2\varepsilon} \left\|\omega\right\|^{2}.
\end{aligned}
\end{equation*}
\end{proof}

Put $\mathcal{U} = \bigcup_{p\in \crit\left(f\right)} U_{p}$ (with $U_{p}$ identified as $B_{\frac{3}{2}}\left(0\right)$ under the coordinate chart $\varphi_{p}$ for each $p\in \crit\left(f\right)$), and let $\tau\in \mathcal{C}^{\infty}_{c}\left(\mathcal{U}\right)$ be a cut-off function such that $\tau=1$ in $\bigcup_{p\in \crit \left(f\right)}\varphi_{p}^{-1}\left(B_{\frac{1}{2}}\left(0\right)\right)$ and $\tau=0$ in $M\setminus \bigcup_{p
\in \crit \left(f\right)}\varphi_{p}^{-1}\left(B_{1}\left(0\right)\right)$. For (large) $k$, put $\tau_{k}\in \mathcal{C}^{\infty}_{c}\left(\mathcal{U}^{k}\right)$ such that
\begin{equation*}
\begin{aligned}
\tau_{k}\circ \varphi_{p}^{-1} \left(q\right) = \tau\left(k^{\frac{1}{2}-\varepsilon}q\right)
\end{aligned}
\end{equation*}
for each $q\in \varphi_{p}\left(U_{p}\right)\subset \mathbb{R}^{n}$ and for each $p\in \crit\left(f\right)$. Finally, we set
\begin{equation}\label{Defi: chi_k}
\begin{aligned}
\chi_{k} =1 - \tau_{k} \in \mathcal{C}^{\infty}_{c}\left(M\setminus  \bigcup_{p\in \crit\left(f\right)}\varphi_{p}^{-1}\left(B_{\frac{1}{2}k^{-\frac{1}{2}+\varepsilon}}\left(0\right)\right)\right)
\end{aligned}
\end{equation}
Note that $\chi_{k}=1$ in $M\setminus \mathcal{U}^{k}$. Moreover, we can see $D^{\alpha}\chi_{k}\in \mathcal{C}_{c}^{\infty}\left(\mathcal{U}^{k}\right)$ and
\begin{equation*}
\begin{aligned}
\sup_{\mathbb{R}^{n}} \left| D^{\alpha} \chi_{k}\right|
 \leq C\left(\tau\right)k^{-\varepsilon\left|\alpha\right|}k^{\frac{\left|\alpha\right|}{2}}
\end{aligned}
\end{equation*}
with respect to each of the coordinate charts $\left(U_{p},\varphi_{p}\right)$, for each multi-index $\alpha\neq 0$. 

Again, we have the following $L^{2}$-estimate via Lemma \ref{Lem: Bochner-type Estimate}.

\begin{theo}\label{Thm: L^2 esitmate on Delta^m chi e^t/k = O(k^-N) outside crit. pts}
Let $\chi_{k}$ be the cut-off function given by the cut-off function $\tau$ as in \eqref{Defi: chi_k}. If $k$ large, then for each $N\in \mathbb{N}$ and for each fixed $m\in \mathbb{N}$,
\begin{equation}\label{L^2 esitmate on Delta^m chi e^t/k = O(k^-N) outside crit. pts}
\begin{aligned}
\left\| \left( \Delta_{k}^{\left(r\right)}\right)^{m} \left( \chi_{k} e^{-\frac{t}{k}\Delta_{k}^{\left(r\right)}}\omega\right) \right\| \leq C\left(\tau, t,m,N\right) k^{m-N} \left\|\omega\right\|
\end{aligned}
\end{equation}
for each $\omega\in \Omega^{r}\left(M\right)$, where $C\left(\tau, t,m,N\right)$ depends on $\tau,m, N$ and smoothly on $t$ but is independent of $k$
\end{theo}
\begin{proof}
Theorem \ref{Thm: L^2 esitmate on Delta^m chi e^t/k = O(k^-N) outside crit. pts} can be established by the arguments in Theorem \ref{Thm: L^2 estimate bdd by |x|^-N w/o diff.} and Corollary \ref{Cor: L^2 estimate bdd by |x|^-N} by replacing $\left|x\right|$ with $k^{\varepsilon}$, and $\chi^{\left[1\right]},\chi^{\left[2\right]}$ with $\tau$.
\end{proof}

\begin{theo}\label{Thm: Mapping propety outside criti pts 1}
For each $p\in \crit \left(f\right)$, let $\chi_{k}$ be the cut-off function given by the cut-off function $\tau$ as in \eqref{Defi: chi_k}.  If $k$ is sufficiently large, then for each $N\in \mathbb{N}$, for each $m\in \mathbb{N}\cup\left\{0\right\}$,
\begin{equation}\label{Ineq: Mapping propety outside criti pts 1}
\begin{aligned}
\left\| \chi_{k} e^{-\frac{t}{k}\Delta_{k}^{\left(r\right)}} \chi_{k} \omega\right\|_{2m} 
\leq C\left(\tau ,t,m, N\right) k^{-N}\left\| \omega \right\|_{-2m}
\end{aligned}
\end{equation}
for each $\omega\in \Omega^{r}_{c}\left(V\right)$, where $C\left(\tau ,t,m,N\right)>0$ depends on $\tau,m,N$ and smoothly on $t$ but is independent of $k$.
\end{theo}
\begin{proof}
The main idea of proving Theorem \ref{Thm: Mapping propety outside criti pts 1} is very similar to Theorem \ref{Critical Thm: mapping property of s.h.o.}. Choose a pair $\left(\mathcal{V},\mathcal{P},\mathcal{E}\right)$ so that the Sobolev norms are defined on $M$.

By definition,
\begin{equation}\label{Eq: Split for 2m-Sobo. norm of chi_k e^-t/k Delta}
\begin{aligned}
\left\| \chi_{k} e^{-\frac{t}{k}\Delta_{k}^{\left(r\right)}}\chi_{k} \omega\right\|_{2m}^{2} = \sum_{\psi \in \mathcal{P}} \left\| \psi  \chi_{k} e^{-\frac{t}{k}\Delta_{k}^{\left(r\right)}}\chi_{k} \omega\right\|^{2}_{2m}.
\end{aligned}
\end{equation}
By Gårding's inequality, we see that
\begin{equation*}
\begin{aligned}
&\left\| \psi \chi_{k} e^{-\frac{t}{k}\Delta_{k}^{\left(r\right)}} \chi_{k}\omega \right\|_{2m}\\
&\leq C_{1}\left(\psi\right) k^{m}  \left\| \tilde{\psi}_{1} \left(\Delta_{k}^{\left(r\right)}\right)^{m}\left( \chi_{k} e^{-\frac{t}{k}\Delta_{k}^{\left(r\right)}}\chi_{k} \omega\right) \right\|_{0} + C_{2}\left(\psi\right)k^{m} \left\|\tilde{\psi}_{2} \chi_{k} e^{-\frac{t}{k}\Delta_{k}^{\left(r\right)}} \chi_{k} \omega\right\|_{0},
\end{aligned}
\end{equation*}
where $\tilde{\psi}_{1},\tilde{\psi}_{2}$ are cut-off functions with compact support in the coordinate domain in which $\supp \psi$ lies and $\tilde{\psi}_{1}=1=\tilde{\psi}_{2}$ in $\supp \psi$. Moreover, by Theorem \ref{Thm: Mapping propety outside criti pts 1}, we obtain for each $N_{1}\in \mathbb{N}$,
\begin{equation}\label{Ineq: One piece of 2m-Sobo. norm}
\begin{aligned}
&\left\| \psi \chi_{k} e^{-\frac{t}{k}\Delta_{k}^{\left(r\right)}} \chi_{k}\omega \right\|_{2m}\\
&\leq \left( C_{3}\left(\psi, \tau, \frac{t}{2},m,N\right)k^{2m-N_{1}}+ C_{4}\left(\psi, \tau, \frac{t}{2},m,N\right)k^{m-N_{1}} \right) \left\| e^{-\frac{t}{2k}\Delta_{k}^{\left(r\right)}}\chi_{k}\omega\right\|_{L^{2}\left(M\right)}\\
&\leq C_{5}\left(\psi,\tau, t,m,N\right)k^{2m-N_{1}} \left\| e^{-\frac{t}{2k}\Delta_{k}^{\left(r\right)}}\chi_{k}\omega\right\|_{L^{2}\left(M\right)},
\end{aligned}
\end{equation}
where $C_{5}\left(\psi,\tau, t,m,N\right)$ depends on $\psi,\tau,m,N$ and smoothly on $t$ but is independent of $k$.

Now, since $\sum_{\psi\in \mathcal{P}} \psi^{2} =1$, for each $\eta\in \Omega^{r}\left(M\right)$, $\eta\neq 0$, we deduce
\begin{equation*}
\begin{aligned}
\frac{\left|\left( e^{-\frac{t}{k}\Delta_{k}^{\left(r\right)}}\chi_{k}\omega \bigg| \eta\right) \right|}{\left\|\eta\right\|_{L^{2}\left(M\right)}}
&=\frac{\left|\left( \omega \bigg| \chi_{k} e^{-\frac{t}{k}\Delta_{k}^{\left(r\right)}}\eta \right) \right|}{\left\|\eta\right\|_{L^{2}\left(M\right)}}
\leq \sum_{\psi\in \mathcal{P}} \frac{\left|\left( \psi\omega \bigg| \psi \chi_{k} e^{-\frac{t}{k}\Delta_{k}^{\left(r\right)}}\eta \right) \right|}{\left\|\eta\right\|_{L^{2}\left(M\right)}}\\
&\leq  \sum_{\psi\in \mathcal{P}} \frac{ \left\|\psi \omega\right\|_{-2m} \left\| \psi \chi_{k} e^{-\frac{t}{k}\Delta_{k}^{\left(r\right)}} \eta\right\|_{2m} }{\left\|\eta\right\|_{L^{2}\left(M\right)}}.
\end{aligned}
\end{equation*} 
Using Gårding's inequality again, we obtain
\begin{equation*}
\begin{aligned}
&\left\| \psi \chi_{k} e^{-\frac{t}{k}\Delta_{k}^{\left(r\right)}}\eta\right\|_{2m}\\
&\leq C_{6}\left(\psi\right) k^{m} \left\| \tilde{\psi}_{3}\left(\Delta_{k}^{\left(r\right)}\right)^{m} \chi_{k}  e^{-\frac{t}{k}\Delta_{k}^{\left(r\right)}} \eta\right\|_{0} 
+ C_{7}\left(\psi\right) k^{m} \left\| \tilde{\psi}_{4} \chi_{k} e^{-\frac{t}{k}\Delta_{k}^{\left(r\right)}} \eta\right\|_{0},
\end{aligned}
\end{equation*}
where $\tilde{\psi}_{3},\tilde{\psi}_{4}$ are cut-off functions with compact support in the coordinate domain in which $\supp \psi$ lies and $\tilde{\psi}_{3}=1=\tilde{\psi}_{4}$ in $\supp \psi$. By Theorem \ref{Thm: Mapping propety outside criti pts 1}, we obtain for each $N_{2}\in \mathbb{N}$, 
\begin{equation*}
\begin{aligned}
\left\| \psi \chi_{k} e^{-\frac{t}{k}\Delta_{k}^{\left(r\right)}}\eta\right\|_{2m}
\leq C_{8}\left(\psi,\tau,t,m,N\right)k^{2m-N_{2}} \left\|\eta\right\|_{L^{2}\left(M\right)},
\end{aligned}
\end{equation*}
$C_{8}\left(\psi,\tau, t,m,N\right)$ depends on $\psi,\tau,m,N$ and smoothly on $t$ but is independent of $k$. Hence, we conclude
\begin{equation}\label{Ineq: L^2 norm estimate for e^-t/k Delta chi_k}
\begin{aligned}
\left\| e^{-\frac{t}{k}\Delta_{k}^{\left(r\right)}} \chi_{k}\omega\right\|_{L^{2}\left(M\right)} 
&\leq C_{9}\left(\tau,t,m,N\right)k^{2m-N_{2}} \sum_{\psi \in \mathcal{P}} \left\|\psi\omega\right\|_{-2m}\\
&\leq C_{10}\left(\tau,t,m,N\right)k^{2m-N_{2}}  \left( \sum_{\psi\in \mathcal{P}} \left\|\psi \omega\right\|_{-2m}^{2}\right)^{\frac{1}{2}}\\
&= C_{10}\left(\tau,t,m,N\right)k^{2m-N_{2}}   \left\| \omega\right\|_{-2m}.
\end{aligned}
\end{equation}
where $C_{10}\left(\tau, t,m,N\right)$ depends on $\tau,m,N$ and smoothly on $t$ but is independent of $\psi, k$. 

Hence, by \eqref{Eq: Split for 2m-Sobo. norm of chi_k e^-t/k Delta}, \eqref{Ineq: One piece of 2m-Sobo. norm}, and \eqref{Ineq: L^2 norm estimate for e^-t/k Delta chi_k}, we deduce for each $N_{1},N_{2}\in \mathbb{N}$,
\begin{equation*}
\begin{aligned}
\left\| \chi_{k} e^{-\frac{t}{k}\Delta_{k}^{\left(r\right)}}\chi_{k} \omega\right\|_{2m}
= C_{11}\left(\tau,t,m,N\right) k^{4m-N_{1}-N_{2}} \left\| \omega\right\|_{-2m},
\end{aligned}
\end{equation*}
where $C_{11}\left(\tau,t,m,N\right)$ depends on $\tau,m,N$ and smoothly on $t$ but is independent of $\psi, k$. Finally, take $N_{1},N_{2}$ large enough, we can conclude for each $N\in \mathbb{N}$,
\begin{equation*}
\begin{aligned}
\left\| \chi_{k} e^{-\frac{t}{k}\Delta_{k}^{\left(r\right)}}\chi_{k} \omega\right\|_{2m}
= C_{12}\left(\tau,t,m,N\right) k^{-N} \left\| \omega\right\|_{-2m}
\end{aligned}
\end{equation*}
as desired.
\end{proof}

Now, we can start to prove Theorem \ref{Main Thm: Asymptotic Behavior of Heat Kernel outside Critical Points}.

\begin{proof}[Proof of Theorem \ref{Main Thm: Asymptotic Behavior of Heat Kernel outside Critical Points}]

The argument of Theorem \ref{Main Thm: Asymptotic Behavior of Heat Kernel outside Critical Points} is very similar to Theorem \ref{Main Thm: Asymptotic Behavior of Scaled Heat Kernel in between}.

Choose a pair $\left(\mathcal{V},\mathcal{P}, \mathcal{E}\right)$ such that the sup-norm $\left\|\cdot\right\|_{\mathcal{C}^{0}}$ and the Sobolev norms are defined. 

For each large $k>0$, let $x\in M\setminus \mathcal{U}^{k}$. Put $\omega_{x,I,\delta} = \chi_{x,I,\delta} E^{I}$ and $\omega_{x,J,\delta'}=\chi_{x,J,\delta'} E^{J}$ with $\chi_{x,I,\delta},\chi_{y,J,\delta'}$ as in \eqref{Defi: approximate identity} and $\left\{E^{I}\right\}_{I}\subset \mathcal{E}$ is a local orthonormal frame.  By integration by part, we see that
\begin{equation*}
\begin{aligned}
\left| e^{-\frac{t}{k}\Delta_{k}^{\left(r\right)}}\phantom{}_{I,J}\left(t,x,x\right)\right|
=\lim_{\delta\to 0}\lim_{\delta'\to 0} \left| \left( e^{-\frac{t}{k}\Delta_{k}^{\left(r\right)}}\omega_{x,J,\delta'} \bigg| \omega_{x,I,\delta} \right) \right|.
\end{aligned}
\end{equation*}
Let $\chi_{k}$ be the cut-off function as in \eqref{Defi: chi_k}. Fix a large $m$, and we obtain
\begin{equation*}
\begin{aligned}
 \left| \left( e^{-\frac{t}{k}\Delta_{k}^{\left(r\right)}}\omega_{x,J,\delta'} \bigg| \omega_{x,I,\delta} \right) \right|
& \leq \sum_{\psi\in \mathcal{P}} \left| \left( \psi \chi_{k} e^{-\frac{t}{k}\Delta_{k}^{\left(r\right)}}\chi_{k} \omega_{x,J,\delta'} \bigg| \psi \omega_{x,I,\delta} \right) \right|\\
&\leq \sum_{\psi \in \mathcal{P}}  \left\| \psi \chi_{k} e^{-\frac{t}{k}\Delta_{k}^{\left(r\right)}}\chi_{k} \omega_{x,J,\delta'} \right\|_{2m} \left\|\psi \omega_{x,I,\delta} \right\|_{-2m}.
\end{aligned}
\end{equation*}
Then in view of Lemma \ref{Lem: subseqential convergence of seq. of norms of a.i.}, we can deduce $\left\| \psi \omega_{x,I,\delta}\right\|_{-2m}\leq C_{1}\left(\psi\right)$; moreover, by definition, $ \left\| \psi \chi_{k} e^{-\frac{t}{k}\Delta_{k}^{\left(r\right)}}\chi_{k} \omega_{x,J,\delta'} \right\|_{2m}\leq \left\| \chi_{k}e^{-\frac{t}{k}\Delta_{k}^{\left(r\right)}}\chi_{k}\omega_{x,J,\delta'}\right\|_{2m}$. Hence, by Theorem \ref{Thm: Mapping propety outside criti pts 1}, we conclude for each $N\in \mathbb{N}$,
\begin{equation*}
\begin{aligned}
 \left| \left( e^{-\frac{t}{k}\Delta_{k}^{\left(r\right)}}\omega_{x,J,\delta'} \bigg| \omega_{x,I,\delta} \right) \right|
 \leq C_{2} \left\| \chi_{k}e^{-\frac{t}{k}\Delta_{k}^{\left(r\right)}}\chi_{k}\omega_{x,J,\delta'}\right\|_{2m}
\leq C_{3}\left(\tau,t,N\right) k^{-N},
\end{aligned}
\end{equation*}
where $C_{3}\left(\tau,t,N\right)$ depends on $\tau,N$ and smoothly on $t$ but is independent of $x, k,\delta,\delta'$ (Dependence of $m$ is again redundant). In addition, note that the cut-off function $\chi_{k}$ is constructed so that $\chi_{k}=1$ in $M\setminus \mathcal{U}^{k}$, which indicates $e^{-\frac{t}{k}\Delta_{k}^{\left(r\right)}}\phantom{}_{I,J}\left(t,x,x\right)$ shares the same upper bound for each $x\in M\setminus \mathcal{U}^{k}$. Therefore, we obtain for each $t\in T$, for each $x\in M\setminus \mathcal{U}^{k}$, and for each $N\in \mathbb{N}$, 
\begin{equation*}
\begin{aligned}
\left| e^{-\frac{t}{k}\Delta_{k}^{\left(r\right)}}\phantom{}_{I,J}\left(t,x,x\right)\right|\leq C_{4}\left(t,N\right)k^{-N},
\end{aligned}
\end{equation*}
where $C_{4}\left(t,N\right)$ depends on $N$ and smoothly on $t$ and is independent of $x,k$.

Finally, for each $t>0$ and for each $N\in \mathbb{N}$, we deduce
\begin{equation*}
\begin{aligned}
\left\| e^{-\frac{t}{k}\Delta_{k}^{\left(r\right)}}\left(t,x,x\right)\right\|_{\mathcal{C}^{0}\left(M\setminus \mathcal{U}^{k}\right)}^{2}
&=\sum_{\psi\in \mathcal{P}} \left\| \psi e^{-\frac{t}{k}\Delta_{k}^{\left(r\right)}}\left(t,x,x\right) \right\|_{\mathcal{C}^{0}\left(M\setminus \mathcal{U}^{k}\right)}^{2}\\
&=\sum_{\psi\in \mathcal{P}}  \left(\sup_{\supp \psi \cap M\setminus \mathcal{U}^{k}}  \left( \sum_{I,J} \left| \psi e^{-\frac{t}{k}\Delta_{k}^{\left(r\right)}}\phantom{}_{I,J}\left(t,x,x\right)  \right|^{2} \right)^{\frac{1}{2}} \right)^{2} \\
&\leq C_{5}\left(t,N\right)k^{-N},
\end{aligned}
\end{equation*}
where $C_{5}\left(t,N\right)$ depends on $N$ and smoothly on $t$. This furnishes Theorem \ref{Main Thm: Asymptotic Behavior of Heat Kernel outside Critical Points}.
\end{proof}

\section{Morse Inequalities}\label{M.I.}

In this section, we give a new analytic proof of the Morse inequalities as an application of our heat kernel results: Theorems \ref{Main Thm: Semi-classical Heat Kernel Asymptotics}, \ref{Main Thm: Asymptotic Behavior of Scaled Heat Kernel in between}, and \ref{Main Thm: Asymptotic Behavior of Heat Kernel outside Critical Points}.

First, we review the Morse inequalities:
\begin{theo}[The Morse Inequalities] \label{Main Theorem: Morse Inequalities} Let $M$ be a compact orientable smooth manifold of dimension $n$ and let $f$ be a Morse function. Then
\begin{enumerate}
\item[(a)] for each $0\leq r\leq n$,
\begin{equation}\label{Weak Morse Ineq}
\begin{aligned}
\dim H_{dR}^{r}\left(M\right) \leq m_{r};
\end{aligned}
\end{equation}
\item[(b)] for each $0\leq r\leq n$,
\begin{equation}\label{Strong Morse Ineq 1}
\begin{aligned}
\sum_{j=0}^{r} \left(-1\right)^{r-j} \dim H_{dR}^{j}\left(M\right)
\leq \sum_{j=0}^{r} \left(-1\right)^{r-j} m_{j}
\end{aligned}
\end{equation}
and the equality holds if $r=n$; namely,
\begin{equation}\label{Strong Morse Ineq 2}
\begin{aligned}
\sum_{j=0}^{n} \left(-1\right)^{n-j} \dim H_{dR}^{j}\left(M\right)
= \sum_{j=0}^{n} \left(-1\right)^{n-j} m_{j}.
\end{aligned}
\end{equation}
\end{enumerate}
\end{theo}

To prove Theorem \ref{Main Theorem: Morse Inequalities}, let us recall the local index theory.  Let $X$ be an inner product space and let $\left\{E_{I}\right\}_{I}$ be an orthonormal basis for $X$. Recall that a trace of a linear transformation $A:X\to X$ is given by
\begin{equation*}
\begin{aligned}
\tr A = \sum_{I} \langle AE_{I} | E_{I} \rangle. 
\end{aligned}
\end{equation*}
Then we can derive the following McKean-Singer type trace integral formula:
\begin{lem}[McKean-Singer Type Trace Integral Formula, cf. \cite{MS67}] \label{Thm: local index theory for witten laplacian} Let $M$ be a compact orientable Riemannian manifold of dimension $n$. Then for each $t>0$ and for each $k>0$, we have
\begin{enumerate}
\item[(a)] for each $r$,
\begin{equation*}\label{Weak local index theory (Witten)}
\begin{aligned}
\dim H_{dR}^{r}\left(M\right)\leq \int_{M} \tr e^{-\frac{t}{k}\Delta_{k}^{\left(r\right)}}\left(x,x\right) \ d V;
\end{aligned}
\end{equation*}
\item[(b)] for each $r$,
\begin{equation*}\label{Strong local index theory 1 (Witten)}
\begin{aligned}
\sum_{j=0}^{r} \left(-1\right)^{r-j}\dim H_{dR}^{j}\left(M\right)\leq \sum_{j=0}^{r}\left(-1\right)^{r-j} \int_{M} \tr  e^{-\frac{t}{k}\Delta_{k}^{\left(j\right)}}\left(x,x\right)  \ d V,
\end{aligned}
\end{equation*}
and the equality holds if $r=n$; namely,
\begin{equation*}\label{Strong local index theory 2 (Witten)}
\begin{aligned}
\sum_{j=0}^{n} \left(-1\right)^{n-j}\dim H_{dR}^{j}\left(M\right)= \sum_{j=0}^{n}\left(-1\right)^{n-j} \int_{M} \tr  e^{-\frac{t}{k}\Delta_{k}^{\left(j\right)}}\left(x,x\right) \ d V.
\end{aligned}
\end{equation*}
\end{enumerate}
\end{lem}

\begin{proof}

We begin by noting that, from \eqref{Eq: Heat Kernel in terms of eigenforms},
\begin{equation*}
\begin{aligned}
\tr  e^{-\frac{t}{k}\Delta_{k}^{\left(r\right)}}\left(x,x\right)
=\sum_{\lambda\in \spe \Delta_{k}^{\left(r\right)}} \sum_{i=1}^{d_{\lambda}}
e^{-\frac{t}{k}\lambda} \tr \varphi_{i}^{\lambda}\left(x\right)\otimes \left(\varphi_{i}^{\lambda}\right)^{*}\left(x\right)
=\sum_{\lambda\in \spe \Delta_{k}^{\left(r\right)}} \sum_{i=1}^{d_{\lambda}} e^{-\frac{t}{k}\lambda} \left| \varphi_{i}^{\lambda}\left(x\right)\right|^{2},
\end{aligned}
\end{equation*}
where $d_{\lambda}=\dim E_{\lambda,k}^{\left(r\right)}\left(M\right)$.

For each $r$, observe that
\begin{equation*}
\begin{aligned}
Z^{r} 
=\int_{M} \tr  e^{-\frac{t}{k}\Delta_{k}^{\left(r\right)}}\left(x,x\right) \ d V 
=\dim H_{k}^{r}\left(M\right)+\sum_{\lambda \in \spe \Delta_{k}^{\left(r\right)}\setminus \left\{0\right\}} e^{-\frac{t}{k}\lambda}\dim E_{\lambda,k}^{\left(r\right)}\left(M\right),
\end{aligned}
\end{equation*}
which follows from the fact that the order of integral and infinite sum can interchange since the series representation \eqref{Eq: Heat Kernel in terms of eigenforms} converges uniformly on compact subsets. Hence, we conclude
\begin{equation*}
\begin{aligned}
\dim H_{k}^{r}\left(M\right)\leq Z^{r}.
\end{aligned}
\end{equation*}
This proves (a).

To see (b), notice that
\begin{equation*}
\begin{aligned}
\sum_{j=0}^{r} \left(-1\right)^{r-j} Z^{j} 
&= \sum_{j=0}^{r} \left(-1\right)^{r-j} \dim H_{k}^{j}\left(M\right)+\sum_{j=0}^{r} \left(-1\right)^{r-j} \sum_{\lambda^{\left(j\right)}\in \spe \Delta_{k}^{\left(j\right)}\setminus \left\{0\right\}}e^{-\frac{t}{k}\lambda^{\left(j\right)}}\dim E_{\lambda^{\left(j\right)},k}^{\left(j\right)}\left(M\right)\\
&=\sum_{j=0}^{r} \left(-1\right)^{r-j} \dim H_{k}^{j}\left(M\right)+ \sum_{\lambda\in \mathbb{R}^{+}} e^{-\frac{t}{k}\lambda} \sum_{j=0}^{r} \left(-1\right)^{r-j} \dim E_{\lambda,k}^{\left(j\right)}\left(M\right),
\end{aligned}
\end{equation*}
where we interpret $E_{\lambda,k}^{\left(j\right)}\left(M\right)=\left\{0\right\}$ if $\lambda$ is not an eigenvalue of $\Delta_{k}^{\left(j\right)}$. Finally, by Proposition \ref{Prop: Alternative sum related to eigenspaces} and Proposition \ref{Prop: Isomorphism between two coho. gps}, we have established (b).

\end{proof}

Thanks to Lemma \ref{Thm: local index theory for witten laplacian}, proving Theorem \ref{Main Theorem: Morse Inequalities} boils down to investigating the trace integral of the heat kernel $e^{-\frac{t}{k}\Delta_{k}^{\left(r\right)}}\left(x,y\right)$, from which our main results (Theorems \ref{Main Thm: Semi-classical Heat Kernel Asymptotics}, \ref{Main Thm: Asymptotic Behavior of Scaled Heat Kernel in between}, \ref{Main Thm: Asymptotic Behavior of Heat Kernel outside Critical Points}) come in handy.

\subsection{Model Kernels}\label{m.k.}

To deal with the trace integral in question, it is important to know of the heat kernel $e^{-t\Delta_{f,p}^{\left(r\right)}}\left(x,y\right)$ with respect to $\Delta_{f,p}^{\left(r\right)}$ that we call the model kernel in this paper. In this subsection, we give the explicit expression for the trace of this model kernel by the Mehler's formula (see Theorem \ref{Thm: Mahler's formula}). With that in mind, we will be able to see that the trace integral of $e^{-t\Delta_{f,p}^{\left(r\right)}}\left(x,y\right)$ can be considered as an indicator of critical points of index $r$ (see Theorem \ref{Thm: Trace integral for perturbed harmonic oscillator}).

First, we review some facts about the usual harmonic operators. Let $L$ be the harmonic oscillator given by
\begin{equation*}
\begin{aligned}
L = -\frac{d^{2}}{dx^{2}} + x^{2}
\end{aligned}
\end{equation*}
on $\dom L:= \left\{ f\in L^{2}\left(\mathbb{R}\right): Lf\in L^{2}\left(\mathbb{R}\right) \right\}$. It is well-known that the eigenfunctions of $L$ are given by  for each $N\in \mathbb{N}\cup \left\{0\right\}$,
\begin{equation*}
\begin{aligned}
\Phi_{N}\left(x\right) = \frac{H_{N}\left(x\right)e^{-\frac{x^{2}}{2}}}{\pi^{\frac{1}{4}}\sqrt{2^{N}N!}},
\end{aligned}
\end{equation*}
where
\begin{equation*}
\begin{aligned}
H_{N}\left(x\right) = \left(-1\right)^{N}e^{x^{2}}\frac{d^{N}}{d x^{N}}e^{-x^{2}},
\end{aligned}
\end{equation*}
and with respect to which the eigenvalue is $2N+1$. 

Also, we have the following well-known Mehler's formula:
\begin{theo}[Mehler's formula]\label{Thm: Mahler's formula}
 For each $\rho\in \left[0,1\right)$ and for $x,y \in \mathbb{R}$, we have
\begin{equation}\label{Eq: Mehler's formula}
\begin{aligned}
\sum_{n\geq 0} \frac{\rho^{n}}{2^{n} n!}H_{n}\left(x\right)H_{n}\left(y\right)e^{-\frac{x^{2}+y^{2}}{2}} 
= \frac{1}{\sqrt{1-\rho^{2}}}\exp\left(\frac{4xy\rho - \left(1+\rho^{2}\right)\left(x^{2}+y^{2}\right)}{2\left(1-\rho^{2}\right)}\right).
\end{aligned}
\end{equation}
\end{theo}

Similarly, put
\begin{equation*}
\begin{aligned}
L^{\pm}= -\frac{d^{2}}{dx^{2}} + x^{2} \pm 1
\end{aligned}
\end{equation*}
on $\dom L^{\pm} := \left\{ f\in L^{2}\left(\mathbb{R}\right): L^{\pm}f\in L^{2}\left(\mathbb{R}\right)\right\}$. Note that the eigenfuctions of both operators $L^{\pm}$ are again given by $\left\{\Phi_{N}\right\}_{N\in\mathbb{N}\cup\left\{0\right\}}$, but the eigenvalue corresponding to $\Phi_{N}$ is $2N+2$ for $L^{+}$ while is $2N$ for $L^{-}$. Thus, by the Mehler's  formula \eqref{Eq: Mehler's formula} with $\rho=e^{-2t}$, we obtain
\begin{equation}\label{Eq: Heat Kernel for perturbed Harmonic Oscillator (Function)}
\begin{aligned}
e^{-tL^{\pm}}\left(x,y\right)  
&= e^{\left(-1\mp 1\right)t}\sum_{N\in \mathbb{N}\cup\left\{0\right\}} e^{-2Nt}\Phi_{N}\left(x\right)\Phi_{N}\left(y\right) \\
&= e^{\left(-1\mp 1\right)t}\frac{1}{\pi^{\frac{1}{2}}} \sum_{N\in \mathbb{N}\cup\left\{0\right\}} e^{-2Nt} \frac{1}{2^{N}N!} H_{N}\left(x\right)H_{N}\left(y\right)e^{-\frac{x^{2}+y^{2}}{2}}\\
&= e^{\left(-1\mp 1\right)t} \frac{1}{\pi^{\frac{1}{2}}} \frac{1}{\sqrt{1-e^{-2t}}}\exp\left(\frac{4xy e^{-2t} - \left(1+e^{-4t}\right)\left(x^{2}+y^{2}\right)}{2\left(1-e^{-4t}\right)}\right).
\end{aligned}
\end{equation}

To write down the heat kernel explicitly, recall that for each $p\in \crit \left(f\right)$ and for each multi-index $I$, 
\begin{equation*}
\begin{aligned}
\Delta_{f,p}^{\left(r\right)}\left(g dx^{I}\right) 
= \left[ -\sum_{i=1}^{n} \frac{\partial^{2}}{\partial\left(x^{i}\right)^{2}} +\left(x^{i}\right)^{2} + \sum_{i=1}^{n}\varepsilon_{i}\varepsilon_{i}^{I}\right]  g dx^{I},
\end{aligned}
\end{equation*}
where $\varepsilon_{i},\varepsilon_{I}$ as indicated in \eqref{Eq: Local expression of Witten Laplacian under flat metric}. For each strictly increasing multi-index $I$ with $\left|I\right|=r$, define $\Delta_{f,p}^{I}:\dom \Delta_{f,p}^{I} \to \mathcal{C}^{\infty}\left(\mathbb{R}^{n}\right)$ by
\begin{equation*}
\begin{aligned}
\Delta_{f,p}^{I} g =  -\sum_{i=1}^{n} \frac{\partial^{2}g}{\partial\left(x^{i}\right)^{2}} +\left(x^{i}\right)^{2}g + \sum_{i=1}^{n}\varepsilon_{i}\varepsilon_{i}^{I}g,
\end{aligned}
\end{equation*}
where $\dom \Delta_{f,p}^{I}= \left\{ g\in L^{2}\left(\mathbb{R}^{n}\right): \Delta_{f,p}^{I}g\in L^{2}\left(\mathbb{R}^{n}\right)\right\}$. As we can see, $\Delta_{f,p}^{I}g$ is given by taking the $I$-coefficient from $\Delta_{f,p}^{\left(r\right)} \left(g dx^{I}\right)$ (Note that $\left\{dx^{I}\right\}_{I}$ is a global orthonormal frame for $\bigwedge^{r}T^{*}\mathbb{R}^{n}$); namely, 
\begin{equation*}
\begin{aligned}
\Delta_{f,p}^{I} g = \left(dx^{I}\right)^{*} \left( \Delta_{f,p}^{\left(r\right)}\left(g dx^{I}\right)\right)
\end{aligned}
\end{equation*}
for each $g\in \dom \Delta_{f,p}^{I}$. Moreover, $\Delta_{f,p}^{I}$ is self-adjoint and non-negative, so we can speak of its heat kernel $e^{-t\Delta_{f,p}^{I}}\left(x,y\right)\in \mathcal{C}^{\infty}\left(\mathbb{R}^{+}\times \mathbb{R}^{n}\times \mathbb{R}^{n}\right)$.

Set
\begin{equation*}
\begin{aligned}
L_{i}^{\pm}  -\frac{\partial^{2}}{\partial \left(x^{i}\right)^{2}}+\left(x^{i}\right)^{2} \pm 1.
\end{aligned}
\end{equation*}
and, we can write
\begin{equation}\label{Eq: Componentwise Flat Witten Laplacian}
\begin{aligned}
\Delta_{f,p}^{I} 
= \sum_{ \left\{i\in I:i \leq \ind_{f} p\right\} } L^{-}_{i}
+ \sum_{ \left\{i\in I:i > \ind_{f} p\right\} } L^{+}_{i} 
+ \sum_{ \left\{i\notin I:i \leq \ind_{f} p\right\} } L^{-}_{i}
+ \sum_{ \left\{i\notin I:i > \ind_{f} p\right\} } L^{+}_{i}.
\end{aligned}
\end{equation}
From the previous presented facts, we can see
\begin{equation*}
\begin{aligned}
\Phi_{N_{1},\cdots,N_{n}}\left(x^{1},\cdots,x^{n}\right)= \Phi_{N_{1}}\left(x^{1}\right)\cdots \Phi_{N_{n}}\left(x^{n}\right)
\end{aligned}
\end{equation*}
constitute the set of orthonormal eigenfunctions for $\Delta_{f,p}^{I}$, with respect to which the eigenvalue is 
\begin{equation*}
\begin{aligned}
&\lambda_{N_{1},\cdots,N_{n}}\\
&=\sum_{ \left\{i\in I:i \leq \ind_{f} p\right\}} 2N_{i} +\sum_{ \left\{i\in I:i > \ind_{f} p\right\} } \left( 2N_{i}+2\right) 
+ \sum_{ \left\{i\notin I:i \leq \ind_{f} p\right\} } 2N_{i} + \sum_{ \left\{i\notin I:i > \ind_{f} p\right\} }  \left( 2N_{i}+2\right).
\end{aligned}
\end{equation*}
Hence, put $x=\left(x^{1},\cdots,x^{n}\right), y=\left(y^{1},\cdots,y^{n}\right)$  and the heat kernel $e^{-t\Delta_{f,p}^{I}}$ can be given by
\begin{equation}\label{Eq: I,I component of model kernel}
\begin{aligned}
e^{-t\Delta_{f,p}^{I}}\left(x,y\right) 
&= \sum_{N_{1},\cdots,N_{n}\in \mathbb{N}} e^{-t\lambda_{N_{1},\cdots,N_{n}}} \Phi_{N_{1},\cdots N_{n}}\left(x\right) \Phi_{N_{1},\cdots,N_{n}}\left(y\right)\\
&= \prod_{ \left\{i\in I:i \leq \ind_{f} p\right\} } e^{-tL_{i}^{-}}\left(x^{i},y^{i}\right)
 \prod_{ \left\{i\in I:i > \ind_{f} p\right\} }  e^{-tL_{i}^{+}}\left(x^{i},y^{i}\right) \\
&\times  \prod_{ \left\{i\notin I:i \leq \ind_{f} p\right\} } e^{-tL_{i}^{+}}\left(x^{i},y^{i}\right)
 \prod_{ \left\{i\notin I:i > \ind_{f} p\right\} } e^{-tL_{i}^{-}}\left(x^{i},y^{i}\right),
\end{aligned}
\end{equation}
and each of the heat kernels $e^{-tL_{i}^{\pm}}\left(x^{i},y^{i}\right)$ can be written explicitly via \eqref{Eq: Heat Kernel for perturbed Harmonic Oscillator (Function)}.

\begin{prop}\label{Prop: Componentwise term}
For any two $x,y \in \mathbb{R}^{n}$, write
\begin{equation*}
\begin{aligned}
e^{-t\Delta_{f,p}^{\left(r\right)}}\left(x,y\right)
= \sideset{}{'}\sum_{I,J} e^{-t\Delta_{f,p}^{\left(r\right)}}\phantom{}_{I,J}\left(x,y\right) \left(dx^{I}\right)\left(x\right) \otimes \left(dx^{J}\right)^{*}\left(y\right).
\end{aligned}
\end{equation*}
Then
\begin{equation}\label{Eq: Componentwise term I=J}
\begin{aligned}
 e^{-t\Delta_{f,p}^{\left(r\right)}}\phantom{}_{I,I}\left(x,y\right) =e^{-t\Delta_{f,p}^{I}}\left(x,y\right).
\end{aligned}
\end{equation}
\end{prop}
\begin{proof}
Put $A_{I,I}\left(t\right):\mathcal{C}^{\infty}_{c}\left(\mathbb{R}^{n}\right)\to\mathcal{C}^{\infty}\left(\mathbb{R}^{n}\right)$
by
\begin{equation*}
\begin{aligned}
\left( A_{I,I}\left(t\right) g\right) \left(x\right)= \int_{\mathbb{R}^{n}}  e^{-t\Delta_{f,p}^{\left(r\right)}}\phantom{}_{I,I}\left(x,y\right) g\left(y\right) \ d y.
\end{aligned}
\end{equation*}

To prove \eqref{Eq: Componentwise term I=J}, it suffices to show $ A_{I,I}\left(t\right)g\in \dom \Delta_{f,p}^{I}$ and 
\begin{equation*}
\begin{aligned}
\left\{\begin{array}{l}
\frac{\partial}{\partial t} A_{I,I}\left(t\right)g +\Delta_{f,p}^{I} A_{I,I}\left(t\right) g=0\\
\lim_{t\to 0^{+}}\left\| A_{I,I}\left(t\right)g -g \right\|_{L^{2}\left(\mathbb{R}^{n}\right)}=0
\end{array}\right. 
\end{aligned}
\end{equation*}
for each $g\in \mathcal{C}^{\infty}_{c}\left(\mathbb{R}^{n}\right)$.

Note that
\begin{equation*}
\begin{aligned}
\left\|\Delta_{f,p}^{I} A_{I,I}\left(t\right) g \right\|_{L^{2}\left(\mathbb{R}^{n}\right)} 
&= \left\| \left(dx^{I}\right)^{*} \left( \Delta_{f,p}^{\left(r\right)} \left( A_{I,I}\left(t\right) g\right) dx^{I} \right) \right\|_{L^{2}\left(\mathbb{R}^{n}\right)} \\
&\leq \left\| \Delta_{f,p}^{\left(r\right)}e^{-t\Delta_{f,p}^{\left(r\right)}}\left(g dx^{I}\right)  \right\|_{L^{2}\left(\mathbb{R}^{n}\right)} < \infty
\end{aligned}
\end{equation*}
and 
\begin{equation*}
\begin{aligned}
\Delta_{f,p}^{I} A_{I,I}\left(t\right) g
&= \left(dx^{I}\right)^{*} \left[ \Delta_{f,p}^{\left(r\right)}  \left( \int_{\mathbb{R}^{n}} e^{-t\Delta_{f,p}^{\left(r\right)}}\left(x,y\right) \left( g\left(y\right) \ dx^{I}\right) dy\right)\right] \\ 
&= \left(dx^{I}\right)^{*} \left[ -\frac{\partial}{\partial t} \left( \int_{\mathbb{R}^{n}} e^{-t\Delta_{f,p}^{\left(r\right)}}\left(x,y\right) \left( g\left(y\right) \ dx^{I}\right) dy\right)\right] \\
&=-\frac{\partial}{\partial t} A_{I,I}\left(t\right) g
\end{aligned}
\end{equation*}
for each $g\in \mathcal{C}^{\infty}_{c}\left(\mathbb{R}^{n}\right)$. Since $\left\{dx^{I}\right\}_{I}$ is an orthonormal frame, we see that for each $g\in \mathcal{C}^{\infty}_{c}\left(\mathbb{R}^{n}\right)$,
\begin{equation*}
\begin{aligned}
& \left\| e^{-t\Delta_{f,p}^{\left(r\right)}}\left(g dx^{I}\right) - g dx^{I}\right\|_{L^{2}\left(\mathbb{R}^{n}\right)}^{2}\\
&\sideset{}{'}\sum_{K\neq I} \int_{\mathbb{R}^{n}}\left| \int_{\mathbb{R}^{n}} e^{-t\Delta_{f,p}^{r}}\phantom{}_{K,I}\left(x,y\right) g\left(y\right) dy\right|^{2} dx+ \int_{\mathbb{R}^{n}}\left|  \int_{\mathbb{R}^{n}} e^{-t\Delta_{f,p}^{r}}\phantom{}_{I,I}\left(x,y\right) g\left(y\right) dy - g\left(x\right) \right|^{2} dx
\to 0
\end{aligned}
\end{equation*}
as $t\to 0^{+}$, implying
\begin{equation*}
\begin{aligned}
\left\| A_{I,I}\left(t\right) g-g\right\|_{L^{2}\left(\mathbb{R}^{n}\right)}^{2} = \int_{\mathbb{R}^{n}}\left|  \int_{\mathbb{R}^{n}} e^{-t\Delta_{f,p}^{r}}\phantom{}_{I,I}\left(x,y\right) g\left(y\right) dy - g\left(x\right) \right|^{2} dx\to 0
\end{aligned}
\end{equation*}
as $t\to 0^{+}$. Finally, by using the fundamental theorem of Calculus, together with $\Delta_{f,p}^{I} A_{I,I}\left(t\right) g\in L^{2}\left(\mathbb{R}^{n}\right)$ as in the last part of proof of Theorem \ref{Main Thm: Semi-classical Heat Kernel Asymptotics}, we obtain $A_{I,I}\left(t\right)=e^{-t\Delta_{f,p}^{I}}$ in $\mathcal{C}^{\infty}_{c}\left(\mathbb{R}^{n}\right)$. Hence, we have furnished \eqref{Eq: Componentwise term I=J}.
\end{proof}

Proposition \ref{Prop: Componentwise term} shows that, under the global orthonormal frame $\left\{d x^{I}\right\}_{I}$, the diagonal entries of the model kernel $e^{-t\Delta_{f,p}^{\left(r\right)}}\left(x,y\right)$ as an linear transformation from $\bigwedge^{r} T_{x}^{*}\mathbb{R}^{n}$ to $\bigwedge^{r} T_{y}^{*}\mathbb{R}^{n}$ are nothing but $e^{-t\Delta_{f,p}^{I}}\left(x,y\right)$. That being said, we see that the trace of $e^{-t\Delta_{f,p}^{\left(r\right)}}\left(x,y\right)$ is given by
\begin{equation}\label{Eq: Trace of model kernel}
\begin{aligned}
\tr e^{-t\Delta_{f,p}^{\left(r\right)}}\left(x,y\right) 
= \sideset{}{'}\sum_{I} e^{-t\Delta_{f,p}^{\left(r\right)}}\phantom{}_{I,I}\left(x,y\right)
=  \sideset{}{'}\sum_{I} e^{-t\Delta_{f,p}^{I}}\left(x,y\right)
\end{aligned}
\end{equation}
and $e^{-t\Delta_{f,p}^{I}}\left(x,y\right)$ can be written as in \eqref{Eq: I,I component of model kernel}.

\begin{theo}\label{Thm: Trace integral for perturbed harmonic oscillator}
\begin{equation}\label{Eq: Trace integral for perturbed harmonic oscillator}
\begin{aligned}
\lim_{t\to \infty} \int_{\mathbb{R}^{n}} \tr e^{-t\Delta_{f,p}^{\left(r\right)}}\left(x,x\right) \ d V_{x} =\left\{ \begin{array}{ll}
1 &\text{, if $r=\ind_{f} p$}\\
0 &\text{, otherwise}
\end{array} \right. .
\end{aligned}
\end{equation}
\end{theo}
\begin{proof}

By \eqref{Eq: Trace of model kernel}, \eqref{Eq: I,I component of model kernel}, and the Fubini theorem, the trace integral in question is determined by the (trace) integrals of $e^{-tL_{i}^{\pm}}\left(x,y\right)$: 
\begin{equation*}
\begin{aligned}
\lim_{t\to \infty}\int_{\mathbb{R}}  e^{-tL_{i}^{\pm}}\left(x^{i},x^{i}\right) \ dx^{i}.
\end{aligned}
\end{equation*}

By \eqref{Eq: Heat Kernel for perturbed Harmonic Oscillator (Function)}, we can write
\begin{equation*}
\begin{aligned}
e^{-tL_{i}^{\pm}}\left(x^{i},x^{i}\right)  
= e^{\left(-1\mp 1\right)t} \frac{1}{\pi^{\frac{1}{2}}} \frac{1}{\sqrt{1-e^{-2t}}}\exp\left(\frac{4\left(x^{i}\right)^{2}e^{-2t} - 2\left(1+e^{-4t}\right)\left(x^{i}\right)^{2}}{2\left(1-e^{-4t}\right)}\right).
\end{aligned}
\end{equation*}
Since $e^{-t}\to 0$ as $t\to \infty$, $e^{-tL_{i}^{\pm}}\left(x^{i},x^{i}\right)$ are both bounded by an integrable function. Also, we can see
\begin{equation*}
\begin{aligned}
\lim_{t\to \infty} e^{-tL_{i}^{+}}\left(x^{i},x^{i}\right) = 0
\end{aligned}
\end{equation*}
and
\begin{equation*}
\begin{aligned}
\lim_{t\to \infty}  e^{-tL_{i}^{-}}\left(x^{i},x^{i}\right) = \frac{1}{\pi^{\frac{1}{2}}} e^{-\left(x^{i}\right)^{2}}.
\end{aligned}
\end{equation*}
Thus, by the Lebesgue dominated convergence theorem, we obtain
\begin{equation*}
\begin{aligned}
\lim_{t\to \infty}\int_{\mathbb{R}} e^{-tL_{i}^{+}}\left(x^{i},x^{i}\right) \ dx^{i} = 0
\end{aligned}
\end{equation*}
and
\begin{equation*}
\begin{aligned}
\lim_{t\to \infty}\int_{\mathbb{R}} e^{-tL_{i}^{-}}\left(x^{i},x^{i}\right) \ dx^{i} = 1.
\end{aligned}
\end{equation*}

Now, to see \eqref{Eq: Trace integral for perturbed harmonic oscillator}, if $r=\ind_{f} p$, choose $I_{0}=\left(1,\cdots,r\right)$ and we obtain
\begin{equation*}
\begin{aligned}
e^{-t\Delta_{f,p}^{I_{0}}}\left(x,x\right) = \prod_{i=1}^{n} e^{-tL_{i}^{-}}\left(x^{i},x^{i}\right),
\end{aligned}
\end{equation*} 
in which case, we deduce
\begin{equation*}
\begin{aligned}
\lim_{t\to \infty}\int_{\mathbb{R}^{n}} e^{-t\Delta_{f,p}^{I_{0}}}\left(x,x\right)  dx = 1.
\end{aligned}
\end{equation*}
For each of the other strictly increasing indices $I\neq I_{0}$, there must exists $s\in I$ such that $s>r=\ind_f p$, implying, from \eqref{Eq: I,I component of model kernel}, $e^{-t\Delta_{f,p}^{I}}\left(x,x\right)$ contains the kernel $e^{-tL_{s}^{+}}\left(x^{s},x^{s}\right)$. This leads to
\begin{equation*}
\begin{aligned}
\lim_{t\to \infty}\int_{\mathbb{R}^{n}} e^{-t\Delta_{f,p}^{I}}\left(x,x\right) dx =0.
\end{aligned}
\end{equation*}
Hence, we conclude $\lim_{t\to \infty} \int_{\mathbb{R}^{n}} \tr e^{-t\Delta_{f,p}^{\left(r\right)}}\left(x,x\right) dx = 1$ if $r=\ind_{f} p$.

If $r\neq \ind_{f} p$, we discuss in the two cases for strictly increasing indices $I$: $I$ whose elements are lower or equal to $\ind_{f} p$ and the others. For each strictly increasing index $I$ whose elements are lower or equal to $\ind_{f} p$, there must exist $s\notin I$ such that $s\leq \ind_{f} p$.  This implies  $e^{-t\Delta_{f,p}^{I}}\left(x,x\right)$ contains $e^{-tL_{s}^{+}}\left(x^{s},x^{s}\right)$, resulting in $\lim_{t\to \infty}\int_{\mathbb{R}^{n}} e^{-t\Delta_{f,p}^{I}}\left(x,x\right) dx =0$. For each of the other strictly increasing indices $I$, $I$ must contains at least one element $s$ such that $s>\ind_{f} p$. This implies again  $e^{-t\Delta_{f,p}^{I}}\left(x,x\right)$ contains $e^{-tL_{s}^{+}}\left(x^{s},x^{s}\right)$ and so $\lim_{t\to \infty}\int_{\mathbb{R}^{n}} e^{-t\Delta_{f,p}^{I}}\left(x,x\right) dx =0$. Therefore, we conclude $\lim_{t\to \infty} \int_{\mathbb{R}^{n}} \tr e^{-t\Delta_{f,p}^{\left(r\right)}}\left(x,x\right) dx = 0$ if $r\neq \ind_{f} p$. Hence, we have furnished Theorem \ref{Thm: Trace integral for perturbed harmonic oscillator}.
\end{proof}

We will see in a moment that Theorem \ref{Thm: Trace integral for perturbed harmonic oscillator} plays a central role in tackling the trace integral of  $e^{-\frac{t}{k}\Delta_{k}^{\left(r\right)}}\left(x,y\right)$.

\subsection{Heat Kernel Proof} In this subsection, we give our heat kernel proof of the Morse inequalities. The essence of our proof is that our main results (Theorems \ref{Main Thm: Semi-classical Heat Kernel Asymptotics}, \ref{Main Thm: Asymptotic Behavior of Scaled Heat Kernel in between}, and \ref{Main Thm: Asymptotic Behavior of Heat Kernel outside Critical Points}) allow us to see as $k\to \infty$, the trace integral of $e^{-\frac{t}{k}\Delta_{k}^{\left(r\right)}}\left(x,y\right)$ is approximately close to the sum of the trace integral of the model kernels, which further converges to the Morse number $m_{r}$ by Theorem \ref{Thm: Trace integral for perturbed harmonic oscillator} as $t\to \infty$. In other words, we can deduce the following result regarding the trace integral in question based upon our main results:

\begin{theo}\label{Main Thm: Critical point info captured by the asymptotic behaviors}
\begin{equation}\label{Eq: Critical point info captured by the asymptotic behaviors}
\begin{aligned}
\lim_{t\to \infty} \lim_{k\to \infty} \int_{M} \tr e^{-\frac{t}{k}\Delta_{k}^{\left(r\right)}}\left(x,x\right) \ d V = m_{r}
\end{aligned}
\end{equation}
for each $r$. 
\end{theo}
\begin{proof}
Let $D$ be the positive number given in Theorem \ref{Main Thm: Asymptotic Behavior of Scaled Heat Kernel in between}. For each $p\in \crit\left(f\right)$, put $D_{p}^{k}=\varphi_{p}^{-1}\left(B_{2Dk^{-\frac{1}{2}}}\left(0\right)\right)$ and $\mathcal{D}^{k}=\bigcup_{p\in \crit \left(f\right)} D_{p}^{k}$. Moreover, let $\mathcal{U}^{k}$ be as given in Theorem \ref{Main Thm: Asymptotic Behavior of Heat Kernel outside Critical Points}. Note that
\begin{equation}\label{Eq: Split of trace integral}
\begin{aligned}
&\int_{M} \tr e^{-\frac{t}{k}\Delta_{k}^{\left(r\right)}}\left(x,x\right) dV\\
&=\underbrace{\int_{\mathcal{D}^{k}} \tr e^{-\frac{t}{k}\Delta_{k}^{\left(r\right)}}\left(x,x\right) dV}_{(A)}
+ \underbrace{\int_{\mathcal{U}^{k}\setminus \mathcal{D}^{k}} \tr e^{-\frac{t}{k}\Delta_{k}^{\left(r\right)}}\left(x,x\right) dV}_{(B)}
+\underbrace{\int_{M\setminus \mathcal{U}^{k}} \tr e^{-\frac{t}{k}\Delta_{k}^{\left(r\right)}}\left(x,x\right) dV}_{(C)}.
\end{aligned}
\end{equation} 

By a change of variable, we see that
\begin{equation*}
\begin{aligned}
(A)
= \sum_{p\in \crit\left(f\right)}\int_{B_{2D}\left(0\right)} \tr A_{\left(k\right),p}^{r}\left(t,x,x\right) dx.
\end{aligned}
\end{equation*}
By Theorem \ref{Main Thm: Semi-classical Heat Kernel Asymptotics}, we obtain
\begin{equation}\label{Eq: Trace integral A}
\begin{aligned}
\lim_{k\to \infty} (A) = \sum_{p\in \crit\left(f\right)}\int_{B_{2D}\left(0\right)} \tr e^{-t\Delta_{f,p}^{\left(r\right)}}\left(x,x\right) dx.
\end{aligned}
\end{equation}

Similarly, we obtain
\begin{equation*}
\begin{aligned}
(B) = \sum_{p\in \crit\left(f\right)}\int_{\mathbb{R}^{n}} \mathbf{1}_{B_{k^{\varepsilon}}\left(0\right)\setminus B_{2D}\left(0\right)}\tr A_{\left(k\right),p}^{r}\left(t,x,x\right) dx,
\end{aligned}
\end{equation*}
where $\mathbf{1}_{B_{k^{\varepsilon}}\left(0\right)\setminus B_{2D}\left(0\right)}$ is the characteristic function of $B_{k^{\varepsilon}}\left(0\right)\setminus B_{2D}\left(0\right)$. Now, Theorem \ref{Main Thm: Asymptotic Behavior of Scaled Heat Kernel in between} indicates that the integrand is bounded above by an integrable function $\left|x\right|^{-N}$ with large $N$; namely, 
\begin{equation*}
\begin{aligned}
\left| \mathbf{1}_{B_{k^{\varepsilon}}\left(0\right)\setminus B_{2D}\left(0\right)}\tr A_{\left(k\right),p}^{r}\left(t,x,x\right) \right| \leq \frac{C_{1}}{\left|x\right|^{N}},
\end{aligned}
\end{equation*}
where $C_{1}$ is independent of $D, x,k$, for each $x\in \mathbb{R}^{n}$ and for large $k$. Hence, by Theorem \ref{Main Thm: Semi-classical Heat Kernel Asymptotics} and Lebesgue dominated convergence theorem, we retrieve
\begin{equation}\label{Eq: Trace integral B}
\begin{aligned}
\lim_{k\to \infty} (B) = \sum_{p\in \crit\left(f\right)}\int_{\mathbb{R}^{n}\setminus B_{2D}\left(0\right)} e^{-t\Delta_{f,p}^{\left(r\right)}}\left(x,x\right) dx.
\end{aligned}
\end{equation}

To deal with (C), choose a pair $\left(\mathcal{V},\mathcal{P},\mathcal{E}\right)$ so that the $\mathcal{C}^{0}$-norm is defined on the compact manifold $M$. Using the partition of unity, we can rewrite (C) as the finite sum:
\begin{equation*}
\begin{aligned}
(C) 
=\sum_{\psi \in \mathcal{P}} \int_{M} \mathbf{1}_{\supp \psi \cap M\setminus \mathcal{U}^{k}}\cdot \psi \tr e^{-\frac{t}{k}\Delta_{k}^{\left(r\right)}}\left(x,x\right) dV.
\end{aligned}
\end{equation*}
By Theorem \ref{Main Thm: Asymptotic Behavior of Heat Kernel outside Critical Points}, choose a positive integer $N\in \mathbb{N}$ and we see that
\begin{equation*}
\begin{aligned}
\left| \mathbf{1}_{\supp \psi \cap M\setminus \mathcal{U}^{k}}\cdot \psi \tr e^{-\frac{t}{k}\Delta_{k}^{\left(r\right)}}\left(x,x\right) \right| \leq \frac{C_{2}\left(t\right)}{k^{N}} \to 0, \text{as $k\to \infty$}
\end{aligned}
\end{equation*}
where $C_{2}\left(t\right)$ depends on $t$ but is independent of $k$ and $\psi\in \mathcal{P}$, for each $x\in M$. Hence, by Lebesgue dominated convergence theorem, we conclude
\begin{equation}\label{Eq: Trace integral C}
\begin{aligned}
\lim_{k\to \infty} (C) = 0. 
\end{aligned}
\end{equation}

Therefore, by \eqref{Eq: Split of trace integral}, \eqref{Eq: Trace integral A}, \eqref{Eq: Trace integral B}, and \eqref{Eq: Trace integral C}, we obtain
\begin{equation*}
\begin{aligned}
\lim_{k\to \infty} \int_{M} \tr e^{-\frac{t}{k}\Delta_{k}^{\left(r\right)}}\left(x,x\right) dV = \int_{\mathbb{R}^{n}} \tr e^{-t\Delta_{f,p}^{\left(r\right)}}\left(x,x\right) dx.
\end{aligned}
\end{equation*}
Finally, Theorem \ref{Thm: Trace integral for perturbed harmonic oscillator} gives
\begin{equation*}
\begin{aligned}
\lim_{t\to \infty}\lim_{k\to \infty} \int_{M} \tr e^{-\frac{t}{k}\Delta_{k}^{\left(r\right)}}\left(x,x\right) dV = m_{r}.
\end{aligned}
\end{equation*}
Hence, we have established Theorem \ref{Main Thm: Critical point info captured by the asymptotic behaviors}. 

\end{proof}

Morse inequalities now follows from by Lemma \ref{Thm: local index theory for witten laplacian} and Theorem \ref{Main Thm: Critical point info captured by the asymptotic behaviors}. 

\section*{Acknowledgement}
This paper is adapted from the author's master thesis. The author would like to thank deeply to Professor Chin-Yu Hsiao for his kind, dedicated supervision.

\bibliography{sampartb}

\end{document}